\documentclass[reqno]{amsart}
\usepackage{graphicx}
\usepackage{amstext}
\usepackage{amssymb}
\usepackage{amsmath}
\usepackage{color}
\newtheorem{theorem}{Theorem}[section]
\newtheorem{lemma}[theorem]{Lemma}

\newtheorem{proposition}[theorem]{Proposition}
\theoremstyle{definition}

\theoremstyle{remark}
\newtheorem{remark}[theorem]{Remark}
\numberwithin{equation}{section}
\usepackage{hyperref}
\hypersetup{colorlinks,linkcolor=blue,citecolor=blue}
\newcommand{\tin}{t_{\rm in}}
\newcommand{\e}{\varepsilon}
\newcommand{\h}{H^{\frac{1}{2}}}
\newcommand{\D}{|D|^{\frac{1}{2}}}
\newcommand{\dd}{{\rm d}}
\newcommand{\R}{\widetilde{R}}

\newcommand{\x}{\tilde{\tilde{x}}}
\newcommand{\U}{\tilde{\tilde{\mu}}}

\begin{document}

\title{Strongly interacting multi-solitons for generalized Benjamin-Ono equations}\def\rightmark{STRONGLY INTERACTING MULTI-SOLITON}

\author{Yang Lan}
\address{Yau Mathematical Sciences Center, Tsinghua University, 100084 Beijing, P. R. China}
\email{lanyang@mail.tsinghua.edu.cn}

\author{Zhong Wang}
\address{School of Mathematics and Big Data, Foshan University, 528000, P. R. China}
\email{wangzh79@fosu.edu.cn}



\begin{abstract}
We consider the generalized Benjamin-Ono equation: 
$$\partial_tu+\partial_x(-|D|u+|u|^{p-1}u)=0,$$ with $L^2$-supercritical power $p>3$ or $L^2$-subcritical power $2<p<3$. We will construct strongly interacting multi-solitary wave of the form: $\sum_{i=1}^nQ(\cdot-t-x_i(t))$, where $n\geq 2$, and the parameters $x_i(t)$ satisfying $x_{i}(t)-x_{i+1}(t)\sim \sqrt{t}$ as $t\rightarrow +\infty$. We will also prove the uniqueness of such solutions in the case of $n=2$ and $p>3$.
\end{abstract}

\maketitle

\section{Introduction}
\subsection{Setting of the problem}
In this paper, we consider the following generalized Benjamin-Ono equations:
\begin{equation}
\label{CP}
\begin{cases}
\partial_tu+\partial_x(-|D|u+|u|^{p-1}u)=0,\quad (t,x)\in[0,+\infty)\times\mathbb{R},\\
u(0,x)=u_0(x)\in \h,
\end{cases}
\end{equation}
where $p>1$, and $|D|$ is defined as follows:
$$\mathcal{F}(|D|u)(\xi)=|\xi|\widehat{u}(\xi).$$

This is a natural generalization of the classical Benjamin-Ono equation:
\begin{equation}
\label{11}
\partial_tu+\partial_x(-|D|u+u^2)=0,
\end{equation}
introduced by Benjamin \cite{BJ} and Ono \cite{Ono} as a model for one-dimensional waves in deep water. We also refer to \cite{ABFS, BK, KM, KMR, LMFR1, LMFR2} and references therein for intensively study of this equation both mathematically and numerically.

The Cauchy Problem \eqref{CP} is locally well-posed in $\h$, in the sense that for all $u_0\in H^{\frac{1}{2}}(\mathbb{R})$, there exists a unique maximal solution $u(t)\in \mathcal{C}([0,T),\h)$ for \eqref{CP}. Here $T\in(0,+\infty]$ is the maximal lifespan of this solution, we also have the following blow-up criterion: if $T<+\infty$, then
\begin{equation}
\label{12}
\lim_{t\rightarrow T^-}\|u(t)\|_{\h}=+\infty.
\end{equation}
We refer to \cite{BP,IK,KM,KMR,KPV1,KT,LMFR1,LMFR2,Vento} for the proof of the local wellposedness result (We mention here for certain $p$, sharper results are proved. But local wellposedness in $\h$ is sufficient for this paper ).
%\footnote{We mention here, most reference mentioned above dealing with the case when $p$ is an integer and for $p\not=3$, sharper results are presented. While for non-integer $p$, we may use almost the same argument as in \cite{Vento} for $k=3$ (This corresponds to $p=4$ in our notation.) to obtain a local wellposedness result in $\h$.  }
 Moreover, the mass and energy are conserved by the flow of \eqref{CP}:
\begin{equation}
\label{13}
M(u(t))=\frac{1}{2}\int u^2(t),\quad E(u(t))=\frac{1}{2}\int \big|\D u(t)\big|^2-\frac{1}{p+1}\int |u|^{p+1}(t).
\end{equation}

The equation \eqref{CP} also has the following symmetries: if $u(t,x)$ is a solution, then
\begin{equation}
\label{14}
u_{\lambda_0,t_0,x_0}(t,x)=\frac{1}{\lambda_0^{1/(p-1)}}u\bigg(\frac{t-t_0}{\lambda_0^2},\frac{x-x_0}{\lambda_0}\bigg).
\end{equation}
It is easy to see that the above transform leaves the $\dot{H}^{s_c}(\mathbb{R})$ norm of the initial data invariant, where $s_c=1/2-1/(p-1)$. The Cauchy problem \eqref{CP} is called 
\begin{itemize}
\item $L^2$-subcritical, if $p<3$ (or equivalently $s_c<0$); 
\item $L^2$-critical, if $p=3$ (or equivalently $s_c=0$); 
\item $L^2$-supercritical, if $p>3$ (or equivalently $s_c>0$).
\end{itemize}

There exists a special class of solutions called \emph{solitary waves} given by 
$$u(t,x)=Q_c(x-ct),$$
where $Q_c(y)=c^{\frac{1}{p-1}}Q(cy)$, $c>0$ and $Q\in H^{\frac{1}{2}}(\mathbb{R})$ satisfying
\begin{equation}
\label{15}
-|D|Q-Q+Q^p=0.
\end{equation}
This function $Q$ is called \emph{ground state} and is related to the best constant problem of the following Gagliardo-Nirenberg inequality:
\begin{equation}
\label{16}
\forall v\in\h,\quad \int |v|^{p+1}\leq C_p\bigg(\int \big|\D v\big|^2\bigg)^{\frac{p-1}{2}}\bigg(\int |v|^2\bigg).
\end{equation}
Existence of a nonnegative even solution to \eqref{15} was proved by Albert-Bona-Saut \cite{ABS} and Weinstein \cite{W1,W2}. Uniqueness of such a solution (up to symmetries) was proved by Amick-Toland \cite{AT}, Frank-Lenzmann \cite{FL} and Frank-Lenzmann-Silvestre \cite{FLS}. Moreover, the ground state $Q$ has the following asymptotic behavior: 
$$Q(y)\sim\frac{1}{y^2},$$
as $|y|\rightarrow+\infty.$

Recall that in the subcritical case ($p<3$), the solitary waves are orbitally stable due to \cite{BJ2, BBSSB, W1}. The solitary wave of the classical Benjamin-Ono equation \eqref{11} is asymptotically stable due to \cite{KM}. We refer to \cite{GTT,KM} for the  stability in $\h$ of the sum of solitary waves with distinct velocity%
\footnote{This case corresponds to weakly interacting multi-solitary waves.}
 for classical Benjamin-Ono equation \eqref{11}. We also refer to \cite{Ea}  for the existence of the sum of solitary waves in energy space for $L^2$ subcritical fractional KdV equations (including the original Benjamin-Ono equations).

\subsection{Main result}
In this paper, we consider the  generalized  Benjamin-Ono equations in both $L^2$ subcritical and supercritical cases. We are searching for strongly interacting multi-solitary wave of the following form:
\begin{equation}
\label{17}
u(t,x)\sim \sum_{i=1}^nQ(x-t-x_i(t)),
\end{equation}
with $x_1(t)>x_2(t)>\cdots>x_n(t)$, and $x_i(t)-x_{i+1}(t)\ll t$, as $t\rightarrow+\infty$. More precisely, we have:
\begin{theorem}[Existence]\label{MT1}
Let $p\in(2,3)\cup(3,+\infty)$, $n\geq 2$. There exist $t_0\gg1$ and a solution $u\in\mathcal{C}([t_0,+\infty),H^{\frac{1}{2}}(\mathbb{R}))$ to \eqref{CP}, with the following behavior:
\begin{equation}\label{18}
\lim_{t\rightarrow+\infty}\bigg\|u(t,\cdot)-\sum_{i=1}^n\sigma_iQ\big(\cdot-t-x_i(t)\big)\bigg\|_{\h}=0,
\end{equation}
where 
\begin{equation}
\label{130}
x_{i}(t)=\alpha_i\sqrt{t}+\beta_i\log t+\gamma_i
\end{equation}
with some universal constants $\alpha_i=\alpha_i(p,n)$, $\beta_i=\beta_i(p,n)$, $\gamma_i=\gamma_i(p,n)$, satisfying $\alpha_1>\alpha_2>\cdots>\alpha_n$. Moreover, $\sigma_i=(-1)^{i-1}$ if $2<p<3$; $\sigma_1=\cdots=\sigma_n=1$, if $p>3$.
\end{theorem}

In case of $p>3$ and $n=2$, we can also prove the uniqueness of solutions with the form \eqref{17}:
\begin{theorem}[Uniqueness]\label{MT2}
Let $p>3$, $n=2$, and $u\in\mathcal{C}([t_0,+\infty),H^{\frac{1}{2}}(\mathbb{R}))$ be a solution to \eqref{CP} satisfying:
\begin{equation}
\label{19}
\lim_{t\rightarrow+\infty}\bigg\|u(t,\cdot)-\sum_{i=1}^2\sigma_{i}Q\big(\cdot-t-x_i(t)\big)\bigg\|_{\h}=0,
\end{equation}
for some functions $x_i(t)$ and $\sigma_i\in\{\pm1\}$. Assume that $\sigma_1=1$ and 
\begin{equation}\label{110}
\lim_{t\rightarrow+\infty}x_1(t)-x_{2}(t)=+\infty,
\end{equation}
then we have $\sigma_2=1$, and
\begin{equation}\label{111}
\lim_{t\rightarrow+\infty}\frac{x_1(t)-x_2(t)}{\sqrt{t}}=\alpha_1-\alpha_2,
\end{equation}
where $\alpha_i$ are the universal constants defined in Theorem \ref{MT1}.
\end{theorem}

\begin{remark}
The solution mentioned above is also known as \emph{multi-pole solution}. The interaction between any two bubbles of the solution is strong in the sense that the relative velocities between any two bubbles are the same and the distance between them is asymptotically $O(\sqrt{t})$, which is much smaller than $O(t)$. In the context of inverse scattering theory, the existence of multi-pole solutions  are proved for many completely integrable models such as mKdV \cite{WO} and cubic NLS equations \cite{Olm}. We mention here the classic Benjamin-Ono equation \eqref{11} does not possess multi-pole solutions since the operator $L$ of the Lax pair has finitely many simple discrete eigenvalues due to \cite{Wy}. There is another example of integrable models which possesses the multi-pole solutions, the Kadomtsev-Petviashvil I (KP-I) equation. It was shown in \cite{VA} that there exist multi-pole solutions with the relative distance $O(\sqrt{t})$, the same as the generalized Benjamin-Ono equations.

There are some other examples of multiple pole solutions for non-integrable models:
\begin{enumerate}
\item Nguyen \cite{V1} constructed double pole solution for both $L^2$ subcritical and $L^2$ supercritical generalized KdV equations;
\item Nguyen \cite{V2} constructed double pole solution for $L^2$ subcritical and supercritical nonlinear Schr\"odinger equations;
\item Martel-Nguyen \cite{MN} constructed double pole solution for one dimensional cubic Schr\"odinger system;
\item Aryan \cite{As} for nonlinear Klein-Gordon equations.
\end{enumerate}
\end{remark}
\begin{remark}
There are also some other examples for blow-up solutions with strongly interacting bubbles:
\begin{enumerate}
\item Martel-Rapha\"{e}l \cite{MR} for $L^2$ critical NLS;
\item Cort\'{a}zar-Del Pino-Musso \cite{CDM} for energy critical nonlinear heat equations in domain;
\item Jendrej \cite{J1,J2} for focusing energy critical wave equations;
\item Combet-Martel \cite{CM} for $L^2$ critical gKdV equations.
\end{enumerate}
\end{remark}

\begin{remark}
In this paper, the authors dealt with the non-integrable Benjamin-Ono equations where the arguments of inverse scattering do not work. We constructed multi-pole solutions in these cases in a dynamical way. Some arguments here are similar to \cite{V1}, but due to different structures of KdV and Benjamin-Ono equations, there are some essential difficulties for the Benjamin-Ono cases:
\begin{enumerate}
\item The relative distance between nearby bubbles for multi-pole solutions for the generalized Benjamin-Ono equations is $O(\sqrt{t})$, while in the KdV case, it is $O(\log t)$. This is mainly due to the different asymptotic behaviors at infinity of the ground states for these two equations%
\footnote{We mention here the ground state in the KdV case has exponential decay at infinity, while the ground state for the Benjamin-Ono case has algebraic decay.}%
. This fact makes it much harder to construct approximate strongly-interacting multi-soliton%
\footnote{See Proposition \ref{P6} for more details.}
in the Benjamin-Ono case. Since for the approximate strongly interacting multi-soliton in both the KdV and Benjamin-Ono cases, we require an estimate of order $t^{-2-\delta}$ with $\delta>0$ for the error term. But in Benjamin-Ono case, there is a huge number of ``lower order terms" (terms which decays slower than or the same as $t^{-2}$) appearing, while in the KdV case, the lowest order term is just $t^{-2}$ and they are only generated by the interaction between each bubbles. Therefore, constructing strongly interacting multi-soliton in the Benjamin-Ono case is much more involved.
\item The nonlocal structure of \eqref{CP} makes the analysis much more complicated than the KdV case. For example, some arguments of integration by parts in the KdV cases have to be replaced by some complicated commutator estimates introduced in \cite{C,DMP,KPV,MP}.
\item Unlike the KdV case, we don't have an explicit expression for the ground state in the Benjamin-Ono case when $p>2$. To understand the interaction between different bubbles, we need to know the exact asymptotic behavior of the ground state at infinity. A lack of explicit expression will creates difficulties in this issue.
\item In contrast to \cite{V1,V2}, we also consider the multi-bubble cases, where the interaction between different bubbles are much more complicated than the two-bubble case. We also mention here, due to the fact that the relative distance between nearby bubbles for multi-pole solutions for the generalized Benjamin-Ono equations is $O(\sqrt{t})$, one has to consider interaction between any two bubbles not just nearby ones, which is more complicated than the KdV cases where for every fixed bubble, one only needs to consider the interaction of the nearest one.
\end{enumerate}
\end{remark}

\begin{remark}
In the $L^2$ critical case, the authors conjectured that such solutions don't exist. However, this type of results still remain open. 
\end{remark}

\begin{remark}
Uniqueness of solutions with asymptotic behavior \eqref{17} when $n\geq 3$ remains open. This is due to the complicated interaction between different bubbles in those cases. While in the $L^2$ subcritical case and $n=2$, we cannot use the energy conservation law to control the error term like the $L^2$ supercritical case, which prevents us from obtaining the uniqueness result.
\end{remark}

\subsection{Notations}
We first introduce the scaling generator:
\begin{equation}\label{113}
\Lambda f(y)=\frac{\partial f_c(y)}{\partial c}\bigg|_{c=1}=\frac{1}{p-1}f(y)+yf'(y),
\end{equation}
where
\begin{equation}
\label{114}
f_c(y)=c^{\frac{1}{p-1}}f(cy).
\end{equation}
Then we introduce the linearized operators at $Q$:
\begin{equation}\label{116}
\mathcal{L}f=|D|f+f-pQ^{p-1}f.
\end{equation}

For all $s>0$, we denote by $\mathcal{Y}_s$ the set of smooth functions $f$ satisfying the following conditions: for all $k\in\mathbb{N}$, there exists a constant $C_k>0$ such that
$$|\partial_y^kf(y)|\leq C_k\langle y\rangle^{-k-s},$$
where $\langle y\rangle=\sqrt{1+y^2}$.

We denote the $L^2$ scalar product on $\mathbb{R}$ by:
\begin{equation}\label{115}
(f,g)=\int_{\mathbb{R}}f(x)g(x)dx.
\end{equation}

For suitable operators $A, B$, we denote by $[A,B]=AB-BA$ the commutator of $A$ and $B$.

Finally, we denote by $\delta(\alpha)$ a small positive constant such that:
\begin{equation}\label{117}
\lim_{\alpha\rightarrow0^+}\delta(\alpha)=0.
\end{equation}

\subsection{Outline of the proof}
%This paper is organized as follows: In Section \ref{S2}, we introduce some basic spectral properties of the ground state $Q$. Then in Section \ref{S3} we will construct an approximate strongly interacting multi-bubble for general cases, and consider the geometrical decomposition near this multi-soliton. We will obtain some crucial modulation estimates on the parameters. In Section \ref{S4}, we prove some crucial backward uniform estimate using the methods of modulation argument, energy estimates and topological argument. Then we finish the proof of Theorem \ref{MT1} by a compactness argument. Finally, in Section \ref{S5}, we prove the uniqueness of such solution in case of $n=2$ and $p>3$.
%
%Here is an outline of the proof of Theorem \ref{MT1} and Theorem \ref{MT2}.

\subsubsection{Formulation of the system}
We consider solutions to \eqref{CP} with the following form:
$$u(t,y+t)\sim \sum_{i=1}^n\sigma_iQ_{1+\mu_i(t)}(y-x_i(t)).$$
To classify the (strong) interaction between each bubble, we introduce an interaction term $r(t,y)$ of the following form 
$$r(t,y)=\sum_{\substack{i,j=1\\ j\not=i}}^n\frac{A_{ij}(t,y-x_{i}(t))}{x^2_{ij}(t)}+\sum_{\substack{i,j=1\\ j\not=i}}^n\frac{B_{ij}(t,y-x_{i}(t))}{x_{ij}^3(t)}\varphi_{ij}(t,y),$$
where $A_{ij}$ and $B_{ij}$ are functions to be chosen later, and $\varphi_{ij}$ are some suitable cut-off functions. By direct computation of the interaction between each bubbles, we have the following formal ODE system of the parameters $x_i(t)$ and $\mu_i(t)$:
\begin{equation}\label{119}
\dot{x}_i\sim\mu_i,\quad\dot{\mu}_i+\sum^n_{\substack{j=1,\\j\not=i}}\frac{a_{ij}}{(x_i-x_j)^3}\sim 0,
\end{equation}
for all $i=1,\ldots,n$, where $a_{ij}$ are some explicit constants depending only on $p$ and the choice of the sign $\sigma_i$. With a suitable choice of the sign $\sigma_i$, the ODE system \eqref{119} leads to the following asymptotic behaviors of the parameters:
\begin{equation}\label{120}
x_i(t)\sim\alpha_i\sqrt{t},\quad \mu_i(t)\sim\frac{\alpha_i}{2\sqrt{t}},
\end{equation} 
for all $i=1,\ldots,n$, where $\alpha_i$ are some universal constants depending on $p,n$, with $\alpha_1>\cdots>\alpha_n$.

\subsubsection{Construction of approximate strongly interacting multi-bubbles}
Let 
$$V(t,y)=\sum_{i=1}^nQ_{1+\mu_i(t)}(y-x_i(t))+r(t,y).$$ 
Then $V$ should be an approximate solution of 
$$\partial_t V-\partial_y(|D|V+V-|V|^{p-1}V)=0.$$
For technical reason, we require an estimate on the error term of the form $t^{-2-\delta}$, where $\delta>0$. But, when considering the scaling symmetry, there will  be a lot of ``lower order terms" appearing%
\footnote{Here, by ``lower order terms" we mean terms which asymptotically behave like $t^{-\alpha}$, with $0<\alpha\leq 2$.}
. Meanwhile, the interaction between different bubbles also contain ``lower order terms". We have to choose some functions%
\footnote{See Proposition \ref{P6} and Proposition \ref{P11} for more details.}
 $A_{ij}$, $B_{ij}$ carefully so that all ``lower order terms" generated by scaling and interaction will be canceled. We mention here this is new feature for generalized Benjamin-Ono equations when comparing to the KdV cases. 

Due to these ``lower order terms", the ODE system \eqref{119} has to be replaced by the following refined one:
\begin{equation}\label{121}
\dot{x}_i\sim\mu_i,\quad\dot{\mu}_i+\sum^n_{\substack{j=1,\\j\not=i}}\frac{a_{ij}}{(x_i-x_j)^3}+\sum_{\substack{j,k=1,\\j\not=i}}^n\frac{b_{ijk}\mu_k}{(x_i-x_j)^3}\sim 0,
\end{equation}
for all $i=1,\ldots,n$, where $b_{ijk}$ are universal constants depending on $p,n$. Hence, the asymptotic behavior of the parameters \eqref{120} can be improved to
\begin{equation}\label{122}
x_i(t)\sim\alpha_i\sqrt{t}+\beta_i\log t+\gamma_i,\quad \mu_i(t)\sim\frac{\alpha_i}{2\sqrt{t}}+\frac{\beta_i}{t},
\end{equation} 
for all $i=1,\ldots,n$, where $\beta_i,\gamma_i$ are universal constants.

\subsubsection{Modulation estimates}
For all initial data close to the multi-soliton
$$\sum_{i=1}^n\sigma_iQ_{1+\mu_{i,0}}(y-x_{i,0}),$$
a standard argument of implicit function theorem, there exist parameters $x_i(t)$, $\mu_i(t)$ and an error term $\e(t,y)$ such that 
$$u(t,y+t)=V(t,y)+\e(t,y)$$
and $\e$ satisfies some well-chosen orthogonality conditions%
\footnote{See \eqref{38} for more details.}
. As a direct consequence, the parameters $x_i(t)$, $\mu_i(t)$ satisfy the approximate ODE system \eqref{121}.

\subsubsection{Existence of the strongly interacting multi-soliton}
The most crucial part for the proof of Theorem \ref{MT1} is to establish a uniform backward estimate%
\footnote{See Proposition \ref{P9} for more details.} 
. This estimate claims that for some fixed time $t_0$, we have for all large enough $\tin>0$, there exists a suitable choice of initial data, such that the parameters $x_i(t)$ and $\mu_i(t)$ satisfy \eqref{122} on $[t_0,\tin]$. The proof can be divide into two parts:
\begin{enumerate}
\item Solving the perturbed ODE system \eqref{121};
\item Complete the estimate on the error term $\e$.
\end{enumerate}
The first part requires a suitable choice of $x_i(\tin)$ and $\mu_i(\tin)$ which is done by a topological argument. While for the second part, we have to consider the following localized energy functional:
\begin{align*}
W(t)=&\int\big|\D (\e\sqrt{\Phi_1})\big|^2+\e^2(\Phi_1+\Phi_2)\nonumber\\
&-\frac{2}{p+1}\Big[\big(|V+\e|^{p+1}-|V|^{p+1}-(p+1)|V|^{p-1}V\e\big)\Big]\Phi_1,
\end{align*}
for some well-chosen weight functions $\Phi_1,\Phi_2$. The energy conservation law of \eqref{CP}, will lead to some important monotonicity formula for $W$%
\footnote{See Proposition \ref{P10} for more details.}
. Now, it suffices to show that $W$ is coercive: $W\sim\|\e\|^2_{\h}$. The proof of this estimate is different for the sub-critical and supercritical cases. Since, in the sub-critical cases, the unstable direction of $W$ can be controlled by the orthogonality conditions \eqref{38}. Hence, we can directly choose $\e(\tin)=0$ to conclude the proof. But in the supercritical cases, this argument does not work. We have to control the directions $(\e,Z^\pm(\cdot-x_i))$, where $Z^\pm$ are eigenfunctions of $\mathcal{L}\partial_y$. We need an additional topological argument to find suitable initial data such that $W$ is coercive on $[t_0,\tin]$. Similar argument can also be found in \cite{CMM,V1}.

With this uniform backward estimate, we can easily prove Theorem \ref{MT1} by a compactness argument.

\subsubsection{Uniqueness of two-bubble solutions in the supercritical cases}
We will show that in case of $p>3$ and $n=2$, the solution mentioned in Theorem \ref{MT1} is unique in a certain sense. 

The proof is parallel to the proof of the existence result. We first establish estimates for the parameters and error terms by modulation arguments. Again, in the supercritical cases, we have to control the direction $(\epsilon,Z^\pm(\cdot-\mathfrak{q}_i))$. The most important part is to use the coercivity of two-soliton induced by the energy conservation law to show that the two bubble must have the same sign. The remaining part of the proof is just to solve the perturbed ODE system \eqref{121}, which is similar to the proof of the existence result.

\subsection*{Acknowledgements}
Y. Lan acknowledges the support of the China National Natural Science Foundation under grant number 12201340. Z. Wang acknowledges the support of the China National Natural Science Foundation under grant number 11901092 and Guangdong Natural Science Foundation under grant number 2023A1515010706. The authors are indebted to Prof. Yvan Martel for introducing the problem and stimulating discussions.

\section{Properties of the ground state}\label{S2}
In this section, we list some basic properties about the ground sate $Q$ and the linearized operator $\mathcal{L}$. Most of them can be found in \cite{AT,FL,J3,MP,PW,W1,W2}.

\subsection{Spectral properties of the linearized operator}
We start with some basic properties about the ground state $Q$. Most of these properties are proved in \cite{FL,FLS}:
\begin{proposition}\label{P1}
Let $p\geq2$, consider functional
$$I(u)=\frac{\big\|\D u\big\|_{L^2}^{p-1}\|u\|_{L^2}^{p+1}}{\|u\|_{L^{p+1}}^{p+1}},\quad \forall u\in H^{\frac{1}{2}}(\mathbb{R})\backslash\{0\}.$$
We have:
\begin{enumerate}
\item \textbf{Existence:} There exists a minimizer $Q\in \h\cap\mathcal{Y}$ of $I(u)$, which is an even positive smooth function on $\mathbb{R}$ satisfying \eqref{15}. 
\item \textbf{Asymptotic behavior: }The function $Q$ has the following asymptotic behavior at infinity: 
\begin{align}
&Q(y)=\frac{\kappa_0}{y^2}+\frac{g(y)}{y^4}+O\bigg(\frac{1}{y^6}\bigg),\;\;\text{as }|y|\rightarrow+\infty,\label{21}\\
&Q'(y)=-\frac{2\kappa_0}{y^3}+O\bigg(\frac{1}{|y|^5}\bigg),\;\;\text{as }|y|\rightarrow+\infty.\label{233}
\end{align} 
where $\kappa_0>0$ and $g\in L^\infty(\mathbb{R})\cap C^1(\mathbb{R})$. Moreover, we have
\begin{equation}\label{234}
|g'(y)|=O\bigg(\frac{1}{|y|}\bigg)\;\;\text{as }|y|\rightarrow+\infty.
\end{equation}
\item \textbf{Uniqueness: }Any minimizer of $I$ must have the form $\alpha_0 Q_{c_0}(\cdot-x_0)$, where $\alpha_0\in\mathbb{C}\backslash\{0\}$, $c_0>0$, $x_0\in\mathbb{R}$.
\end{enumerate}
\end{proposition}
\begin{remark}
Existence, uniqueness and regularity of the minimizer $Q$ are proved in \cite{FL,FLS}. We will prove the asymptotic formula \eqref{21}--\eqref{234} in Appendix \ref{Ap1}.
\end{remark}

Next, we recall some properties about the linearized operator $\mathcal{L}$ at $Q$:
\begin{proposition}\label{P2}
$\mathcal{L}$ is a self-adjoint operator in $L^2$ with domain $H^1$. Moreover, the following properties hold:
\begin{enumerate}
\item \textbf{Spectrum:} The operator $\mathcal{L}$ has exactly one negative eigenvalue $-\kappa$ ($\kappa>0$) associated to an even positive function $\chi_0$; $\sigma_{\rm ess}(\mathcal{L})=[1,+\infty)$; $\ker \mathcal{L}=\{a Q':\,a\in\mathbb{R}\}$.
\item \textbf{Scaling:} $\mathcal{L}\Lambda Q=-Q$;
\item \textbf{Regularity:} if $f\in H^1$ such that $\mathcal{L}f\in\mathcal{Y}_1$, then $f\in\mathcal{Y}_1$;
\item \textbf{Invertibility:} for all $g\in L^2$ such that $(g,Q')=0$, then there exists a unique $f\in L^2$ with $(f,Q')=0$ and $\mathcal{L}f=g$;
\item \textbf{Coercivity: }if $p<3$, then there exists a universal constant $\mu>0$ such that
\begin{equation}
\label{226}
(\mathcal{L}v,v)\geq \mu\|v\|_{\h}^2-\frac{1}{\mu}\big[(v,Q)^2+(v,Q')^2\big].
\end{equation}
\end{enumerate} 
\end{proposition}
\begin{proof}
We refer to \cite{FL,MP,W1,W2} for the proof of (1)--(4).

Finally, property (5) follows from the fact $Q$ is the unique minimizer (up to symmetries) of the following minimizing problem:
$$\inf_{\substack{u\in \h,\\\|u\|_{L^2}=\|Q\|_{L^2}}}E(u)$$
when $p<3$.
\end{proof}

Now, we introduce some spectral properties about the operators $\partial_y\mathcal{L}$ and $\mathcal{L}\partial_y$ in the supercritical case $p>3$:
\begin{proposition}\label{P3}
There exist functions $Y^\pm\in\mathcal{Y}_2$ and $e_0>0$ such that
\begin{gather}
\partial_y(\mathcal{L}Y^-)=-e_0Y^-,\quad \partial_y(\mathcal{L}\textbf{}Y^+)=e_0Y^+,\label{22}\\
Y^{+}(y)=Y^-(-y),\quad\int_{\mathbb{R}}Y^\pm=0\label{23},\\
(Y^-,\mathcal{L}Y^-)=(Y^+,\mathcal{L}Y^+)=0,\label{24}\\
(Y^-,\mathcal{L}Y^+)=(Y^+,\mathcal{L}Y^-)\not=0.
\end{gather}
\end{proposition}
%We leave the proof of this Proposition in Appendix \ref{A}.
The existence of $Y^\pm$ is proved by \cite{KS,Lz}. The rest part of the proof of this Proposition is similar to \cite[Proposition 2.9]{J3}.
\begin{proposition}\label{P4}
There exist functions $Z^\pm\in\mathcal{Y}_2$ such that
\begin{enumerate}
\item $\mathcal{L}(\partial_y Z^-)=-e_0Z^-$, $\mathcal{L}(\partial_y Z^+)=e_0Z^+$;
\item $(Z^\pm,Y^\pm)=1$, $(Z^\pm,Y^\mp)=0$, $(Z^\pm,Q')=0$;
\item Fix a $\mathcal{Z}\in L^2$ such that $(\mathcal{Z},Q')\not=0$. There exists $\mu_0>0$ such that for all $v\in \h$, we have
\begin{equation}
\label{25}
(v,\mathcal{L}v)\geq \mu_0\|v\|^2_{\h}-\frac{1}{\mu_0}\big[(Z^-,v)^2+(Z^+,v)^2+(\mathcal{Z},v)^2\big].
\end{equation}
\end{enumerate}
\end{proposition} 
The proof of proposition \ref{P4} is almost the same as \cite[Proposition 2.10, Lemma 2.11, Proposition 2.13]{J3}. We omit the details here.

\subsection{Coercivity near a two-soliton} We introduce the following functional:
$$H(u)=E(u)+M(u)=\frac{1}{2}\int u^2+\frac{1}{2}\int \big|\D u\big|^2-\frac{1}{p+1}\int u^{p+1},$$
for $u\in \h$. We are interested in the coercivity properties of $H$ near a two-soliton.

More precisely, we have:
\begin{proposition}\label{P5}
Let us fix $\sigma\in\{\pm1\}$ and $\mathcal{Z}\in L^2$ such that $(\mathcal{Z},Q')\not=0$. There exist constants $\omega_0,z_0,C_0>0$, such that if $\e, x_1,x_2$ satisfy 
$$H(Q(\cdot-x_1)+\sigma Q(\cdot-x_2)+\e)=2H(Q),$$
and $\|\e\|_{\h}\leq \omega_0$, $x_1-x_2\geq z_0$, then
\begin{enumerate}
\item if $\sigma=1$, there holds
\begin{align}\label{26}
\|\e\|^2_{\h}\leq C_0\bigg[&\frac{1}{(x_1-x_2)^2}+\big(Z^-(\cdot-x_1),\e\big)^2+\big(Z^+(\cdot-x_1),\e\big)^2+\big(\mathcal{Z}(\cdot-x_1),\e\big)^2\nonumber\\
&+\big(Z^-(\cdot-x_2),\e\big)^2+\big(Z^+(\cdot-x_2),\e\big)^2+\big(\mathcal{Z}(\cdot-x_2),\e\big)^2\bigg];
\end{align}
\item if $\sigma=-1$, there holds
\begin{align}\label{27}
\|\e\|^2_{\h}+&\frac{1}{(x_1-x_2)^2}\leq C_0\Big[\big(Z^-(\cdot-x_1),\e\big)^2+\big(Z^+(\cdot-x_1),\e\big)^2+\big(\mathcal{Z}(\cdot-x_1),\e\big)^2\nonumber\\
&+\big(Z^-(\cdot-x_2),\e\big)^2+\big(Z^+(\cdot-x_2),\e\big)^2+\big(\mathcal{Z}(\cdot-x_2),\e\big)^2\Big].
\end{align}
\end{enumerate}
\end{proposition}
To prove Proposition \ref{P5}, we need the following two lemmas:
\begin{lemma}\label{L1}
Let $\sigma\in\{\pm1\}$ and $\mathcal{Z}\in L^2$ such that $(\mathcal{Z},Q')\not=0$, then there exist constants $\omega_0,z_0,\lambda_0>0$, such that if $U, x_1,x_2$ satisfy $\|U-Q(\cdot-x_1)-\sigma Q(\cdot-x_2)\|_{\h}\leq \omega_0$ and $x_1-x_2\geq z_0$, then for all $\e\in\h$, there holds
\begin{align}
\label{28}
&\big(\e,D^2H(U)\e\big)\nonumber\\
&\geq\lambda_0\|\e\|_{\h}^2-\frac{1}{\lambda_0}\Big[\big(Z^-(\cdot-x_1),\e\big)^2+\big(Z^+(\cdot-x_1),\e\big)^2+\big(\mathcal{Z}(\cdot-x_1),\e\big)^2\nonumber\\
&\qquad\qquad+\big(Z^-(\cdot-x_2),\e\big)^2+\big(Z^+(\cdot-x_2),\e\big)^2+\big(\mathcal{Z}(\cdot-x_2),\e\big)^2\Big],
\end{align}
where $D^2H(U)=|D|-p|U|^{p-1}+1$.
\end{lemma}

\begin{lemma}\label{L2}
There exists a positive constant $\kappa_0>0$ such that
\begin{equation}
\label{29}
H(Q(\cdot-x_1)+\sigma Q(\cdot-x_2))=2H(Q)-\frac{\sigma\kappa_0}{(x_1-x_2)^2}+O\bigg(\frac{1}{|x_1-x_2|^3}\bigg),
\end{equation}
as $x_1-x_2\rightarrow+\infty$.
\end{lemma}

\begin{proof}[Proof of Lemma \ref{L1}]
Without loss of generality, we can assume that $x_2=0$ and $x_1\geq z_0$. Consider the following operator:
$$T'=|D|+1-pQ^{p-1}-pQ^{p-1}(\cdot-x_1).$$
Then we have
\begin{equation}
\label{210}
\big(\e,D^2H(U)\e\big)-(\e,T'\e)=-p\int\big[|U|^{p-1}-Q^{p-1}-Q^{p-1}(\cdot-x_1)\big]\e^2.
\end{equation}
By H\"older's inequality, we have
\begin{align}
\label{211}
&\bigg|\int\big[|U|^{p-1}-Q^{p-1}-Q^{p-1}(\cdot-x_1)\big]\e^2\bigg|\lesssim\|\e\|_{L^4}\Big\||U|^{p-1}-\big|Q+Q(\cdot-x_1)\big|^{p-1}\Big\|_{L^2}\nonumber\\
&\qquad+\|\e\|_{L^4}\Big\|\big|Q+Q(\cdot-x_1)\big|^{p-1}-Q^{p-1}-Q^{p-1}(\cdot-x_1)\Big\|_{L^2}.
\end{align}
By Gagliardo-Nirenberg's inequality, we know that 
\begin{align}
\label{212}
&\Big\||U|^{p-1}-\big|Q+Q(\cdot-x_1)\big|^{p-1}\Big\|_{L^2}\nonumber\\
&\lesssim \big\|U-Q-\sigma Q(\cdot-x_1)\big\|_{L^{2(p-1)}}\Big(\|U\|^{p-2}_{L^{2(p-1)}}+\|Q+Q(\cdot-x_1)\|^{p-2}_{L^{2(p-1)}}\Big)\lesssim \delta(\alpha).
\end{align}
On the other hand, by integrating on the regions $x\leq x_1/2$ and $x> x_1/2$ separately, we know that
\begin{equation}
\label{213}
\Big\|\big|Q+Q(\cdot-x_1)\big|^{p-1}-Q^{p-1}-Q^{p-1}(\cdot-x_1)\Big\|_{L^2}\leq \frac{1}{1000},
\end{equation}
for $x_1=x_1-x_2$ sufficiently large. Combining \eqref{210}, \eqref{211}, \eqref{212} and \eqref{213}, we have 
\begin{equation}
\label{214}
\big|\big(\e,D^2H(U)\e\big)-(\e,T'\e)\big|\leq\frac{1}{100}\|\e\|^2_{\h}.
\end{equation}

Finally, we consider a decomposition $\e=\e_1+\e_2$, such that 
$$\text{Supp }\e_1\subset (-\infty,2x_1/3],\quad\text{Supp }\e_2\subset [x_1/3,+\infty)$$
and $\|\e\|_{\h}\sim \|\e_1\|_{\h}+\|\e_2\|_{\h}$. This can be easily done by letting $\e_1=\phi\e$, $\e_2=(1-\phi)\e$ for some suitable smooth function $\phi$. When $x_1$ is large enough, we can easily show that
\begin{align}
\label{215}&\big(Z^-(\cdot-x_1),\e_1\big)^2+\big(Z^+(\cdot-x_1),\e_1\big)^2+\big(\mathcal{Z}(\cdot-x_1),\e_1\big)^2\leq \frac{1}{100}\|\e\|^2_{L^2}\\
\label{216}&\big(Z^-(\cdot-x_2),\e_2\big)^2+\big(Z^+(\cdot-x_2),\e_2\big)^2+\big(\mathcal{Z}(\cdot-x_2),\e_2\big)^2\leq \frac{1}{100}\|\e\|^2_{L^2}.
\end{align}
Combining \eqref{25}, \eqref{214}, \eqref{215} and \eqref{216}, we conclude the proof of Lemma \ref{L1}.
\end{proof}

\begin{proof}[Proof of Lemma \ref{L2}]
We introduce the following notation:
$$R_1(y)=Q(y-x_1),\quad R_2(y)=Q(y-x_2).$$
Then we have:
\begin{align}
\label{217}
&H(R_1+\sigma R_2)\nonumber\\
=&2H(Q)+\int\bigg[\sigma (\D R_1)(\D R_2)+\sigma R_1R_2-\sigma R_1R_2^p-\sigma R_2R_1^p\nonumber\\
&\quad-\frac{1}{p+1}\big(|R_1+\sigma R_2|^{p+1}-R_1^{p+1}-R_2^{p+1}-(p+1)\sigma R_1R_2^p-(p+1)\sigma R_2R_1^p\big)\bigg]\nonumber\\
=&2H(Q)-\sigma\int R_1R_2^p\nonumber\\
&\quad-\frac{1}{p+1}\int\big(|R_1+\sigma R_2|^{p+1}-R_1^{p+1}-R_2^{p+1}-(p+1)\sigma R_1R_2^p-(p+1)\sigma R_2R_1^p\big).
\end{align}

We first estimate the term $\int R_1R_2^p$. For $y>(x_1+x_2)/2$, we have:
\begin{equation}
\label{218}
\int_{y>(x_1+x_2)/2} R_1R_2^p\lesssim\frac{1}{|x_1-x_2|^{2p}}\int R_1=O\bigg(\frac{1}{|x_1-x_2|^3}\bigg),
\end{equation}
as $x_1-x_2\rightarrow+\infty$. Recall that the ground state $Q$ satisfies $Q(y)\sim 1/y^2$. For $y<(x_1+x_2)/2$, using \eqref{21}, we have:
\begin{align}
\label{219}
&\int_{y<(x_1+x_2)/2} R_1R_2^p=\kappa_0\int_{y<(x_1+x_2)/2}  \frac{Q^p(y-x_2)}{(y-x_1)^2}\,\dd y+O\bigg(\int_{y<(x_1+x_2)/2}\frac{Q^p(y-x_2)}{(y-x_1)^4}\,\dd y\bigg)\nonumber\\
&=\kappa_0\int_{y<(x_1+x_2)/2}  \frac{Q^p(y-x_2)}{(y-x_1)^2}\,\dd y+O\bigg(\frac{1}{|x_1-x_2|^3}\bigg).
\end{align}
Since $y-x_1=y-x_2+x_2-x_1$, we have
\begin{align}
\label{220}
&\int_{y<(x_1+x_2)/2}  \frac{Q^p(y-x_2)}{(y-x_1)^2}\,\dd y\nonumber\\
&=\frac{\kappa_0}{(x_2-x_1)^2}\int_{y<(x_1+x_2)/2}  {Q^p(y-x_2)}\frac{1}{[1+(y-x_2)/(x_2-x_1)]^2}\,\dd y\nonumber\\
&=\frac{\kappa_0}{(x_2-x_1)^2}\int_{y<(x_1+x_2)/2}  {Q^p(y-x_2)}\bigg[1+O\bigg(\frac{|y-x_2|}{|x_2-x_1|}\bigg)\bigg]\,\dd y \nonumber\\
&=\frac{\kappa_0}{(x_2-x_1)^2}\int{Q^p(y-x_2)}\,\dd y+O\bigg(\frac{1}{|x_1-x_2|^3}\bigg).
\end{align}
Combining \eqref{218}, \eqref{219} and \eqref{220}, we have
\begin{equation}
\label{221}
\int R_1R_2^p=\frac{\kappa_0}{(x_2-x_1)^2}\int Q^p+O\bigg(\frac{1}{|x_1-x_2|^3}\bigg).
\end{equation}

Next, we estimate the nonlinear term. We consider a function 
$$F(x)=|1+x|^{p+1}-1-|x|^{p+1}-(p+1)x-(p+1)x|x|^{p-1}.$$
Since $p>2$, it is easy to verify that $F\in C^2([-2,2])$ and $F(0)=F'(0)=0$. Hence, there exists a constant $C_0>0$, such that 
\begin{equation}
\label{222}
\forall x\in[-2,2],\quad|F(x)|\leq C_0|x|^2.
\end{equation}
For any fixed $y\in\mathbb{R}$, if $R_1(y)\leq |R_2(y)|$, then we choose%
\footnote{Here we use the fact that for all $y\in\mathbb{R}$, $R_2(y)\not=0$.}
 $x=R_1(y)/[\sigma R_2(y)]$, by \eqref{222}, we have
\begin{equation}
\label{223}
\big|\big(|R_1+\sigma R_2|^{p+1}-R_1^{p+1}-R_2^{p+1}-(p+1)\sigma R_1R_2^p-(p+1)\sigma R_2R_1^p\big)\big|\lesssim R_1^2R_2^{p-1}+R_2^2R_1^{p-1}.
\end{equation}
If $R_1(y)> |R_2(y)|$, by choosing $x=\sigma R_2(y)/R_1(y)$, we can also obtain \eqref{223}. Hence, by separating the integral on regions $y>(x_1+x_2)/2$ and $y\leq (x_1+x_2)/2$, we have:
\begin{align}
\label{224}
&\frac{1}{p+1}\bigg|\int\big(|R_1+\sigma R_2|^{p+1}-R_1^{p+1}-R_2^{p+1}-(p+1)\sigma R_1R_2^p-(p+1)\sigma R_2R_1^p\big)\bigg|\nonumber\\
&\lesssim \int R_1^2R_2^{p-1}+R_2^2R_1^{p-1}=O\bigg(\frac{1}{|x_1-x_2|^3}\bigg).
\end{align}
Collecting \eqref{221} and \eqref{224}, we conclude the proof of Lemma \ref{L2}.
\end{proof}

Now, we are able to prove Proposition \ref{P5}:
\begin{proof}[Proof of Proposition \ref{P5}]
Let $U=R_1+\sigma R_2$, where 
$$R_1(y)=Q(y-x_1),\quad R_2(y)=Q(y-x_2).$$
By Taylor's expansion, we have:
\begin{equation*}
H(U+\e)=H(U)+(DH(U),\e)+\frac{1}{2}(D^2H(U)\e,\e)+O\big(\|\e\|^3_{\h}\big).
\end{equation*}
From Lemma \ref{L1} and Lemma \ref{L2}, we have
\begin{align*}
&\frac{\lambda_0}{2}\|\e\|^2_{\h}-\frac{\sigma\kappa_0}{(x_1-x_2)^2}+(DH(U),\e)\nonumber\\
&\leq \frac{1}{2\lambda_0}\bigg[\big(Z^-(\cdot-x_1),\e\big)^2+\big(Z^+(\cdot-x_1),\e\big)^2+\big(\mathcal{Z}(\cdot-x_1),\e\big)^2\nonumber\\
&\quad+\big(Z^-(\cdot-x_2),\e\big)^2+\big(Z^+(\cdot-x_2),\e\big)^2+\big(\mathcal{Z}(\cdot-x_2),\e\big)^2\bigg]+O\bigg(\frac{1}{|x_1-x_2|^3}\bigg).
\end{align*}
Now it suffices to prove that
\begin{equation*}
|(DH(U),\e)|\ll \|\e\|^2_{\h}+\frac{1}{(x_1-x_2)^3}.
\end{equation*}
By Cauchy-Schwarz inequality, we only need to prove that
\begin{equation}
\label{225}
\big\|(R_1+\sigma R_2)|R_1+\sigma R_2|^{p-1}-R_1^p-\sigma R_2^p\big\|_{L^2}\lesssim \frac{1}{|x_1-x_2|^2}.
\end{equation}
This can be done by a similar argument as we did for \eqref{224}. Hence, we conclude the proof of Proposition \ref{P5}.
\end{proof}

\section{Modulation argument}\label{S3}
In this section, we will construct a strongly interacting multi-bubble, and introduce the geometrical decomposition for solutions near this multi-bubble. Some suitable orthogonality conditions will be chosen, which will lead to some important modulation estimates on the parameters.
\subsection{Strongly interacting multi-bubble}
We first introduce the following change of variables:
$$w(t,y)=u(t,y+t),$$
where $u(t,x)$ is a solution of \eqref{CP}. We can easily see that $w$ satisfies:
\begin{equation}
\label{316}
w_t-\partial_y(|D|w+w-|w|^{p-1}w)=0.
\end{equation}

We will construct a multi-soliton of the following form:
$$w(t,y)\sim\sum_{i=1}^n\sigma_iQ_{1+\mu_i(t)}(y-x_i(t))+r(t,y),$$
where $r(t,y)$ is the strongly interacting term and $\sigma_i\in\{\pm1\}$. 

More precisely, let $I=[t_0,t_1]$ be an interval, and $\{x_{i}(t)\}_{i=1}^n$, $\{\mu_{i}(t)\}_{i=1}^n$ be $C^1$-functions defined on $I$ and $\sigma_i\in\{\pm1\}$. We \emph{a priorily} assume that for all $t\in I$
\begin{equation}
\label{36}
|\mu_{i}(t)|\leq \omega_0,\quad d(t):=\min_{i\in\{1,\ldots,n-1\}}[x_{i}(t)-x_{i+1}(t)]\geq \frac{1}{\omega_0}>0,
\end{equation}
where $\omega_0\ll1$ is a small constant. Let $\psi$ be a smooth function such that $\psi(y)=1$ if $y>-1$; $\psi(y)=0$ if $y<-2$. For all $i\not=j$, we denote by
\begin{equation}
\label{314}
\varphi_{ij}(t,y):=\psi\bigg(\frac{y-x_i(t)}{x_{ij}^3(t)}\bigg),
\end{equation}
where $x_{ij}(t)=x_{i}(t)-x_{j}(t)$. We are searching for a multi-soliton of the following form
\begin{align}
\label{35}
V(t,y)&=\Theta(\vec{x}(t),\vec{\mu}(t),y)=\sum_{i=1}^n\sigma_iQ_{1+\mu_{i}(t)}(y-x_{i}(t))+\sum_{\substack{i,j=1\\ j\not=i}}^n\frac{A_{ij}(t,y-x_{i}(t))}{x^2_{ij}(t)}\nonumber\\
&+\sum_{\substack{i,j=1\\ j\not=i}}^n\frac{B_{ij}(t,y-x_{i}(t))}{x_{ij}^3(t)}\varphi_{ij}(t,y),
\end{align}
where $\vec{x}(t)=(x_{1}(t),\ldots,x_{n}(t))$, and  $\vec{\mu}(t)=(\mu_{1}(t),\ldots,\mu_{n}(t))$. 
\begin{remark}
We will choice functions $A_{ij}(t), B_{ij}(t)$ suitably such that $V(t,y)$ is an approximate solution to \eqref{316} with an error term of order $1/d^{9/2}(t)$. 
\end{remark}
\begin{remark}
We will see from the construction of $A_{ij}, B_{ij}$ that the functions $A_{ij},B_{ij}$ are time dependent, but they actually depend only on the parameters $\vec{x}(t)$ and $\vec{\mu}(t)$. They don't explicitly depend on the time $t$. 
\end{remark}
More precisely, we have:

\begin{proposition}\label{P6}
Let $p>2$, $p\not=3$, $n\geq 2$ and $\sigma_i\in\{\pm1\}$. We assume that $\{x_{i}(t)\}_{i=1}^n$, $\{\mu_{i}(t)\}_{i=1}^n$ are $C^1$-functions defined on $I=[t_0,t_1]$ satisfying \eqref{36}. Then for all $i\not=j$, $k=1,\ldots,n$ there exist smooth bounded functions $A_{ij}(t), B_{ij}(t)$ and constants $a_{ij}, b_{ijk}\in\mathbb{R}$, such that the following properties hold. Let $\Psi_V=\partial_tV-\partial_y(|D|V+V-|V|^{p-1}V)$, and 
\begin{equation}
\R_i(t,y)=Q_{1+\mu_i(t)}(y-x_i(t)).
\end{equation}
Then we have, for all $t\in [t_0,t_1]$,
\begin{align}
\label{322}
\Psi_V=&E_V-\sum_{i=1}^n(\dot{x}_i-\mu_i)\sigma_i\partial_y\widetilde{R}_i+\sum_{i=1}^n\bigg(\dot{\mu}_i+\sum^n_{\substack{j=1,\\j\not=i}}\frac{a_{ij}}{x^3_{ij}}+\sum^n_{\substack{k,j=1,\\j\not=i}}\frac{b_{ijk}\mu_k}{x^3_{ij}}\bigg)\frac{\sigma_i\Lambda\widetilde{R}_i}{1+\mu_i},
\end{align}
% where
where $a_{ij}$ are given by
\begin{equation}
\label{320}
a_{ij}=\frac{4\sigma_i\sigma_j\kappa_0(p-1)\int Q^p}{(p-3)\int Q^2}.
\end{equation}
Moreover, $E_V$ satisfies for all $t\in[t_0,t_1]$,
\begin{align}\label{323}
\|E_V(t)\|_{\h}\lesssim& \frac{1}{d^{9/2}(t)}+\frac{M^{3/2}_1(t)}{d^2(t)}+\frac{M_2(t)}{d^3(t)}+\frac{M_3(t)}{d^2(t)}.
\end{align}
where
\begin{align}
&M_1(t)=\sum_{i=1}^n|\mu_i(t)|^2,\quad M_2(t)=\sup_{1\leq i\leq n}|\dot{x}_i(t)-\mu_i(t)|,\\
&M_3(t)=\sup_{1\leq i\leq n}\bigg|\dot{\mu}_i(t)+\sum^n_{j=1,j\not=i}\frac{a_{ij}}{x^3_{ij}(t)}+\sum^n_{\substack{k,j=1,\\j\not=i}}\frac{b_{ijk}\mu_k(t)}{x^3_{ij}(t)}\bigg|.
\end{align}
\end{proposition}

\begin{remark}
We will see that under suitable assumptions the right hand side of \eqref{323} can be controlled by $1/d^{9/2}(t)$.
\end{remark}

We will see that the functions  $A_{ij}(t), B_{ij}(t)$ and constants $a_{ij}, b_{ijk}\in\mathbb{R}$ satisfy the following:
\begin{proposition}\label{P11}
Let $A_{ij}(t), B_{ij}(t), a_{ij}, b_{ijk}$ be the functions and constants introduced in Proposition \ref{P6}. Then the following properties hold.
\begin{enumerate}
\item Equation of $A_{ij}$: for all $i\not=j$, there exists even functions $A_{ij,0}\in\mathcal{Y}_2$ and smooth maps 
$$\mathfrak{A}_{ij}:\quad (\vec{x},\vec{\mu},y)\mapsto  \mathfrak{A}_{ij}(\vec{x},\vec{\mu},y)\in\mathbb{R}$$ satisfying:
\begin{equation}
\label{333}
A_{ij}(t,y)=A_{ij,0}(y)+\mathfrak{A}_{ij}(\vec{x}(t),\vec{\mu}(t),y).
\end{equation}
with
\begin{align}
\mathcal{L}A_{ij,0}=p\kappa_0\sigma_jQ^{p-1},\label{31}
\end{align}

Moreover, the smooth maps $\mathfrak{A}_{ij}$ satisfy
\begin{align}
\mathcal{L}\mathfrak{A}_{ij}=\sum^n_{k=1}\mu_kE_{ijk}+F_{ij},\label{355}
\end{align}
where $E_{ijk}\in\mathcal{Y}_2$ are even functions independent of time, while $F_{ij}$ are smooth maps 
$$F_{ij}:\quad \mathbb{R}^n\times\mathbb{R}^n\times\mathbb{R}\ni(\vec{x},\vec{\mu},y)\mapsto  F_{ij}(\vec{x},\vec{\mu},y)\in\mathbb{R}$$ satisfying
\begin{align}
&F_{ij}(\vec{x}(t),\vec{\mu}(t),\cdot)\in\mathcal{Y}_2,\quad\big(F_{ij}(\vec{x}(t),\vec{\mu}(t),\cdot),Q'(\cdot)\big)=0,\label{337}\\
&\|F_{ij}(\vec{x}(t),\vec{\mu}(t),\cdot)\|_{\h}\leq C_0\bigg(\frac{1}{d^{3/2}(t)}+M_1(t)\bigg),\label{360}\\
&\big\|\nabla_{\vec{\mu}}[F_{ij}(\vec{x}(t),\vec{\mu}(t),\cdot)]\big\|_{\h}+d^3(t)\big\|\nabla_{\vec{x}}[F_{ij}(\vec{x}(t),\vec{\mu}(t),\cdot)]\big\|_{\h}\leq C_0,\label{338}
\end{align}
for some universal constant $C_0>0$.

\item Equation of $B_{ij}$: for all $i,j,k=1,\ldots,n$ and $i\not=j$, there exist smooth bounded functions $B_{ij,0}$ and $\mathfrak{B}_{ijk}$ (independent of time) such that
\begin{equation}
\label{33}
B_{ij}(t,y)=B_{ij,0}(y)+\sum_{k=1}^n\mu_k(t)\mathfrak{B}_{ijk}(y).
\end{equation}
Here the functions $B_{ij,0}$ satisfy
\begin{align}
&(-\mathcal{L}B_{ij,0})'=-a_{ij}\sigma_i\Lambda Q+2p\kappa_0\sigma_j(yQ^{p-1})',\label{32}\\
&\int B_{ij,0}Q'=0,\quad\lim_{y\rightarrow+\infty}B_{ij,0}(y)=0,\label{328}\\
&\lim_{y\rightarrow-\infty}B_{ij,0}(y)=-\frac{a_{ij}\sigma_i(p-2)}{p-1}\int Q\label{313},
\end{align}
where $a_{ij}$ are given by \eqref{320}. Moreover, we have
\begin{align}
&\partial_y(\mathcal{L}\mathfrak{B}_{ijk})=-b_{ijk}\sigma_i\Lambda Q+G_{ijk},\label{336}
\end{align}
where $G_{ijk}\in\mathcal{Y}_2$ are functions independent of time.
\end{enumerate}
\end{proposition}

Before proving Proposition \ref{P6}, we need the following lemmas:
\begin{lemma}\label{L5}
Let $g\in\mathcal{Y}_2$ such that $(g,Q)=0$, then there exists a bounded smooth function $f$ such that $f'\in\mathcal{Y}_2$, $f\perp Q'$ and
\begin{equation}
\label{321}
(\mathcal{L}f)'=g,\quad \lim_{y\rightarrow+\infty}f(y)=0,\quad \lim_{y\rightarrow-\infty}f(y)=-\int g.
\end{equation}
\end{lemma}
\begin{proof}
We are seeking for $f$ of the following form: $f=f_0-\int_y^{+\infty}g$ where $f_0\in\mathcal{Y}_1$. By direct computation, we know that $f_0$ must satisfy
$$(\mathcal{L}f_0)'=g+\bigg(\mathcal{L}\int_y^{+\infty}g(y')\,\dd y'\bigg)':=R',$$
where $R=-\mathcal{H}g-pQ^{p-1}\int_y^{+\infty}g$ and $\mathcal{H}=-\partial_y|D|^{-1}$ denotes the Hilbert transform. From \cite[Lemma 2.5]{MP}, we know that $\mathcal{H}g\in\mathcal{Y}_1$, hence $R\in\mathcal{Y}_1$. From Proposition \ref{P1}, the existence of $f_0\in\mathcal{Y}_1$ is equivalent to $(R,Q')=0$. This is a direct consequence of $(g,Q)=0$ and $\mathcal{L}Q'=0$.
\end{proof}
\begin{lemma}\label{L6}
Let $f,g\in\mathcal{Y}_2$ and $|x_i-x_j|\gg1$. Then there exists constant $C_0=C_0(f,g)$ independent of $x_i,x_j$ such that
$$\|f(\cdot-x_i)g(\cdot-x_j)\|_{\h}\leq \frac{C_0}{|x_i-x_j|^2}.$$
\end{lemma}
\begin{proof}
Without loss of generality, we assume that $x_i>x_j$. We only need to prove
$$\|f(\cdot-x_i)g(\cdot-x_j)\|_{L^2}\leq \frac{C_0}{|x_i-x_j|^2}.$$
But this can be done by splitting the integral on regions $y<(x_i+x_j)/2$ and $y>(x_i+x_j)/2$ and using the fact that $f,g\in L^2$.
\end{proof}

We also need the following notations:
\begin{enumerate}
\item for any function $f$, we denote by
\begin{equation}\label{339}
\tau_if(t,y)=f(t,y-x_i(t));
\end{equation}
\item we denote by
\begin{align}
\label{342}
r(t,y)=&\sum_{\substack{i,j=1\\ j\not=i}}^n\frac{A_{ij}(t,y-x_{i}(t))}{x^2_{ij}(t)}+\sum_{\substack{i,j=1\\ j\not=i}}^n\frac{B_{ij}(t,y-x_{i}(t))}{x_{ij}^3(t)}\varphi_{ij}(t,y),
\end{align}
the interaction term in $V(t,y)$;
\item for all $C^1$ functions $(\vec{x}(t),\vec{\mu}(t))$ satisfying \eqref{36} we denote by $S$ the set consisting of all smooth maps $F$: $\mathbb{R}^n\times\mathbb{R}^n\times\mathbb{R}\mapsto\mathbb{R}$ 
\begin{align}
&F(\vec{x}(t),\vec{\mu}(t),\cdot)\in\mathcal{Y}_2,\quad\big(F(\vec{x}(t),\vec{\mu}(t),\cdot),Q'(\cdot)\big)=0,\label{340}\\
&\|F(\vec{x}(t),\vec{\mu}(t),\cdot)\|_{\h}\leq C_0\bigg(\frac{1}{d^{3/2}(t)}+M_1(t)\bigg),\label{341}\\
&\big\|\nabla_{\vec{\mu}}[F(\vec{x}(t),\vec{\mu}(t),\cdot)]\big\|_{\h}+d^3(t)\big\|\nabla_{\vec{x}}[F(\vec{x}(t),\vec{\mu}(t),\cdot)]\big\|_{\h}\leq C_0,\label{329}
\end{align}
for some universal constant $C_0>0$ (independent of $\vec{x}$ and $\vec{\mu}$).
\item for all $t\in I_0$, we denote by
\begin{align}\label{311}
&\Gamma(t)=\frac{1}{d^{9/2}(t)}+\frac{M^{3/2}_1(t)}{d^2(t)}+\frac{M_2(t)}{d^3(t)}+\frac{M_3(t)}{d^2(t)}+\sup_{i\not=j}\bigg(\frac{\|\mathfrak{A}_{ij}\|_{\h}}{d^4(t)}+\frac{\|\mathfrak{A}_{ij}\|_{\h}^2}{d^3(t)}\bigg)\nonumber\\
&\quad+\sup_{\substack{i,j,k=1\\i\not=j}}\bigg(\frac{M_1(t)\|\mathfrak{B}_{ijk}\|_{L^\infty}}{d^{3}(t)}+\frac{\|\mathfrak{B}_{ijk}\|_{L^\infty}+\|\mathfrak{B}_{ijk}\|_{L^\infty}^2}{d^5(t)}\bigg)\nonumber\\
&\quad+\sup_{i\not=j}\bigg(\|\nabla_{\vec{x}}\mathfrak{A}_{ij}\|_{\h}\frac{1+\sqrt{M_1(t)}d(t)}{d^3(t)}\bigg)+\sup_{i\not=j}\bigg(\|\nabla_{\vec{\mu}}\mathfrak{A}_{ij}\|_{\h}\frac{1+M_2(t)d^3(t)}{d^5(t)}\bigg).
\end{align}
\end{enumerate}

Now we can start the proof of Proposition \ref{P6} and Proposition \ref{P11}. The idea of the proof is to expand $\Psi_V$ and keep track of terms of order $1/x_{ij}^k$, for $k=2,3,4$,  and finally find suitable $A_{ij}(t)$, $B_{ij}(t)$ to cancel all lower order terms. We \emph{a priorily} assume that $A_{ij}(t)$, $B_{ij}(t)$, has the form \eqref{333} and \eqref{33}. More precisely, we first choose $A_{ij,0}$, $B_{ij,0}$ such that \eqref{31} and \eqref{32}--\eqref{313} are satisfied. Then we choose $\mathfrak{A}_{ij}$ and $\mathfrak{B}_{ij}$ suitably such that all lower order terms are canceled. We will see that lower order terms can be divided into the following categories:
\begin{itemize}
\item Simple terms that can be written down explicitly, for example:
$$\frac{\partial_y[\tau_i(Q^{p-1})]}{x^2_{ij}},\quad \frac{\partial_y[\tau_i(yQ^{p-1})]}{x^3_{ij}},$$
which are canceled due to the choice of $A_{ij,0}$, $B_{ij,0}$.
\item Terms of the form 
$$\frac{\partial_y[\tau_i(f_{ij})]}{x^2_{ij}},\quad \frac{\mu_k\partial_y[\tau_i(e_{ijk})]}{x^2_{ij}},$$
where $f_{ij}\in S$ and $e_{ijk}\in\mathcal{Y}_2$ are even functions independent of time. Those terms are canceled due to the choice of $\mathfrak{A}_{ij}$.
\item  Terms of the form 
$$\frac{\mu_k\tau_i(g_{ijk})}{x^3_{ij}}, $$ 
where $g_{ijk}\in\mathcal{Y}_2$ are functions independent of time. Those terms will be canceled by the choice of $\mathfrak{B}_{ijk}$.
\end{itemize}

\begin{proof}[Proof of Proposition \ref{P6} and \ref{P11}]

We divide the proof into several steps.

\noindent\textbf{Step 1:} Expansion of $\Psi_V$.

From the fact that:
$$|D|\R_i+(1+\mu_i)\R_i-\R_i^p=0,$$
we have the following expansion of $\Psi_V$: 
\begin{align}
\Psi_V=&\sum_{i=1}^n\big[\dot{\mu}_i\sigma_i\Lambda \R_i-(\dot{x}_i-\mu_i)\sigma_i\partial_y\R_i\big]+I+J+K+L,\label{343}
%&+\nonumber\\
%&+I(r)
\end{align}
where
\begin{align}
&I=\partial_y\Bigg[\bigg|\sum_{i=1}^n\sigma_i\R_i\bigg|^{p-1}\bigg(\sum_{i=1}^n\sigma_i\R_i\bigg)-\sum_{i=1}^n\sigma_i|\R_i|^{p-1}\R_i\Bigg],\label{345}\\
&J=-\partial_y\Bigg[|D|r+r-p\bigg|\sum_{i=1}^n\sigma_i\R_i\bigg|^{p-1}r\Bigg],\label{344}\\
&K=\partial_y\Bigg[\bigg|\sum_{i=1}^n\sigma_i\R_i+r\bigg|^{p-1}\bigg(\sum_{i=1}^n\sigma_i\R_i+r\bigg)\nonumber\\
&\qquad\qquad-\sum_{i=1}^n\sigma_i|\R_i|^{p-1}\R_i-p\bigg|\sum_{i=1}^n\sigma_i\R_i\bigg|^{p-1}r\Bigg],\label{346}\\
&L=\partial_tr=L_A+L_B,\label{356}.
\end{align}
Here $L_A$ is defined as follows:
\begin{align}
L_A(t,y)=&-\sum_{\substack{i,j=1\\ j\not=i}}^n\bigg(\frac{\dot{x}_i(t)\partial_yA_{ij}(t,y-x_i(t))}{x^2_{ij}(t)}+\frac{2\dot{x}_{ij}(t)A_{ij}(t,y-x_{i}(t))}{x^3_{ij}(t)}\bigg),\nonumber\\
&+\sum_{\substack{i,j=1\\ j\not=i}}^n\frac{(\partial_t\vec{x}(t),\partial_t\vec{\mu}(t))\cdot\nabla_{\vec{x},\vec{\mu}}\mathfrak{A}_{ij}(\vec{x}(t),\vec{\mu}(t),y-x_i(t))}{x^2_{ij}(t)},\label{347}
\end{align}
while $L_B$ is given by
\begin{align}
&L_B(t,y)=-\sum_{\substack{i,j=1\\ j\not=i}}^n\bigg(\frac{\dot{x}_i\partial_yB_{ij}(t,y-x_i(t))}{x^3_{ij}(t)}+\frac{3\dot{x}_{ij}B_{ij}(t,y-x_{i}(t))}{x^4_{ij}(t)}\bigg)\varphi_{ij}(t,y)\nonumber\\
&\qquad+\sum_{\substack{i,j,k=1\\ j\not=i}}^n\frac{\dot{\mu}_k\mathfrak{B}_{ijk}(y-x_i(t))}{x^3_{ij}(t)}\varphi_{ij}(t,y)+\sum_{\substack{i,j=1\\ j\not=i}}^n\frac{B_{ij}(t,y-x_{i}(t))}{x^3_{ij}(t)}\partial_t\varphi_{ij}(t,y),\label{348}
\end{align}
for all $(t,y)\in I\times \mathbb{R}$.

\noindent\textbf{Step 2:} Existence of $A_{ij,0}$, $B_{ij,0}$.

First, for $A_{ij,0}$, from Proposition \ref{P1} and the fact that $Q$ is an even function, we obtain the existence of $A_{ij,0}$ immediately.

Existence of $B_{ij,0}$ satisfying \eqref{32}--\eqref{313} follows from Lemma \ref{L5} directly. Here we need to choose suitable $a_{ij}\in\mathbb{R}$ so that the right hand side of \eqref{32} is orthogonal to $Q$, which coincides with \eqref{320}.

\noindent\textbf{Step 3:} Estimates of $I$. 

We now estimate the interaction term generated by different bubbles. We denote by
$$\R=\bigg|\sum_{j=1}^n\sigma_j\R_j\bigg|^{p-1}\bigg(\sum_{j=1}^n\sigma_j\R_j\bigg)%-\sum_{j=1}^n\sigma_j|\R_j|^{p-1}\R_j-p\sum_{\substack{j=1\\ j\not=k}}^n\R^{p-1}_j\sigma_k\R_k
.$$

From partition of unity and the assumption that $x_1(t)>\cdots>x_n(t)$, there exists smooth functions $\{\eta_{i}(t,y)\}_{i=0}^{n+1}$ such that
\begin{itemize}
\item $0\leq \eta_i\leq 1$, $\sum_{i=0}^{n+1}\eta_i\equiv1$;
\item $|\partial_y\eta_i(t)|\lesssim 1/d(t)$;
\item for all $t\in I_0$, there holds 
$$\begin{cases}
&\text{Supp }\eta_0(t)\subset\big(x(t)+d(t),+\infty\big);\\
&\text{Supp }\eta_1(t)\subset\big([2x_2(t)+x_1(t)]/3,x_1(t)+2d(t)\big);\\
&\text{Supp }\eta_i(t)\subset\big([2x_{i+1}(t)+x_i(t)]/3,[x_i(t)+2x_{i-1}(t)]/3\big),\; \forall 1<i<n;\\
&\text{Supp }\eta_n(t)\subset\big(x_n(t)-2d(t),[2x_{n-1}(t)+x_n(t)]/3\big);\\
&\text{Supp }\eta_{n+1}(t)\subset\big(-\infty,x_n(t)-d(t)\big).
\end{cases}$$
\end{itemize}
For all $i\in\{1,\ldots,n\}$, we consider $y\in \text{Supp }\eta_i(t)$, then using \eqref{21} and \eqref{233}, we have for all $j\not=i$
\begin{align}\label{361}
\tau_jQ(y)&=\frac{\kappa_0}{(y-x_j)^2}+\frac{g(y-x_j)}{(y-x_j)^4}+O\bigg(\frac{1}{1+(y-x_j)^6}\bigg)\nonumber\\
&=\frac{\kappa_0}{x^2_{ij}}\frac{1}{[1+(y-x_i)/x_{ij}]^2}+\frac{1}{x^4_{ij}}\frac{g(y-x_j)}{[1+(y-x_i)/x_{ij}]^3}+O\bigg(\frac{1}{1+(y-x_j)^6}\bigg)\nonumber\\
&=\frac{\kappa_0}{x_{ij}^2}-\frac{2\kappa_0(y-x_i)}{x_{ij}^3}+\frac{\kappa_0}{x^4_{ij}}\int_0^1\frac{(y-x_i)^2}{[1+s(y-x_i)/x_{ij}]^4}\,\dd s+\frac{g(y-x_j)}{x_{ij}^4}\nonumber\\
& \qquad+O\bigg(\frac{1+|y-x_i|}{|x_{ij}|^5}\bigg).
\end{align}
On the other hand, using the fact that $\R_i(y)>0$, we have
\begin{align}\label{363}
\partial_y(\R\eta_i)=&\partial_y\Bigg[\bigg|1+\sum_{\substack{j=1,\\ j\not=i}}^n\frac{\sigma_i\sigma_j\R_j}{\R_i}\bigg|^{p-1}\bigg(1+\sum_{\substack{j=1,\\ j\not=i}}^n\frac{\sigma_i\sigma_j\R_j}{\R_i}\bigg)\sigma_i\R_i^p\eta_i\Bigg]\nonumber\\
=&\partial_y\bigg(\sigma_i\R_i^p\eta_i+\sum_{\substack{j=1,\\j\not=i}}^np\sigma_j\R_i^{p-1}\R_j\eta_i\bigg)+C,
%
%&\quad+
%&+O_{\h}\bigg(\frac{1}{d^{5}(t)}\bigg)\nonumber\\
%&\eta_i\
\end{align}
where
\begin{align*}
C=&\partial_y\Bigg[\eta_i\int_0^1\frac{p(p-1)}{2}\bigg|1+s\sum_{\substack{j=1,\\ j\not=i}}^n\frac{\sigma_i\sigma_j\R_j}{\R_i}\bigg|^{p-3}\bigg(1+s\sum_{\substack{j=1,\\ j\not=i}}^n\frac{\sigma_i\sigma_j\R_j}{\R_i}\bigg)\times\sigma_i\R_i^p\bigg(\sum_{\substack{j=1,\\ j\not=i}}^n\frac{\sigma_i\sigma_j\R_j}{\R_i}\bigg)^2\,\dd s\Bigg].
\end{align*}
Recall from Taylor's formula, we have for all regular enough function $f$, there holds
\begin{align}\label{370}
f_{1+\mu_j}&=f+\int_0^1\frac{\mu_j}{1+s\mu_j}(\Lambda f)_{1+s\mu_j}=f+\mu_j(\Lambda f)+\frac{\mu_j^2}{2}\int_0^1\frac{\big[(\Lambda^2 f)_{1+s\mu_j}-(\Lambda f)_{1+s\mu_j}\big]}{(1+s\mu_j)^2}\,\dd s.
\end{align}
Since for $y\in \text{Supp }\eta_i$, we have $|\R_j(y)|\lesssim \frac{1}{d^2}$, combining with \eqref{361}, we have
\begin{align*}
&C=\partial_y\Bigg[\eta_i\frac{p(p-1)}{2}\sigma_i(\tau_iQ)^p\bigg(\sum_{\substack{j=1,\\ j\not=i}}^n\frac{\sigma_i\sigma_j\tau_jQ}{\tau_iQ}\bigg)^2\Bigg]+O_{\h}\bigg(\frac{1+\sqrt{M_1}d}{d^5}\bigg)\nonumber\\
&=\partial_y\bigg(\eta_i\frac{p(p-1)}{2}\sigma_i(\tau_iQ)^{p-2}\sum_{\substack{k,j=1,\\ k,j\not=i}}^n\frac{\kappa_0^2\sigma_k\sigma_j}{x^2_{ij}x^2_{ik}}\bigg)+O_{\h}\bigg(\frac{1+\sqrt{M_1}d}{d^5}\bigg)\nonumber\\
&=\sum_{\substack{j=1,\\i\not=j}}^n\frac{\partial_y(\tau_iH_{ij,1})}{x_{ij}^2}+O_{\h}\bigg(\frac{1+\sqrt{M_1}d}{d^5} \bigg),
\end{align*}
where
 $$H_{ij,1}=\frac{\kappa_0^2p(p-1)\sigma_i\sigma_jQ^{p-2}\tau_i^{-1}\eta_i}{2}\sum_{\substack{k=1,\\ k\not=i}}^n\frac{\sigma_k}{x^2_{ik}}.$$
Let $f_{ij,1}=H_{ij,1}-\frac{(H_{ij,1},Q')}{(Q',Q')}Q'$ be the orthogonal projection of $H_{ij,1}$ onto $\{Q'\}^\perp$. It is easy to verify that%
\footnote{Recall that the set $S$ is defined by \eqref{340}--\eqref{329}.}
 $f_{ij,1}\in S$ using the definition of $\eta_i$. We also have:
 \begin{align}\label{371}
|(H_{ij,1},Q')|&\lesssim \frac{1}{d^2}\big|(\tau_i^{-1}\eta_i,Q'Q^{p-2})\big|\lesssim \frac{1}{d^2}\big|(\tau_i^{-1}\eta'_i,Q^{p-1})\big|\lesssim \frac{1}{d^5}.
\end{align}
Hence, 
\begin{align}\label{372}
C=\sum_{\substack{i,j=1,\\i\not=j}}^n\frac{\partial_y(\tau_if_{ij,1})}{x_{ij}^2}+O_{\h}\bigg(\frac{1+\sqrt{M_1}d}{d^5} \bigg),
\end{align}
for some $f_{ij,1}\in S$.

We also have for all $j\not=i$,
\begin{align}
&\partial_y(\R_{j}^p\eta_i)=O_{\h}\bigg(\frac{1}{|x_{ij}|^{2p+\frac{1}{2}}}\bigg)=O_{\h}\bigg(\frac{1}{d^{9/2}}\bigg),\label{362}\\
&\sum_{\substack{j=1,\\j\not=i}}^n\partial_y\Big[\eta_i\big(p\sigma_j\R_i^{p-1}\R_j-p\sigma_j(\tau_iQ)^{p-1}\tau_jQ\big)\Big]\nonumber\\
&=\sum_{\substack{j=1,\\j\not=i}}^np\sigma_j\partial_y\bigg(\eta_i\int_0^1\frac{\mu_j\tau_j(\Lambda Q)_{1+s\mu_j}(\tau_iQ)^{p-1}_{1+s\mu_i}}{1+s\mu_j}+\frac{(p-1)\mu_i\tau_i(\Lambda Q)_{1+s\mu_i}(\tau_iQ)^{p-2}_{1+s\mu_i}(\tau_jQ)_{1+s\mu_j}}{1+s\mu_i}\,\dd s\bigg).
\end{align}
Together with \eqref{361}, we have
\begin{align}\label{373}
&\sum_{\substack{j=1,\\j\not=i}}^n\partial_y\big(p\sigma_j\R_i^{p-1}\R_j-p\sigma_j(\tau_iQ)^{p-1}\tau_jQ\big)=\sum_{\substack{j,k=1\\j\not=i}}^n\frac{\mu_k\partial_y(\tau_iE_{ijk,1})}{x^2_{ij}}\nonumber\\
&\qquad\quad+\sum_{ \substack{j=1\\ j\not=i}}^n\frac{\partial_y(\tau_if_{ij,2})}{x_{ij}^2}+\sum_{\substack{j,k=1,\\i\not=j}}^n\frac{\mu_k\tau_ig_{ijk,1}}{x_{ij}^3}+O_{\h}\bigg(\frac{1+M_1d^2}{d^5}\bigg),
\end{align}
where $f_{ij,2}\in S$, $g_{ijk,1},E_{ijk,1}\in\mathcal{Y}_2$ and $E_{ijk,1}$ are even functions.

Hence, combining \eqref{361}, \eqref{372} and \eqref{373}, we have for all $j\not=i$, and for all $y\in \text{Supp }\eta_i$,
\begin{align*}
&\partial_y(\R\eta_i)=\partial_y\bigg(\sigma_i\R_i^p\eta_i+\sum_{\substack{j=1,\\j\not=i}}^np\sigma_j(\tau_iQ)^{p-1}\tau_jQ\eta_i\bigg)+\sum_{\substack{j,k=1,\\i\not=j}}^n\frac{\mu_k\partial_y(\tau_iE_{ijk,1})}{x^2_{ij}}\nonumber\\
&\qquad+\sum_{ \substack{j=1\\ j\not=i}}^n\frac{\partial_y[\tau_i(f_{ij,1}+f_{ij,2})]}{x_{ij}^2}+O_{\h}\bigg(\frac{1+M_1d^2}{d^5}\bigg)\nonumber\\
&=\partial_y\bigg(\eta_i\sum_{j=1}^n\sigma_j\R_j^p\bigg)+\partial_y\Bigg[\sum_{\substack{j=1,\\j\not=i}}^np\sigma_j\bigg(\frac{\kappa_0}{x_{ij}^2}-\frac{2\kappa_0(y-x_i)}{x_{ij}^3}\bigg)(\tau_iQ)^{p-1}\Bigg]\nonumber\\
&\qquad+\partial_y\Bigg[\eta_i\sum_{\substack{j=1,\\j\not=i}}^np\sigma_j\bigg(\frac{\tau_jg}{x^4_{ij}}+\frac{\kappa_0}{x^4_{ij}}\int_0^1\frac{(y-x_i)^2}{[1+s(y-x_i)/x_{ij}]^4}\,\dd s\bigg)(\tau_iQ)^{p-1}\Bigg]\nonumber\\
&\qquad+\sum_{\substack{j,k=1,\\i\not=j}}^n\frac{\mu_k\partial_y(\tau_iE_{ijk,1})}{x^2_{ij}}+\sum_{ \substack{j=1\\ j\not=i}}^n\frac{\partial_y[\tau_i(f_{ij,1}+f_{ij,2})]}{x_{ij}^2}+O_{\h}\bigg(\frac{1+M_1d^2}{d^5}\bigg),
\end{align*}
where $f_{ij,1}, f_{ij,2}\in S$ and $E_{ijk,1}\in\mathcal{Y}_2$ are even functions.

Let
$$H_{ij,2}=(\tau_i^{-1}\eta_{i})p\sigma_j\bigg(\frac{\tau_i^{-1}\tau_jg}{x^2_{ij}}+\frac{\kappa_0}{x^2_{ij}}\int_0^1\frac{y^2}{[1+(sy)/x_{ij}]^4}\,\dd s\bigg)Q^{p-1},$$
and $f_{ij,3}=H_{ij,2}-\frac{(H_{ij,2},Q')}{(Q',Q')}Q'$ be the orthogonal projection of $H_{ij,2}$ onto $\{Q'\}^\perp$. We can easily check that $f_{ij,3}\in S$ and%
\footnote{Here we use \eqref{234} and the fact that $\text{Supp }\partial_y(\tau_i^{-1}\eta_{i})\subset\{y:\,|y|\gtrsim \frac{1}{d}\}$ for the last inequality.}
\begin{align*}
&|(H_{ij,2},Q')|\lesssim \frac{1}{d^2}\bigg|\bigg((\tau_i^{-1}\eta_i)\int_0^1\frac{y^2}{[1+(sy)/x_{ij}]^4}\,\dd s+\tau_i^{-1}[\eta_i(\tau_jg)],Q'Q^{p-1}\bigg)\bigg|\\
&\lesssim\frac{1}{d^2}\big|\big((\partial_y\eta_i)(\tau_jg),\tau_i(Q^p)\big)\big|+\frac{1}{d^2}\big|\big(\eta_i\partial_y(\tau_jg),\tau_i(Q^p)\big)\big|\\
&\qquad+ \frac{1}{d^2}\Bigg|\Bigg(\tau_i^{-1}\eta_i\int_0^1\bigg[y^2+\frac{1}{x_{ij}}\int_0^1\frac{sy^3}{[1+(sty)/x_{ij}]^4}\,\dd t\bigg]\,\dd s,\partial_y(Q^p)\Bigg)\Bigg|=O\bigg(\frac{1}{d^3}\bigg).
\end{align*}
Combining all estimates above, we have%
\footnote{Recall that $\Gamma$ is defined by \eqref{311}.}
\begin{align}\label{326}
I=&\partial_y\Bigg[\sum_{\substack{i,j=1,\\j\not=i}}^np\kappa_0\sigma_j\bigg(\frac{1}{x_{ij}^2}-\frac{2(y-x_i)}{x_{ij}^3}\bigg)(\tau_iQ)^{p-1}\Bigg]+\sum_{\substack{j,k=1,\\i\not=j}}^n\frac{\mu_k\partial_y(\tau_iE_{ijk,1})}{x^2_{ij}}\nonumber\\
&\qquad+\sum_{\substack{i,j=1\\ j\not=i}}^n\frac{\partial_y(\tau_iF_{ij,1})}{x^2_{ij}}+\sum_{\substack{j,k=1,\\i\not=j}}^n\frac{\mu_k\tau_ig_{ijk,1}}{x_{ij}^3}++O_{\h}(\Gamma),
\end{align}
where $E_{ijk,1}\in\mathcal{Y}_2$ are even functions and $F_{ij,1}=f_{ij,1}+f_{ij,2}+f_{ij,3}\in S$, $g_{ijk,1}\in\mathcal{Y}_2$.

\noindent\textbf{Step 4:} Estimates of $J$. 

By direct computation, we have 
$$J=J_1+J_2+J_3+J_4+J_5,$$
where%
\footnote{Here we use the notation $[|D|,\varphi_{ij}]f=|D|(\varphi_{ij}f)-\varphi_{ij}|D|f$, for suitable $f$.}
\begin{align}
&J_1=-\partial_y\bigg(\sum_{\substack{i,j=1,\\i\not=j}}^n\frac{\tau_i(\mathcal{L}A_{ij,0})}{x_{ij}^2}+\sum_{\substack{i,j=1,\\i\not=j}}^n\frac{\tau_i(\mathcal{L}B_{ij,0})}{x_{ij}^3}\bigg),\label{350}\\
&J_2=-\partial_y\bigg(\sum_{\substack{i,j=1,\\i\not=j}}^n\frac{\tau_i(\mathcal{L}\mathfrak{A}_{ij})}{x_{ij}^2}+\sum_{\substack{i,j,k=1\\i\not=j}}^n\frac{\mu_{k}\tau_i(\mathcal{L}\mathfrak{B}_{ijk})}{x_{ij}^3}\bigg),\label{357}\\
&J_3=p\partial_y\bigg(\bigg|\sum_{i=1}^n\sigma_i\R_i\bigg|^{p-1}r-\sum_{i=1}^n\R_i^{p-1}r+\sum_{i=1}^n\R_i^{p-1}r-\sum_{i=1}^n(\tau_i Q)^{p-1}r\bigg),\label{358}\\
&J_4=p\partial_y\Bigg[\sum_{\substack{i,j,k=1,\\j,k\not=i}}^n(\tau_kQ)^{p-1}\tau_i\bigg(\frac{A_{ij,0}}{x^2_{ij}}+\frac{B_{ij,0}}{x^3_{ij}}\bigg)\Bigg],\label{374}\\
&J_5=-\sum_{\substack{i,j=1,\\i\not=j}}^n\partial_y\Bigg[\bigg(\frac{\tau_i(\mathcal{L}B_{ij})}{x_{ij}^3}\bigg)(\varphi_{ij}-1)\Bigg]-\sum_{\substack{i,j=1,\\i\not=j}}^n\partial_y\bigg(\frac{\tau_i([|D|,\varphi_{ij}]B_{ij})}{x_{ij}^3}\bigg),\label{359}
\end{align}
We will leave $J_2$ invariant and estimate $J_1$, $J_3$, $J_4$ and $J_5$ separately.

First for $J_1$, using the equation of  $A_{ij,0}$, $B_{ij,0}$, $C_{ij}$, we have
\begin{align}\label{351}
J_1=&-\sum_{\substack{i,j=1,\\i\not=j}}^n\Bigg[\partial_y\bigg(\frac{p\sigma_j (\tau_iQ)^{p-1}}{x_{ij}^2}-\frac{2p\kappa_0\sigma_j \partial_y[\tau_i(yQ^{p-1})]}{x_{ij}^3}\bigg)+\frac{a_{ij}\sigma_i\tau_i(\Lambda Q)}{x_{ij}^3}\Bigg].
\end{align}
Using \eqref{370} again, we have
%\begin{align}\label{352}
%&\sum_{\substack{i,j=1,\\i\not=j}}^n\frac{a_{ij}\sigma_i\partial_y(\tau_iQ-\R_i)}{x^2_{ij}}\nonumber\\
%&=\sum_{\substack{i,j,k=1,\\i\not=j}}^n\frac{\mu_k\partial_y(\tau_iE_{ijk,2})}{x_{ij}^2}+\sum_{\substack{i,j=1,\\i\not=j}}^n\frac{\partial_y(\tau_if_{ij,4})}{x^2_{ij}}+O_{\h}\bigg(\sum_{\substack{i,j,k=1\\i\not=j}}^n\frac{|{\mu}_k|^3}{|x_{ij}|^2}\bigg)\nonumber\\
%&=\sum_{\substack{i,j,k=1,\\i\not=j}}^n\frac{\mu_k\partial_y(\tau_iE_{ijk,2})}{x_{ij}^2}+\sum_{\substack{i,j=1,\\i\not=j}}^n\frac{\partial_y(\tau_if_{ij,4})}{x^2_{ij}}+O_{\h}\bigg(\frac{M^{3/2}_1}{d^2}\bigg),
%\end{align}
%where $f_{ij,4}\in S$ and $E_{ijk,2}\in \mathcal{Y}_2$ are even functions.
%On the other hand, by a similar argument, 
we have 
\begin{align}
\label{353}
&\sum_{\substack{i,j=1,\\i\not=j}}^n\frac{a_{ij}\sigma_i[\tau_i\Lambda Q-(\Lambda\R_i)/(1+\mu_i)]}{x^3_{ij}}=\sum_{\substack{i,j,k=1\\i\not=j}}^n\frac{\mu_k\tau_ig_{ijk,2}}{x^3_{ij}}+O_{\h}\bigg(\frac{M_1}{d^3}\bigg),
\end{align}
where $g_{ijk,2}\in\mathcal{Y}_2$. Combining \eqref{351}--\eqref{353}, we conclude that
\begin{align}\label{327}
J_1=&\sum_{\substack{i,j=1,\\i\not=j}}^n\frac{a_{ij}\sigma_i\Lambda\R_i}{x^3_{ij}}-\partial_y\Bigg[\sum_{\substack{i,j=1,\\i\not=j}}^n\bigg(\frac{p\kappa_0\sigma_j (\tau_iQ)^{p-1}}{x_{ij}^2}-\frac{2p\kappa_0\sigma_j \tau_i(yQ^{p-1})}{x_{ij}^3}\bigg)\Bigg]\nonumber\\
&+\sum_{\substack{i,j,k=1\\i\not=j}}^n\frac{\mu_k\tau_ig_{ijk,2}}{x^3_{ij}}+O_{\h}(\Gamma).
\end{align}

Then we estimate $J_3$. Using again the Taylor's expansion of $\R_i$ with respect to $\mu_i$ as well as the methods of splitting region of the integral,  we have
\begin{align*}
&\bigg|\sum_{i=1}^n\R_i^{p-1}-\sum_{i=1}^n(\tau_i Q)^{p-1}\bigg|=O_{\h}(|\vec{\mu}|)=O_{\h}\big(\sqrt{M_1}\big)\\
&\Bigg|\bigg|\sum_{i=1}^n\sigma_i\R_i\bigg|^{p-1}-\sum_{i=1}^n\R_i^{p-1}\Bigg|=O_{\h}\bigg(\frac{1}{d^2}\bigg).
\end{align*}
Using the fact that $A_{ij,0}$ are even functions, we have
\begin{align}\label{354}
&\partial_y\bigg(\sum_{k=1}^n\big(\R_k^{p-1}-(\tau_k Q)^{p-1}\big)\sum_{\substack{i,j=1\\i\not=j}}^n\frac{\tau_iA_{ij,0}}{x^2_{ij}}\bigg)=\sum_{\substack{i,j=1,\\i\not=j}}^n\frac{\mu_k\partial_y(\tau_iE_{ijk,2})}{x_{ij}^2}+O_{\h}(\Gamma),
\end{align}
where $E_{ijk,2}\in \mathcal{Y}_2$ are even functions. We also have
\begin{align}\label{364}
&\partial_y\Bigg[\sum_{k=1}^n\big(\R_k^{p-1}-(\tau_k Q)^{p-1}\big)\times\sum_{\substack{i,j=1\\i\not=j}}^n\frac{\tau_iB_{ij,0}}{x^3_{ij}}\varphi_{ij}\Bigg]=\sum_{\substack{i,j,k=1\\i\not=j}}^n\frac{\mu_k\tau_ig_{ijk,3}}{x^3_{ij}}+O_{\h}\bigg(\frac{M_1}{d^3}\bigg).
\end{align}
where $g_{ijk,3}\in\mathcal{Y}_2$ are functions independent of time. On the other hand, from \eqref{370} we know that
\begin{align}\label{365}
&\partial_y\Bigg[\sum_{k=1}^n\big(\R_k^{p-1}-(\tau_k Q)^{p-1}\big)\times\bigg(r-\sum_{\substack{i,j=1\\i\not=j}}^n\frac{\tau_iA_{ij,0}}{x^2_{ij}}-\sum_{\substack{i,j=1\\i\not=j}}^n\frac{\tau_iB_{ij,0}}{x^3_{ij}}\varphi_{ij}\bigg)\Bigg]\nonumber\\
&=\sum_{\substack{i,j,k=1\\i\not=j}}^n\frac{\mu_k\partial_y[\tau_i((p-1)\Lambda Q Q^{p-2}\mathfrak{A}_{ij})]}{x_{ij}^2}+O_{\h}(\Gamma).
\end{align}
Now we estimate the term 
$$p\partial_y\bigg(\bigg|\sum_{i=1}^n\sigma_i\R_i\bigg|^{p-1}r-\sum_{i=1}^n\R_i^{p-1}r\bigg).$$
We consider the smooth functions $\{\eta_i\}_{i=0}^{n+1}$ again to obtain
\begin{align}\label{366}
&\partial_y\Bigg[\bigg(\bigg|\sum_{k=1}^n\sigma_k\R_k\bigg|^{p-1}-\sum_{k=1}^n\R_k^{p-1}\bigg)\eta_i\frac{\tau_i(A_{ij,0})}{x_{ij}^2}\Bigg]\nonumber\\
&=\partial_y\Bigg[\bigg(\sum^n_{\substack{k=1,\\ k\not=i}}\frac{(p-1)\kappa_0\sigma_k\tau_i(Q^{p-2}A_{ij,0})}{x^2_{ik}x^2_{ij}}\bigg)\eta_i\Bigg]+O(\Gamma)=\frac{\partial_y(\tau_if_{ij,4})}{x^2_{ij}}+O(\Gamma),
\end{align}
where $f_{ij,4}\in S$. Similar as \eqref{365}, we have:
\begin{align}\label{367}
\partial_y\Bigg[\bigg(\bigg|\sum_{k=1}^n\sigma_k\R_k\bigg|^{p-1}-\sum_{k=1}^n\R_k^{p-1}\bigg)\times\bigg(r-\sum_{\substack{i,j=1,\\i\not=j}}^n\frac{\tau_i(A_{ij,0})}{x_{ij}^2}\bigg)\Bigg]=O_{\h}(\Gamma).
\end{align}
Combining \eqref{354}--\eqref{367}, we have
\begin{align}\label{335}
J_3=&\sum_{\substack{i,j=1,\\i\not=j}}^n\frac{\partial_y(\tau_iF_{ij,2})}{x_{ij}^2}+\sum_{\substack{i,j,k=1,\\i\not=j}}^n\frac{\mu_k\partial_y(\tau_iE_{ijk,2})}{x_{ij}^2}+\sum_{\substack{i,j,k=1\\i\not=j}}^n\frac{\mu_k\tau_ig_{ijk,3}}{x^3_{ij}}\nonumber\\
&+\sum_{\substack{i,j,k=1\\i\not=j}}^n\frac{\mu_k\partial_y[\tau_i((p-1)\Lambda Q Q^{p-2}\mathfrak{A}_{ij})]}{x_{ij}^2}+O_{\h}(\Gamma),
\end{align}
where $F_{ij,2}=f_{ij,4}\in S$ and $E_{ijk,2}\in\mathcal{Y}_2$ are even functions.

Next, we estimate $J_4$. We let
$$H_{ij,3}=p\sum_{\substack{k=1,\\k\not=i}}^n(\tau_i^{-1}\tau_kQ)^{p-1}\bigg(A_{ij,0}+\frac{B_{ij,0}}{x_{ij}}\bigg).$$
Let $f_{ij,5}=H_{ij,3}-\frac{(H_{ij,3},Q')}{(Q',Q')}Q'$. By a similar argument, we have $f_{ij,5}\in S$. We claim that
\begin{equation}
\label{375}
|(H_{ij,3},Q')|\lesssim \frac{1}{d^3}.
\end{equation}
By Lemma \ref{L6} and direct computation, we have
$$\bigg|\bigg(Q',p\sum_{\substack{k=1,\\k\not=i}}^n(\tau_i^{-1}\tau_kQ)^{p-1}\frac{B_{ij,0}}{x_{ij}}\bigg)\bigg|\lesssim \frac{1}{d^3}.$$
While for $\big(Q',(\tau_i^{-1}\tau_kQ)^{p-1}A_{ij,0}\big)$, using the fact that $Q'\in\mathcal{Y}_3$ and $A_{ij,0}\in\mathcal{Y}_2$, we have
$$\int_{|y|>|1/(2x_i-x_k|)}Q'(\tau_i^{-1}\tau_kQ)^{p-1}A_{ij,0}=O\bigg(\frac{1}{d^3}\bigg).$$
From Proposition \ref{P1}, have
\begin{align*}
&\int_{|y|<\frac{1}{2|x_i-x_k|}}Q'(\tau_i^{-1}\tau_kQ)^{p-1}A_{ij,0}=\int_{|y|<\frac{1}{2|x_i-x_k|}}Q'A_{ij}\bigg(\frac{\kappa_0}{[y-(x_k-x_i)]^2}\bigg)^{p-1}+O\bigg(\frac{1}{d^3}\bigg).
\end{align*}
Together with \eqref{361}, we have
\begin{align*}
&\int_{|y|<\frac{1}{2|x_i-x_k|}}Q'(\tau_i^{-1}\tau_kQ)^{p-1}A_{ij,0}=\int_{|y|<\frac{1}{2|x_i-x_k|}}\bigg(\frac{\kappa_0}{(x_k-x_i)^2}\bigg)^{p-1}+O\bigg(\frac{1}{d^3}\bigg)\\
&=\int Q'A_{ij}\bigg(\frac{c_0}{(x_k-x_i)^2}\bigg)^{p-1}+O\bigg(\frac{1}{d^3}\bigg)=O\bigg(\frac{1}{d^3}\bigg),
\end{align*}
where we use the fact that both $A_{ij,0}$ and $Q$ are even functions in the last identity, which concludes the proof of \eqref{375}. Hence, we have
\begin{equation}\label{376}
J_4=\sum_{\substack{i,j=1,\\i\not=j}}^n\frac{\partial_y(\tau_iF_{ij,3})}{x_{ij}^2}+O(\Gamma),
\end{equation}
where $F_{ij,3}=f_{ij,5}\in S$.

Finally, we estimate $J_5$. We have
\begin{align*}
J_5=D-\sum_{\substack{i,j,k=1,\\i\not=j}}^n\frac{\mu_k\partial_y[\tau_i(\mathcal{L}\mathfrak{B}_{ijk})(\varphi_{ij}-1)]}{x^3_{ij}}-\sum_{\substack{i,j,k=1,\\i\not=j}}^n\partial_y\bigg(\frac{\mu_k\tau_i([|D|,\varphi_{ij}]\mathfrak{B}_{ijk})}{x_{ij}^3}\bigg),
%-\sum_{\substack{i,j=1,\\i\not=j}}^n\partial_y\bigg(\frac{\tau_i(LB_{ij_0})}{x_{ij}^3}+\frac{\tau_i(LC_{ij})}{x_{ij}^4}\bigg)(\varphi_{ij}-1)
\end{align*} 
where
\begin{align*}
D=-\sum_{\substack{i,j=1,\\i\not=j}}^n\partial_y\Bigg[\bigg(\frac{\tau_i(\mathcal{L}B_{ij,0})}{x_{ij}^3}\bigg)(\varphi_{ij}-1)\Bigg]-\sum_{\substack{i,j=1,\\i\not=j}}^n\partial_y\bigg(\frac{\tau_i([|D|,\varphi_{ij}]B_{ij,0})}{x_{ij}^3}\bigg).
\end{align*}
By direct computation, we have
\begin{align*}
&\|D\|_{\h}\leq \|D\|_{H^1}\lesssim \sup_{\substack{i\not=j,\\1\leq\ell\leq 3}}\frac{1}{d^3}\Big(\big\|(\partial^\ell_yB_{ij,0})(\varphi_{ij}-1)\big\|_{L^2}+\|B_{ij,0}\varphi'_{ij}\|_{L^2}\Big)\\
&+\sup_{i\not=j}\frac{1}{d^3}\big(\||D|(\varphi'_{ij}B_{ij,0})\|_{H^1}+\|\varphi'_{ij}(|D|B_{ij,0})\|_{H^1}\big)+\sup_{i\not=j}\frac{1}{d^3}\big\|[|D|,\varphi_{ij}](\partial_yB_{ij,0})\big\|_{H^1}+\frac{1}{d^5}.
\end{align*}
From the definition of $\varphi_{ij}$ as well as%
\footnote{Here we also use the fact that $|D|=\partial_y\mathcal{H}$, and $\|\mathcal{H}f\|_{L^p}\lesssim \|f\|_{L^p}$ for all $1<p<+\infty$, where $\mathcal{H}$ is the Hilbert transform.}
 \cite[Lemma 2.15]{MP}, we have
\begin{align*}
\|D\|_{H^1}\lesssim&\sup_{\substack{i\not=j,k\geq 1\\2\leq k+\ell\leq 3}}\frac{1}{d^3}\Big(\big\|(\partial_y^k\varphi_{ij})(\partial_y^\ell B_{ij,0})\big\|_{L^2}\Big)+\frac{1}{d^5}\lesssim\frac{1}{d^5}.
\end{align*}
Hence, we have
\begin{align}
\label{368}
J_5=&-\sum_{\substack{i,j,k=1,\\i\not=j}}^n\frac{\mu_k\partial_y\{\tau_i(\mathcal{L}\mathfrak{B}_{ijk})(\varphi_{ij}-1)+\tau_i([|D|,\varphi_{ij}]\mathfrak{B}_{ijk})\}}{x^3_{ij}}+O_{\h}\bigg(\frac{1}{d^5}\bigg).
\end{align}

\noindent\textbf{Step 5:} Estimates of $K$.

We let 
$$r_0(t,y)=\sum_{\substack{i,j=1,\\i\not=j}}^n\frac{A_{ij,0}(y-x_i)}{x^2_{ij}},\quad r_1(t,y)=r(t,y)-r_{0}(t,y).$$ 
Since $p>2$, we have
\begin{align}
\label{369}
K=&\partial_y\Bigg[\frac{p(p-1)}{2}r^2\int_0^1\bigg|\sum_{i=1}^n\sigma_i\R_i+sr\bigg|^{p-3}\bigg(\sum_{i=1}^n\sigma_i\R_i+sr\bigg)\dd s\Bigg]\nonumber\\
=&\partial_y\Bigg[\frac{p(p-1)}{2}r_0^2\int_0^1\bigg|\sum_{i=1}^n\sigma_i\R_i+sr\bigg|^{p-3}\bigg(\sum_{i=1}^n\sigma_i\R_i+sr\bigg)\dd s\Bigg]\nonumber\\
&+\partial_y\Bigg[\frac{p(p-1)}{2}(2r_0r_1+r_1^2)\int_0^1\bigg|\sum_{i=1}^n\sigma_i\R_i+sr\bigg|^{p-3}\bigg(\sum_{i=1}^n\sigma_i\R_i+sr\bigg)\dd s\Bigg]\nonumber\\
=&\partial_y\Bigg[\frac{p(p-1)}{2}r_0^2\int_0^1\bigg|\sum_{i=1}^n\sigma_i(\tau_iQ)+sr\bigg|^{p-3}\bigg(\sum_{i=1}^n\sigma_i(\tau_iQ)+sr\bigg)\dd s\Bigg]+O_{\h}(\Gamma).
\end{align}
Since 
$$r_0^2=\sum_{\substack{i,j,k,\ell=1,\\i\not=j,k\not=\ell}}^n\frac{(\tau_iA_{ij,0})(\tau_kA_{k\ell,0})}{x^2_{ij}x^2_{k\ell}},$$
by keeping track of all terms that are of order $1/d^4$, we see that 
\begin{align*}
&\partial_y\Bigg[r_0^2\int_0^1\bigg|\sum_{i=1}^n\sigma_i(\tau_iQ)+sr\bigg|^{p-3}\bigg(\sum_{i=1}^n\sigma_i(\tau_iQ)+sr\bigg)\dd s\Bigg]=\sum_{\substack{i,j=1,\\i\not=j}}^n\frac{\partial_y[(\tau_iH_{ij,4})]}{x^2_{ij}}+O_{\h}(\Gamma).
\end{align*}
where
$$H_{ij,4}=\frac{p(p-1)\sigma_iQ^{p-1}}{2}\sum_{\substack{k,\ell=1,\\k\not=\ell}}^n\frac{A_{ij,0}(\tau_i^{-1}\tau_kA_{k\ell,0})}{x^2_{k\ell}}.$$
A similar argument implies that $|(H_{ij,4},Q')|\lesssim 1/d^3$. Hence, by letting $F_{ij,4}=H_{ij,4}-\frac{(H_{ij,4}.Q')}{(Q',Q')}Q'$, we have
\begin{equation}\label{377}
K=\sum_{\substack{i,j=1,\\i\not=j}}^n\frac{\partial_y[(\tau_iF_{ij,4})]}{x^2_{ij}}+O_{\h}(\Gamma),
\end{equation}
where $F_{ij,4}\in S$.

\noindent\textbf{Step 6:} Estimates of $L$.

Recall that $L=L_A+L_B$, where $L_A$, $L_B$, are given by \eqref{347} and \eqref{348}. By direct computation, we have:
\begin{align}\label{324}
L_A=&-\sum_{\substack{i,j=1\\ j\not=i}}^n\bigg(\frac{\dot{x}_i\partial_y(\tau_iA_{ij})}{x^2_{ij}}+\frac{2\dot{x}_{ij}\tau_iA_{ij}}{x^3_{ij}}\bigg)+\sum_{\substack{i,j=1\\ j\not=i}}^n\frac{(\partial_t\vec{x},\partial_t\vec{\mu})\cdot\nabla_{\vec{x},\vec{\mu}}(\tau_i\mathfrak{A}_{ij})}{x^2_{ij}}\nonumber\\
=&-\sum_{\substack{i,j=1\\ j\not=i}}^n\bigg(\frac{{\mu}_i\partial_y\tau_iA_{ij,0}}{x^2_{ij}}+\frac{2{\mu}_{ij}A_{ij,0}}{x^3_{ij}}\bigg)+O_{\h}(\Gamma)\nonumber\\
=&\sum_{\substack{i,j,k=1,\\i\not=j}}^n\frac{\mu_k\partial_y(\tau_iE_{ijk,3})}{x_{ij}^2}+\sum_{\substack{i,j,k=1,\\i\not=j}}^n\frac{\mu_k\tau_ig_{ijk,4}}{x^3_{ij}}+O_{\h}(\Gamma),
\end{align}
where $E_{ijk,3}, g_{ijk,4}\in\mathcal{Y}_2$ and $E_{ijk,3}$ are even functions. Similarly, we have
\begin{align}\label{325}
L_B=\sum_{\substack{i,j,k=1\\i\not=j}}^n\frac{\mu_k\tau_ig_{ijk,5}}{x^3_{ij}}+O_{\h}(\Gamma),
\end{align}
where $g_{ijk,5}\in\mathcal{Y}_2$.

\noindent\textbf{Step 7:} Closing of the proof.
Now, we combine the estimates of \eqref{326}, \eqref{357}, \eqref{327}, \eqref{335}, \eqref{376}, \eqref{368}, \eqref{377}--\eqref{325}, and choose
$$E_{ijk}=\sum_{\ell=1}^3E_{ijk,\ell},\quad G_{ijk}=\sum_{\ell=1}^5g_{ijk,\ell}$$
to obtain that
\begin{align}\label{381}
\Psi_V=&-\sum_{i=1}^n(\dot{x}_i-\mu_i)\sigma_i\partial_y\widetilde{R}_i+\sum_{i=1}^n\bigg(\dot{\mu}_i+\sum^n_{j=1,j\not=i}\frac{a_{ij}}{x^3_{ij}}+\sum^n_{\substack{k,j=1,\\j\not=i}}\frac{b_{ijk}\mu_k}{x^3_{ij}}\bigg)\frac{\sigma_i\Lambda\widetilde{R}_i}{1+\mu_i}\nonumber\\
&+\sum_{\substack{i,j=1,\\i\not=j}}^n\bigg(-\frac{\partial_y[\tau_i(\mathcal{L}\mathfrak{A}_{ij})]}{x_{ij}^2}+\sum_{\ell=1}^4\frac{\partial_y(\tau_iF_{ij,\ell})}{x^2_{ij}}+\sum_{k=1}^n\frac{\mu_k\partial_y(\tau_iE_{ijk})}{x^2_{ij}}\bigg)\nonumber\\
&+\sum_{\substack{i,j,k=1,\\i\not=j}}^n\mu_k\bigg(-\frac{\partial_y[\tau_i(\mathcal{L}\mathfrak{B}_{ijk})]}{x_{ij}^3}+\frac{\tau_iG_{ijk}}{x^3_{ij}}-\frac{b_{ijk}\sigma_i}{x^3_{ij}}\tau_i(\Lambda Q)\bigg)\nonumber\\
&-\sum_{\substack{i,j,k=1,\\i\not=j}}^n\frac{\mu_k\partial_y\{\tau_i(\mathcal{L}\mathfrak{B}_{ijk})(\varphi_{ij}-1)+\tau_i([|D|,\varphi_{ij}]\mathfrak{B}_{ijk})\}}{x^3_{ij}}\nonumber\\
&+\sum_{\substack{i,j,k=1\\i\not=j}}^n\frac{\mu_k\partial_y[\tau_i((p-1)\Lambda Q Q^{p-2}\mathfrak{A}_{ij})]}{x_{ij}^2}+O_{\h}(\Gamma),
\end{align}
where $E_{ijk},G_{ijk}\in\mathcal{Y}_2$ and $E_{ijk}$ are even functions.

Then we choose bounded functions (with respect to $y$) $\mathfrak{B}_{ijk}$, such that
\begin{equation}\label{380}
\partial_y(\mathcal{L}\mathfrak{B}_{ijk})=G_{ijk}-b_{ijk}\sigma_i\Lambda Q.
\end{equation}
From Lemma \ref{L5}, we know that the existence of such $\mathfrak{B}_{ijk}$ is equivalent to $(Q, G_{ijk})=b_{ijk}\sigma_i(Q,\Lambda Q)$. Since $(Q,\Lambda Q)\not=0$ (due to $p\not=3$), we can choose suitable constants $b_{ijk}$ to ensure the existence of $\mathfrak{B}_{ijk}$. We claim that
%From the construction of $\mathfrak{B}_{ijk}$, we obtain immediately that
\begin{align}\label{382}
|(H_{ij,5},Q')|\lesssim \frac{\sqrt{M_1}}{d^3},
\end{align}
where
$$H_{ij,5}=-\sum_{k=1}^n\frac{\mu_k\{(\mathcal{L}\mathfrak{B}_{ijk})(\tau_i^{-1}\varphi_{ij}-1)+[|D|,\tau_i^{-1}\varphi_{ij}]\mathfrak{B}_{ijk}\}}{x_{ij}}.$$
Indeed, from the fact that $\mathcal{L}Q'=0$, $Q'\in\mathcal{Y}_3$ and $\text{Supp }(\tau_i^{-1}\varphi_{ij}-1)\in(-\infty,d^{4}/4)$, we have
$$\big|\big(Q',(\mathcal{L}\mathfrak{B}_{ijk})(\tau_i^{-1}\varphi_{ij}-1)\big)\big|\lesssim \frac{1}{d^3}.$$
On the other hand, from \cite[Lemma 2.15]{MP}, we have
$$\big([|D|,\tau_i^{-1}\varphi_{ij}]\mathfrak{B}_{ijk},Q'\big)=-\big([|D|,\tau_i^{-1}\varphi_{ij}]\mathfrak{B}'_{ijk},Q\big)-\big([|D|,\tau_i^{-1}\varphi'_{ij}]\mathfrak{B}_{ijk},Q\big),$$
and
\begin{align*}
&\big|\big([|D|,\tau_i^{-1}\varphi_{ij}]\mathfrak{B}'_{ijk},Q\big)\big|\lesssim \|Q\|_{L^2}\|\mathfrak{B}_{ijk}'\|_{L^4}\|\varphi_{ij}'\|_{L^4}\lesssim \frac{1}{d^3},\\
&\big|\big([|D|,\tau_i^{-1}\varphi'_{ij}]\mathfrak{B}_{ijk},Q\big)\big|\lesssim \||D|(\mathfrak{B}_{ijk}\varphi_{ij}')\|_{L^2}+\|\varphi_{ij}'\|_{L^4}\|\mathfrak{B}'_{ijk}\|_{L^4}\lesssim \frac{1}{d^3},
\end{align*}
which concludes the proof of \eqref{382}. Using \eqref{382} and a similar arguments as before, we have
\begin{align*}
&\sum_{\substack{i,j,k=1,\\i\not=j}}^n\frac{\mu_k\partial_y\{\tau_i(\mathcal{L}\mathfrak{B}_{ijk})(\varphi_{ij}-1)+\tau_i([|D|,\varphi_{ij}]\mathfrak{B}_{ijk})\}}{x^3_{ij}}=\sum_{\substack{i,j,k=1,\\i\not=j}}^n\frac{\partial_y(\tau_iF_{ij,5})}{x_{ij}^2}+O_{\h}(\Gamma),
\end{align*}
for some $F_{ij,5}\in S$.

Next, we let
$$F_{ij}=\sum_{\ell=1}^5F_{ij,\ell}$$
and choose 
$$\mathfrak{A}_{ij}=\sum_{k=1}^n\mathfrak{A}_{ijk,1}+\mathfrak{A}_{ij,2}$$
for all $i\not=j$ such that
\begin{align}
&\mathcal{L}\mathfrak{A}_{ijk,1}=\mu_kE_{ijk},\label{378}\\
& \mathcal{L}\mathfrak{A}_{ij,2}=F_{ij}+\sum_{k=1}^n\mu_k(p-1)\Lambda Q Q^{p-2}\mathfrak{A}_{ijk,1}.\label{379}
\end{align}
From Proposition \ref{P1} and the fact that $E_{ijk}\in\mathcal{Y}_2$ are even functions, we know that there exist even functions $\mathfrak{A}_{ijk,1}\in\mathcal{Y}_2$ such that \eqref{378} is satisfied. To obtain the existence of $\mathfrak{A}_{ij,2}$ satisfying \eqref{379}, we just need to show that the right hand side of \eqref{379} are orthogonal to $Q'$, which is provided by the fact that $F_{ij}\in S$ and $\mathfrak{A}_{ijk,1}\in\mathcal{Y}_2$ are even functions.

Finally, injecting \eqref{380}--\eqref{379} into \eqref{381}, we have for all $t\in[t_0,t_1]$, 
$$\Gamma(t)\lesssim \frac{1}{d^{9/2}(t)}+\frac{M^{3/2}_1(t)}{d^2(t)}+\frac{M_2(t)}{d^3(t)}+\frac{M_3(t)}{d^2(t)},$$
and
$$\|E_V(t)\|_{\h}\lesssim \Gamma(t)\lesssim\frac{1}{d^{9/2}(t)}+\frac{M^{3/2}_1(t)}{d^2(t)}+\frac{M_2(t)}{d^3(t)}+\frac{M_3(t)}{d^2(t)}.$$
\end{proof}

\subsection{Geometrical decomposition}
With this strongly interacting multi-bubble, we can introduce the following geometrical decomposition for solutions near this multi-bubble:
\begin{proposition}\label{P7}
Let $u_0\in H^{\frac{1}{2}}$ such that there exist $C^1$-function $\vec{x}_0(t),\vec{\mu}_0(t)$ defined on $I=[t_0,t_1]$ satisfying \eqref{36} and
\footnote{Recall that $V(t,y)=\Theta(\vec{x}(t),\vec{\mu}(t),y)$ is the approximate multi-soliton.}
\begin{equation}
\label{37}
u_0(y)=\Theta(\vec{x}_0(t_1),\vec{\mu}_0(t_1),y)+\e_0(y),
\end{equation}
where $\|\e_0\|_{\h}\leq \omega_1$.
If $\omega_1$ is small enough, then there exist $t_0\leq t^*<t_1$, and unique $C^1$-functions
$$\vec{x}(t)=(x_1(t),\ldots,x_n(t)),\quad \vec{\mu}(t)=(\mu_1(t),\ldots,\mu_n(t))$$ and error term $\e(t)$
defined on $[t^*,t_1]$, such that the corresponding solution $u(t)$ to \eqref{CP} with initial data $u(t_1)=u_0$ exists on $[t^*,t_1]$, and satisfies the following conditions.
\begin{enumerate}
\item Geometrical decomposition:
\begin{equation}
\label{GD}
u(t,y+t)=\Theta(\vec{x}(t),\vec{\mu}(t),y)+\e(t,y).
\end{equation}
\item Orthogonality conditions:
\begin{equation}
\label{38}
\big(\e(t), \R_i(t)\big)=\big(\e(t),\partial_y \R_i(t)\big),
\end{equation}
for all $t\in[t^*,t_1]$ and $i\in\{1,\ldots,n\}$.
\item Estimates on the parameters: for all $t\in[t^*,t_1]$, there holds:
\begin{align}
&|\vec{\mu}(t)-\vec{\mu}_0(t)|+|\vec{x}(t)-\vec{x}_0(t)|+\|\e(t)\|_{H^{\frac{1}{2}}}\lesssim \delta(\omega_1), \label{39}\\
&\min_{i\in\{1,\ldots,n-1\}}(x_i(t)-x_{i+1}(t))\geq \frac{1}{\omega_0}-\delta(\omega_1).\label{315}
\end{align}
\item Continuous dependence on the initial data: for all fixed time $t\in[t^*,t_1]$, the parameters $\mu_i(t)$, $x_i(t)$ and the error term $\varepsilon(t)$ depends continuously on the initial data $u_0\in\h$ with respect to the weak topology of $\h$.
\end{enumerate}
\end{proposition}
The proof of Proposition \ref{P7} follows from standard arguments based on the implicit function theorem. We refer to \cite[Proposition 1]{MM} for an example of  such a proof. We mention here that the key point in the proof is the non-degeneracy of the following Jacobian matrix:
$$
\begin{vmatrix}
(\Lambda Q,Q) & (\Lambda Q, Q')\\
(Q',Q) & (Q',Q')
\end{vmatrix}
=(Q,\Lambda Q)(Q',Q)\not=0
$$ 
where we use the fact that $(Q,\Lambda Q)\not=0$ when $p\not=3$.

\subsection{Modulation estimates}
Using the orthogonality conditions \eqref{38}, we have the following modulation estimates:
\begin{proposition}\label{P8}
Let $I_0=[t^*,t_1]$, assume that the geometrical decomposition \eqref{GD} holds on $I_0$, then we have for all $i=1,\ldots,n$,
\begin{align}
&|\dot{x}_i-\mu_i|\lesssim\frac{1}{d^{9/2}}+ \|\e\|_{\h}+\frac{M^{3/2}_1}{d^2},\label{330}\\
&\bigg|\dot{\mu}_i+\sum^n_{\substack{j=1,\\j\not=i}}\frac{a_{ij}}{x^3_{ij}}+\sum^n_{\substack{k,j=1,\\j\not=i}}\frac{b_{ijk}\mu_k}{x^3_{ij}}\bigg|\lesssim \frac{1}{d^{9/2}}+\|\e\|^2_{\h}+\frac{\|\e\|_{\h}}{d^2}+\frac{M^{3/2}_1}{d^2}.\label{331}
\end{align}
\end{proposition}

\begin{proof}
From the geometrical decomposition, we have 
\begin{align}\label{383}
&\partial_t\varepsilon-\partial_y(|D|\e+\e-p|V|^{p-1}\e)\nonumber\\
&\qquad\qquad\qquad+\partial_y\big[|V+\e|^{p-1}(V+\e)-|V|^{p-1}V-p|V|^{p-1}\e\big]+\Psi_V=0.
\end{align}
Using the orthogonality condition \eqref{38}, we have
\begin{align}\label{384}
0=\partial_t\big(\e,\R_i\big)&=(\partial_t\e,\R_i)+\frac{\dot{\mu}_i}{1+\mu_i}(\e,\Lambda \R_i)+\dot{x}_i(\e,\partial_y\R_i)\nonumber\\
&=(\partial_t\e,\R_i)+O\bigg(\|\e\|_{\h}(M_2+M_3)+\frac{\|\e\|_{\h}}{d^3}\bigg),
\end{align}
and
\begin{align}\label{385}
0=\partial_t\big(\e,\partial_y\R_i\big)&=(\partial_t\e,\partial_y\R_i)+\frac{\dot{\mu}_i}{1+\mu_i}(\e,\Lambda (\partial_y\R_i))+\dot{x}_i(\e,\partial_{yy}\R_i)\nonumber\\
&=(\partial_t\e,\partial_y\R_i)+O\big(\|\e\|_{\h}+\|\e\|_{\h}(M_2+M_3)\big).
\end{align}
Combining \eqref{322}, \eqref{38}, \eqref{383}--\eqref{385} and Lemma \ref{L6}, we have:
\begin{align}\label{386}
&\bigg(\dot{\mu}_i+\sum^n_{\substack{j=1,\\j\not=i}}\frac{a_{ij}}{x^3_{ij}}+\sum^n_{\substack{k,j=1,\\j\not=i}}\frac{b_{ijk}\mu_k}{x^3_{ij}}\bigg)\frac{\sigma_i(\Lambda \R_i,\R_i)}{1+\mu_i}-\big(\partial_y(|D|\e+\e-p|V|^{p-1}\e),\R_i\big)\nonumber\\
&=O\bigg(\|\e\|^2_{\h}+\|E_V\|_{\h}+\|\e\|_{\h}(M_2+M_3)+\frac{\|\e\|_{\h}}{d^3}+\frac{M_2+M_3}{d^2}\bigg),
\end{align}
and
\begin{align}
\label{387}
&(\dot{x}_i-\mu_i)\times(\sigma_i\partial_y\R_i,\partial_y\R_i)\nonumber\\
&=O\bigg(\|\e\|_{\h}+\|E_V\|_{\h}+\|\e\|_{\h}(M_2+M_3)+\frac{M_2+M_3}{d^2}\bigg).
\end{align}
From Lemma \ref{L6}, and the construction of $V$, we have
$$(p|V|^{p-1}-p\R_i^{p-1})\partial_y\R_i=O_{L^2}\bigg(\frac{1}{d^2}\bigg).$$
Hence, we have
\begin{align*}
&\big(\partial_y(|D|\e+\e-p|V|^{p-1}\e),\R_i\big)=\big(\e,[|D|+1-p\R_i^{p-1}](\partial_y\R_i)\big)+O\bigg(\frac{\|\e\|_{\h}}{d^2}\bigg)\\
&=\big(\e,[|D|+1+\mu_i-p\R_i^{p-1}](\partial_y\R_i)\big)+O\bigg(\frac{\|\e\|_{\h}}{d^2}\bigg)=O\bigg(\frac{\|\e\|_{\h}}{d^2}\bigg).
\end{align*}
Here we use \eqref{38} as well as the fact that
$$|D|\R_i+(1+\mu_i)\R_i-\R_i^p=0$$
for the last two equalities.

Finally, combining \eqref{323}, \eqref{386} and \eqref{387}, using the fact that $(Q,\Lambda Q)\not=0$, we obtain \eqref{330} and \eqref{331}.
\end{proof}

\section{Existence of strongly interacting multi-solitary waves}\label{S4}
This section is devoted to prove Theorem \ref{MT1} by constructing a solution satisfying the corresponding conditions. In this section, we always assume that $n\geq 2$ and $\sigma_i=1$, if $p>3$; $\sigma_i=(-1)^{i-1}$, if $2<p<3$.

We consider solutions of the form
$$u(t,t+y)=w(t,y)\sim V(t,y)+\e(t,y).$$
Here we expect that the parameters $x_i(t)$ and $\mu_i(t)$ asymptotically behave like
\begin{equation}
\label{410}
 x_i\sim \alpha_i\sqrt{t}+\beta_i\log t+\gamma_i,\quad \dot{x}_i\sim\mu_i\sim \frac{\alpha_i}{2\sqrt{t}}+\frac{\beta_i}{t},
\end{equation}
as $t\rightarrow+\infty$. Here $\beta_i, \gamma_i$ are universal constants to be chosen later, while $\alpha_i$ are given by the following lemma: 

\begin{lemma}\label{L4}
Let $p\in(2,3)\cup(3,+\infty)$, $n\geq 2$. We assume that $\sigma_i=1$, if $p>3$; $\sigma_i=(-1)^{i-1}$, if $2<p<3$. Then there exist real numbers $\alpha_1>\alpha_2>\cdots>\alpha_n$, such that for all $i=1,2,\ldots,n$,
\begin{equation}
\label{411}
\frac{\alpha_i}{4}=\sum_{ \substack{j=1\\ j\not=i}}^n\frac{a_{ij}}{(\alpha_i-\alpha_j)^3},
\end{equation}
where $a_{ij}$ are given by \eqref{320}. Moreover, we have:
for all $i=1,2,\ldots,n$,
\begin{align}
\alpha_i+\alpha_{n+1-i}=0.\label{412}
\end{align}
\end{lemma}
\begin{remark}
The proof of this lemma is similar to \cite[Lemma 3]{JM}.
\end{remark}
\begin{proof}
We first consider the case when $n$ is even. Let $n=2k$, $k\in\mathbb{N}^+$. We consider $\{\alpha_i\}_{i=1}^{2k}$ of the following form:
\begin{align}
&\alpha_i=\sum_{j=i}^k\theta_j,\quad \forall i=1,\ldots,k,\label{462}\\
&\alpha_{n+1-i}=-\alpha_i,\quad \forall i=1,\ldots,k.\label{463}
\end{align}
We define a smooth function $F$ on 
$$S=\{\vec{\theta}=(\theta_1,\ldots,\theta_k)\in\mathbb{R}^k:\, \theta_i>0,\,\forall i=1,\ldots,k\}\subset\mathbb{R}^k$$
as follows
$$F(\vec{\theta})=-\sum_{\substack{i,j=1,\\\i\not=j}}^n\frac{a_{ij}}{(\alpha_i-\alpha_j)^2}-\frac{1}{2}\sum_{i=1}^n\alpha_i^2=-\sum_{\substack{i,j=1,\\\i\not=j}}^na_{ij}\Bigg(\sum_{\ell=\min\{i,j\}}^{\max\{i,j\}-1}\theta_\ell\Bigg)^{-2}-\frac{1}{2}\sum_{i=1}^n\alpha_i^2,$$
where $\alpha_i$ are given by \eqref{462} and \eqref{463}. 
%We denote by
%$$S_0=\bigg\{\vec{\theta}=(\theta_1,\ldots,\theta_k)\in S:\, \sum_{i=1}^k\alpha^2_i=1\bigg\}.$$

In case of $p>3$, we have $a_{ij}>0$ for all $i\not=j$. We observe that for all $\vec{\theta}_0\in\partial S$, there exits $i_0\in\{1,\ldots,k\}$ such that $\theta_{i_0}=0$. Hence, for all $\vec{\theta}_0\in\partial S$, and all $\vec{\theta}_m\in S$ such that $\vec{\theta}_m\rightarrow\vec{\theta}_0$ when $m\rightarrow+\infty$, we have
$$F(\vec{\theta}_m)\rightarrow-\infty.$$
On the other hand, we also have $F(\vec{\theta}_m)\rightarrow-\infty$ if $|\vec{\theta}_m|\rightarrow+\infty.$ Hence, there exits a local maximizer $\vec{\theta}_0=(\theta_{0,1},\ldots,\theta_{0,k})\in S$. We denote by  
\begin{align}
\alpha_{i}=\sum_{j=i}^k\theta_{0,j},\quad
\alpha_{n+1-i}=-\alpha_{i},\quad \forall i=1,\ldots,k.
\end{align}
It is easy to see that $(\alpha_{1},\ldots,\alpha_{n})$ is also a local maximizer of the following function: $G(\vec{\alpha})=F(\vec{\theta}),$ which implies that for all $i=1,\ldots,n$, there holds
$$\frac{\alpha_i}{4}=\sum_{ \substack{j=1\\ j\not=i}}^n\frac{a_{ij}}{(\alpha_i-\alpha_j)^3}.$$
We then conclude the proof of $p>3$ and $n$ even.

In case of $2<p<3$, we notice that $a_{ij}=a_{ji}$ for all $i\not=j$, and $a_{ij}>0$ if $i-j$ is odd; $a_{ij}<0$ if $i-j$ is even. Hence there exists constant $C_0>0$, such that
\begin{align}\label{464}
-\sum_{\substack{i,j=1,\\\i\not=j}}^n\frac{a_{ij}}{(\alpha_i-\alpha_j)^2}=-C_0\sum_{i=1}^{2k}\sum_{\ell=1}^{2k-i+1}\frac{(-1)^{i-1}}{(\theta_\ell+\cdots+\theta_{\ell+i-1})^2},
\end{align}
where $\theta_{2k+1-i}=\theta_i$ for all $i=1,\ldots,k$. For all $1\leq i<2k$, we have
\begin{align*}
&\sum_{\ell=1}^{2k-i+1}\frac{1}{(\theta_\ell+\cdots+\theta_{\ell+i-1})^2}-\sum_{\ell=1}^{2k-i}\frac{1}{(\theta_\ell+\cdots+\theta_{\ell+i})^2}\\
&\geq\frac{1}{(\theta_{2k-i+1}+\cdots+\theta_{2k})^2}+\sum_{\ell=1}^{2k-i}\frac{\theta_{\ell+i}}{(\theta_\ell+\cdots+\theta_{\ell+i-1})(\theta_\ell+\cdots+\theta_{\ell+i})^2}\\
&\qquad+\sum_{\ell=1}^{2k-i}\frac{\theta_{\ell+i}}{(\theta_\ell+\cdots+\theta_{\ell+i-1})^2(\theta_\ell+\cdots+\theta_{\ell+i})}.
\end{align*}
Using the above estimate, it is easy to see from \eqref{464} that 
$$F(\vec{\theta}_m)\rightarrow-\infty,$$
if $\vec{\theta}_m\rightarrow\vec{\theta}_0\in \partial S$ or $|\vec{\theta}_m|\rightarrow+\infty$, as $m$ $\rightarrow+\infty$. By a similar argument, we conclude the proof of Lemma \ref{L4} in case of $2<p<3$ and $n$ even.

Finally, if $n=2k+1$ is an odd integer. We can also consider $\{\alpha_i\}_{i=1}^{2k+1}$ of the following form:
\begin{align*}
&\alpha_i=\sum_{j=i}^k\theta_j,\quad \forall i=1,\ldots,k,\\
&\alpha_{k+1}=0,\quad\alpha_{n+1-i}=-\alpha_i,\quad \forall i=1,\ldots,k.
\end{align*}
Then, Lemma \ref{L4} in this case can be proved similarly.
\end{proof}

\subsection{Uniform backward estimate}
In this subsection, we will derive some crucial backward uniform estimates for solutions  satisfying the conditions introduced in Proposition \ref{P7}. More precisely, we have:

\begin{proposition}[Uniform backward estimates]\label{P9}
There exists $t_0\gg1$, such that for all initial time $\tin>t_0$, there exists a choice of initial data $u(\tin)=u_0$, such that the following \emph{uniform backward estimates: } for all $t\in[t_0,\tin]$, we have
\begin{align}
&|x_i(t)-\alpha_i\sqrt{t}-\beta_i\log t-\gamma_i|\leq \frac{1}{t^{1/4-2\delta_0}},\label{44}\\
& \bigg|\mu_i(t)-\frac{\alpha_i}{2\sqrt{t}}-\frac{\beta_i}{t}\bigg|\leq \frac{1}{t^{5/4-2\delta_0}},\label{45}\\
& \|\e(t)\|_{\h}\leq \frac{1}{t^{5/4-\delta_0}},\label{46}
\end{align}
for some small universal constants $0<\delta_0<\frac{1}{1000}$. Here $\alpha_i$ are given by \eqref{320}, and constants $\beta_i$, $\gamma_i$ will be chosen later%
\footnote{See Remark \ref{R2} for the detailed definition of these constants.}
.
\end{proposition}

\subsubsection{A bootstrap argument}
The main step of the proof of Proposition \ref{P9} is an argument of bootstrap as well as a topological argument. More precisely, for all given $\tin\gg1$, we consider an initial data of the form
$$u(\tin,y)=\Theta(\vec{x}_0(\tin),\vec{\mu}_0(\tin),y)+\e_0(y)$$
where $\vec{x}_0(t),\vec{\mu}_0(t)$ are $C^1$-functions on $[t_0,\tin]$ such that 
\begin{align}
&|x_{i,0}(\tin)-\alpha_i\sqrt{\tin}-\beta_i\log\tin-\gamma_i|\lesssim \tin^{-1/4+\delta_0},\label{453}\\
&\bigg|\mu_{i,0}(\tin)-\frac{\alpha_i}{2\sqrt{\tin}}-\frac{\beta_i}{\tin}\bigg|\lesssim \tin^{-5/4+\delta_0},\label{455}\\
&\|\e_0\|_{\h}\lesssim \tin^{-\frac{3}{2}}.\label{437}
\end{align}
%Let us fix the choice of the parameters $\vec{\mu}_0(\tin)$. _
From the (local) uniqueness of the geometrical decomposition (i.e. Proposition \ref{P7}), \emph{for all choice of $\vec{x}_0(\tin)$, $\vec{\mu}_0(\tin)$ and $\e_0$ satisfying \eqref{453}}, there exist%
\footnote{We mention here $t_*$ depends on the choice of $\vec{x}_0(\tin)$, $\vec{\mu}_0(\tin)$ and $\e_0$.}
 $t_*\in[t_0,\tin]$, $C^1$-parameters $\vec{x}(t),\vec{\mu}(t)$ and error term $\e(t,y)$, such that 
\begin{equation}\label{452}
\|\e(\tin)\|_{\h}\leq C \tin^{-\frac{3}{2}}
\end{equation}
and
\begin{align}
&x_{i}(\tin)=\alpha_i\sqrt{\tin}+\beta_i\log \tin+\gamma_i+O(\tin^{-1/4+\delta_0}),\label{465}\\
&\mu_{i}(\tin)=\frac{\alpha_i}{2\sqrt{\tin}}+\frac{\beta_i}{\tin}+O(\tin^{-5/4+\delta_0}).\label{466}
\end{align}
Moreover, the corresponding geometrical decomposition \eqref{GD} holds on $[t_*,\tin]$, for some $t_*\in[t_0, \tin)$. We still denote by $t_*$ the smallest time $\tau\in[t_0,\tin)$ such that the geometrical decomposition \eqref{GD} holds on $[\tau,\tin]$. We also introduce the following bootstrap assumptionss:
\begin{align}
&|x_i(t)-\alpha_i\sqrt{t}-\beta_i\log t-\gamma_i|\leq \frac{1}{t^{1/4-3\delta_0}},\label{415}\\
& \bigg|\mu_i(t)-\frac{\alpha_i}{2\sqrt{t}}-\frac{\beta_i}{t}\bigg|\leq \frac{1}{t^{5/4-3\delta_0}},\label{416}\\
& \|\e(t)\|_{\h}\leq \frac{1}{t^{5/4-\delta_0}},\label{417}
\end{align}
and define
\begin{equation}
\label{418}
T_0=\inf\{\tau\in[t_*,\tin):\,\text{\eqref{415}--\eqref{417} holds on $[\tau,\tin]$}\}.
\end{equation}
We will show that 
\begin{proposition}\label{P12}
There exists a suitable choice of $\vec{x}_0(\tin)$, $\vec{\mu}_0(\tin)$ and $\e_0$ (or equivalently, a suitable choice of initial data $u(\tin)$) such that%
\footnote{Here we use the fact that when $t_0$ is chosen large enough, then $T_0=t_*$ implies $T_0=t_*=t_0$.}
 $T_0=t_*=t_0$.
\end{proposition}

\begin{remark}
The key step of the proof of Proposition \ref{P12} is to show that with a suitable choice of initial data, we can improve the bootstrap assumptions \eqref{415}--\eqref{417} on $[T_0,\tin]$. We will see that in the proof of Proposition \ref{P12}  in the subcritical case, we can directly choose $\e_0=0$ and use a topological argument to choose suitable $\vec{x}_0(\tin)$ and $\vec{\mu}_0(\tin)$. While in the supercritical case, no explicit expression of $\e_0$ is known. We can only show its existence by a slightly different topological argument%
\footnote{See \cite{V1} for similar arguments.}
.
\end{remark}

\subsubsection{Linearization of the ODE system}
From the bootstrap assumptions \eqref{415}--\eqref{417}, we have the following \emph{a priori} estimates:
\begin{equation}
\label{419}
|x_{ij}(t)-(\alpha_i-\alpha_j)\sqrt{t}|\lesssim \log t,\quad |\mu_i(t)-\alpha_i/(2\sqrt{t})|\lesssim \frac{1}{t},
\end{equation}
for all $i\not=j$ and $t\in[T_0,\tin]$. Hence, from \eqref{330}, \eqref{331} and \eqref{419}, we have for all $i=1,\ldots,n$ and $t\in[T_0, \tin]$,
\begin{align}
&|\dot{x}_i(t)-\mu_i(t)|\lesssim\frac{1}{t^{5/4-\delta_0}},\label{420}\\
&\bigg|\dot{\mu}_i(t)+\sum^n_{\substack{j=1,\\j\not=i}}\frac{a_{ij}}{x^3_{ij}(t)}+\sum^n_{\substack{k,j=1,\\j\not=i}}\frac{b_{ijk}\mu_k(t)}{x^3_{ij}(t)}\bigg|\lesssim \frac{1}{t^{9/4-\delta_0}}.\label{421}
\end{align}
Using the \textit{a priori} estimates \eqref{419} again, we have 
\begin{align*}
\frac{a_{ij}}{x^3_{ij}(t)}&=\frac{a_{ij}}{(\alpha_i-\alpha_j)^3t^{\frac{3}{2}}}\frac{1}{(1+[x_{ij}(t)-(\alpha_i-\alpha_j)\sqrt{t}]/[(\alpha_i-\alpha_j)\sqrt{t}])^3}\\
&=\frac{a_{ij}}{(\alpha_i-\alpha_j)^3t^{\frac{3}{2}}}\Bigg[1-\frac{3[x_{ij}-(\alpha_i-\alpha_j)\sqrt{t}]}{(\alpha_i-\alpha_j)\sqrt{t}}+O\bigg(\frac{|x_{ij}-(\alpha_i-\alpha_j)\sqrt{t}|^2}{t}\bigg)\Bigg]\\
&=\frac{a_{ij}}{(\alpha_i-\alpha_j)^3t^{\frac{3}{2}}}\bigg(4-\frac{3x_{ij}}{(\alpha_i-\alpha_j)\sqrt{t}}\bigg)+O\bigg(\frac{1}{t^{9/4}}\bigg),
\end{align*}
and
\begin{align*}
\frac{b_{ijk}\mu_k(t)}{x^3_{ij}(t)}&=\frac{b_{ijk}\alpha_k}{(\alpha_i-\alpha_j)^3t^{2}}\frac{1+[\mu_k(t)-\alpha_k/\sqrt{t}]/(\alpha_{k}\sqrt{t})}{(1+[x_{ij}(t)-(\alpha_i-\alpha_j)\sqrt{t}]/[(\alpha_i-\alpha_j)\sqrt{t}])^3}\\
&=\frac{b_{ijk}\alpha_k}{(\alpha_i-\alpha_j)^3t^{2}}+O\bigg(\frac{1}{t^{9/4}}\bigg).
\end{align*}
Together with \eqref{421} and  \eqref{411}, we have for all $i=1,\ldots,n$ and $t\in[T_0, \tin]$,
\begin{align}
&|\dot{x}_i(t)-\mu_i(t)|\lesssim\frac{1}{t^{5/4-\delta_0}},\label{422}\\
&\bigg|\dot{\mu}_i(t)+\frac{\alpha_i}{t^{\frac{3}{2}}}-\sum^n_{j=1}\frac{m_{ij}x_j(t)}{t^2}+\sum^n_{\substack{k,j=1,\\j\not=i}}\frac{b_{ijk}\alpha_k}{(\alpha_i-\alpha_j)^3t^2}\bigg|\lesssim \frac{1}{t^{9/4-\delta_0}}\label{423},
\end{align}
where
\begin{align}
&m_{ij}=-\frac{3a_{ij}}{(\alpha_i-\alpha_j)^4},\quad \text {if } i\not=j,\label{413}\\
&m_{ii}=\sum^n_{\substack{j=1\\ j\not=i}}\frac{3a_{ij}}{(\alpha_i-\alpha_j)^4},\quad \forall i=1,\ldots,n.\label{414}
\end{align}
Let $M=(m_{ij})_{n\times n}$ be the symmetric matrix given by \eqref{413} and \eqref{414}. We also let $\lambda_1\leq\cdots\leq\lambda_n$ be all eigenvalues of $M$. We know that there exists $T\in SO(n)$ such that 
$$M=T^{-1}
\begin{bmatrix}
\lambda_1 & & \\
&  \ddots&\\
& &\lambda_n
\end{bmatrix}
T.$$
Let 
\begin{align}
&(\widetilde{x}_1(t),\ldots,\widetilde{x}_n(t))^\intercal=T(x_1(t),\ldots,x_n(t))^\intercal, \label{424}\\
&(\widetilde{\mu}_1(t),\ldots,\widetilde{\mu}_n(t))^\intercal=T(\mu_1(t),\ldots,\mu_n(t))^\intercal,\label{425}\\
&(\widetilde{\alpha}_1,\ldots,\widetilde{\alpha}_n)^\intercal=T(\alpha_1,\ldots,\alpha_n)^\intercal.\label{427}
\end{align}
Then from \eqref{420} and \eqref{422} --\eqref{414}, we have for all $k=1,\ldots,n$ and $t\in[T_0, \tin]$, there holds
\begin{align}
&|\partial_t\widetilde{x}_k(t)-\widetilde{\mu}_k(t)|\lesssim\frac{1}{t^{5/4-\delta_0}},\label{428}\\
&\bigg|\partial_t\widetilde{\mu}_k(t)+\frac{\widetilde{\alpha}_k}{t^{\frac{3}{2}}}-\frac{\lambda_k\widetilde{x}_k(t)}{t^2}+\frac{\widetilde{\beta}_k}{t^2}\bigg|\lesssim \frac{1}{t^{9/4-\delta_0}},\label{429}
\end{align}
where $\widetilde{\beta}_k$ are universal constants. For all $k=1,\ldots,n$ and $t\in[T_0, \tin]$, we let
\begin{align}
&\begin{cases}
\x_k(t)=\widetilde{x}_k(t)-\widetilde{\alpha}_k\sqrt{t}-\widetilde{\beta}_k\log t,\\
 \U_k(t)=\widetilde{\mu}_k(t)-\frac{\widetilde{\alpha}_k}{2\sqrt{t}}-\frac{\widetilde{\beta}_k}{t},\\
\end{cases}& \text{if }\lambda_k=0;\label{430}\\
&\begin{cases}
\x_k(t)=\widetilde{x}_k(t)-\widetilde{\alpha}_k\sqrt{t}-\frac{\widetilde{\beta}_k}{\lambda_k} ,\\
 \U_k(t)=\widetilde{\mu}_k(t)-\frac{\widetilde{\alpha}_k}{2\sqrt{t}},
\end{cases}&\text{if }\lambda_k\not=0,\label{426}
\end{align}
Combining \eqref{428}--\eqref{426}, we have
\begin{align*}
&|\partial_t\x_k(t)-\U_k(t)|\lesssim\frac{1}{t^{5/4-\delta_0}},\\
&\bigg|\partial_t\U_k(t)+\frac{\widetilde{\alpha}_k}{t^{\frac{3}{2}}}\bigg(\frac{3}{4}-\lambda_k\bigg)-\frac{\lambda_k\x_k(t)}{t^2}\bigg|\lesssim \frac{1}{t^{9/4-\delta_0}}.
\end{align*}
It is easy to verify that $(\alpha_1,\ldots,\alpha_n)^\intercal$ is an eigenvector of $M$, whose corresponding eigenvalue is $\frac{3}{4}$. From the definition of $(\widetilde{\alpha}_1,\ldots,\widetilde{\alpha}_n)^\intercal$, we have for all $k=1,\ldots,n$, $\widetilde{\alpha}_k(\frac{3}{4}-\lambda_k)=0$. Collecting all estimates above, we obtain that for all $k=1,\ldots,n$ and $t\in[T_0, \tin]$, there holds
\begin{align}
&|\partial_t\x_k(t)-\U_k(t)|\lesssim\frac{1}{t^{5/4-\delta_0}},\label{431}\\
&\big|\partial_t\U_k(t)-[\lambda_k\x_k(t)]/t^2\big|\lesssim \frac{1}{t^{9/4-\delta_0}}.\label{432}
\end{align}

\begin{remark}\label{R2}
We mention here that from the definition of $\x_k$ and $\U_k$, we know that with suitable choice of constants $\beta_i$ and $\gamma_i$, the bootstrap assumptions \eqref{415} and \eqref{416} are equivalent to 
\begin{equation}
\label{433}
|\x_k(t)|\lesssim \frac{1}{t^{1/4-3\delta_0}},\quad |\U_k(t)|\lesssim \frac{1}{t^{5/4-3\delta_0}},
\end{equation}
for all $t\in [T_0,\tin]$ and $k=1,\ldots,n$.
\end{remark}

Now, for all $k=1,\ldots,n$, we denote by $\omega_k^0,\ \omega_k^1$ the two roots of the polynomial $z^2+z-\lambda_k$ in $\mathbb{C}$. Obviously, we have $\omega_k^0+\omega_k^1=-1$. Without loss of generality, we assume that $\Re \omega_k^0\geq-\frac{1}{2}$ and $\Re\omega_k^1\leq -\frac{1}{2}$.
We let 
\begin{equation}\label{478}
y_k(t)=\omega_k^0\x_k(t)+ t\U_k(t),\quad z_k(t)=\omega_k^1\x_k(t)+ t\U_k(t).
\end{equation}
Then from \eqref{431} and \eqref{432}, we have
\begin{align}\label{434}
\partial_t y_k&=\omega_k^0\partial_t \x_k+\U_k+ [t\partial_t\U_k]=(1+\omega_k^0)\U_k+\frac{\lambda_k\x_k}{t}+O\bigg(\frac{1}{t^{5/4-\delta_0}}\bigg)\nonumber\\
&=\frac{(1+\omega_k^0)y_k}{t}+O\bigg(\frac{1}{t^{5/4-\delta_0}}\bigg),
\end{align}
and
\begin{equation}
\label{435}
\partial_t z_k=\frac{(1+\omega_k^1)z_k}{t}+O\bigg(\frac{1}{t^{5/4-\delta_0}}\bigg).
\end{equation}
\begin{remark}
We mention here,  \eqref{434} and \eqref{435} imply that 
$$|\partial_t(t^{-1-\omega_k^0}y_k)|=O\bigg(\frac{1}{t^{9/4+\omega_k^0-\delta_0}}\bigg),\;\; |\partial_t(t^{-1-\omega_k^1}z_k)|=O\bigg(\frac{1}{t^{9/4+\omega_k^1-\delta_0}}\bigg).$$
If we have $\Re\omega_k^0>-\frac{5}{4}$ and $\Re\omega_k^1>-\frac{5}{4}$ (or equivalently $\lambda_k<\frac{5}{16}$), then we
can arbitrarily choose initial data as long as $|\x_k(\tin)|\lesssim \tin^{-1/4+\delta_0}$ and $|\U_k(\tin)| \lesssim \tin^{-5/4+\delta_0}$ hold true. In this situation, we can improve the bootstrap assumptions \eqref{415} and \eqref{416} by directly integrating the above inequalities from $t$ to $\tin$. However, this condition does not always hold. We need to choose the initial data through a topological argument.
\end{remark}

\subsubsection{A monotonicity formula from the energy conservation law}
In this part, we will prove some monotonicity formula derived from the energy conservation law.

First of all, we recall that $\psi$ is a smooth function such that $\psi(y)=1$ if $y>-1$; $\psi(y)=0$ if $y<-2$, and $\psi'\geq0$. For all $i=1,\ldots,n$, we define $\phi_i^\pm$ as follows
\begin{align*}
&\phi_1^+(y)=\mathbf{1}_{[\alpha_1,+\infty)}(y),\quad\phi_n^-(y)=\mathbf{1}_{(-\infty,\alpha_n]}(y),\\
&\phi_i^-(y)=\psi\bigg(\frac{3(y-\alpha_i)}{\alpha_i-\alpha_{i+1}}\bigg)\mathbf{1}_{[\alpha_{i+1},\alpha_i]}(y),\quad \forall i=1,\ldots,n-1,\\
&\phi_i^+(y)=[1-\phi_{i-1}^-(y)]\mathbf{1}_{[\alpha_i,\alpha_{i-1}]}(y),\quad \forall i=2,\ldots,n.
\end{align*}
Let $\phi_i=\phi_i^++\phi_i^-$, then we can easily show that
\begin{itemize}
\item $\phi_i$ are smooth functions with $0\leq \phi_i\leq 1$ and $\sum_{i=1}^{n}\phi_i\equiv1$;
\item $|\partial_y\phi_i|\lesssim 1/[\min_{i\in\{1,\ldots,n-1\}}(\alpha_i-\alpha_{i+1})]$;
\item there holds 
$$\begin{cases}
&\text{Supp }\phi_1\subset\big([2\alpha_2+\alpha_1]/3,+\infty\big);\\
&\text{Supp }\phi_i\subset\big([2\alpha_{i+1}+\alpha_i]/3,[\alpha_i+2\alpha_{i-1}]/3\big),\; \forall 1<i<n;\\
&\text{Supp }\phi_n\subset\big(-\infty,[2\alpha_{n-1}+\alpha_n]/3\big);\\
\end{cases}$$
\end{itemize}
Then we define weight functions $\{\Phi_j(t,y)\}_{j=1,2}$ as follows:
\begin{align}
\Phi_1(t,y)=\sum_{i=1}^n\phi_i(y/\sqrt{t})[1-\mu_i(t)],\quad\Phi_2(t,y)=\sum_{i=1}^n\frac{\mu_i(t)\phi_i(y/\sqrt{t})}{1+[\mu_i(t)]^2}.
\end{align}
We have the following properties of these two weight functions $\Phi_1$ and $\Phi_2$:
\begin{lemma}\label{L3}
Let $\Phi_1$ and $\Phi_2$ be the weight functions defined as above, then the following properties hold provided that $t_0\gg1$.
\begin{enumerate}
\item If $|y-x_i(t)|\leq d(t)/4$ for some $i\in\{1,\ldots,n\}$, then
$$\Phi_1(t,y)=1-\mu_i,\quad\Phi_2(t,y)=\frac{\mu_i(t)}{1+[\mu_i(t)]^2}.$$
\item There holds
$$\text{Supp }\partial_y\Phi_1(t,\cdot)\subset \bigcup_{i=1}^{n-1}\bigg[\frac{3x_{i+1}+x_{i}}{4},\frac{3x_i+x_{i+1}}{4}\bigg].$$
\item  For all $y\in\mathbb{R}$, there holds
$$\partial_y\Phi_1(t,y)\leq 0,\quad |\partial_y\Phi_k(t,y)|\lesssim\frac{1}{t},\quad  |\partial_{yy}\Phi_k(t,y)|\lesssim\frac{1}{t^{\frac{3}{2}}},$$
for all $k=1,2$.
\item There holds 
$$\|\partial_y\Phi_k(t)\|_{L^2}\lesssim \frac{1}{t^{\frac{3}{4}}},\quad \|\partial_{yy}\Phi_k(t)\|_{L^2}\lesssim \frac{1}{t^{\frac{5}{4}}},$$
for all $k=1,2$.
\item For all $y\in\mathbb{R}$, there holds
$$|\partial_y\Phi_1(t,y)+\partial_y\Phi_2(t,y)|\lesssim \frac{1}{t^2}.$$
\item For all $y\in\mathbb{R}$, there holds
$$|\partial_t\Phi_1(t,y)|+|\partial_t\Phi_2(t,y)|\lesssim \frac{1}{t^{\frac{3}{2}}}.$$
\end{enumerate}
\end{lemma}
\begin{proof}
Most conclusions of this lemma follow directly from the construction of $\phi_i$ and the bootstrap assumptions \eqref{415}--\eqref{417}. While for the first inequality of (3), we use the fact that $\sum_{i=1}^n\phi_i'\equiv0$ and (1) to obtain for all $i\in\{1,\ldots,n-1\}$ and 
$$y\in\bigg[\frac{3x_{i+1}+x_{i}}{4},\frac{3x_i+x_{i+1}}{4}\bigg],$$
we have
\begin{align*}
\partial_y\Phi_1(t,y)&=-\frac{2\mu_i}{\sqrt{t}}[\phi_i^-]'(y)-\frac{2\mu_{i+1}}{\sqrt{t}}[\phi_{i+1}^+]'(y)\\
&=-\frac{2\mu_i}{\alpha_i-\alpha_{i+1}}\psi'\bigg(\frac{3(y/\sqrt{t}-\alpha_i)}{\alpha_i-\alpha_{i+1}}\bigg)+\frac{2\mu_{i+1}}{\alpha_i-\alpha_{i+1}}\psi'\bigg(\frac{3(y/\sqrt{t}-\alpha_i)}{\alpha_i-\alpha_{i+1}}\bigg)
\end{align*}
Together with $\mu_1>\cdots>\mu_n$ and $\psi'\geq 0$, we obtain that $\partial_y\Phi_1(t,y)\leq 0$.
\end{proof}

We then define a localized nonlinear energy functional as follows
\begin{align}
\label{47}
W(t)=&\int\big|\D (\e\sqrt{\Phi_1})\big|^2+\e^2(\Phi_1+\Phi_2)\nonumber\\
&-\frac{2}{p+1}\Big[\big(|V+\e|^{p+1}-|V|^{p+1}-(p+1)|V|^{p-1}V\e\big)\Big]\Phi_1
\end{align}
We have the following monotonicity formula:
\begin{proposition}\label{P10}
For all $t\in[T_0,\tin]$, we have
\begin{equation}
\label{48}
\partial_t W(t)\gtrsim-\bigg(\frac{1}{t^{1+\delta_1}}\|\e(t)\|_{\h}^2+\frac{1}{t^{9/4}}\|\e(t)\|_{\h}\bigg),
\end{equation}
where $0<\delta_1<\min\{1/1000,p-2\}$ is a small universal constant. 
\end{proposition}

\begin{proof}
\noindent\textbf{Step 1:} Expansion of $\partial_tW$.

By direct computation, we have:
\begin{align*}
\frac{1}{2}\partial_tW=W_1+W_2+W_3+W_4,
\end{align*}
where
\begin{align*}
&W_1=\int\partial_t\e\big[|D|\e+\e-(|V+\e|^{p-1}(V+\e)-|V|^{p-1}V)\big]\Phi_1,\\
&W_2=\int\partial_t\e\sqrt{\Phi_1}\big[|D|,\sqrt{\Phi_1}\big]\e+\int\partial_t\e\e\Phi_2,\\
&W_3=-\int\partial_tV\big(|V+\e|^{p+1}-|V|^{p+1}-(p+1)|V|^{p-1}V\e\big)\Phi_1,\\
&W_4=\frac{1}{2}\int\frac{\partial_t\Phi_1}{\sqrt{\Phi_1}}\e\D(\e\sqrt{\Phi_1})+\frac{1}{2}\int\e^2(\partial_t\Phi_1+\partial_t\Phi_2)\\
&\qquad\quad-\frac{1}{p+1}\int\big(|V+\e|^{p+1}-|V|^{p+1}-(p+1)|V|^{p-1}V\e\big)\partial_t\Phi_1
\end{align*}

\noindent\textbf{Step 2:} Estimate of $W_1$. 

From Proposition \ref{P6}, we have
\begin{align*}
\partial_t\e=&\partial_y\big[|D|\e+\e-|V+\e|^{p-1}(V+\e)+|V|^{p-1}V\big]-E_V+\sum_{i=1}^nP_{i,1}+\sum_{i=1}^nP_{i,2},
\end{align*}
where
\begin{align*}
&P_{i,1}=(\dot{x}_i-\mu_i)\sigma_i\partial_y\widetilde{R}_i\\
&P_{i,2}=-\bigg(\dot{\mu}_i+\sum^n_{\substack{j=1,\\j\not=i}}\frac{a_{ij}}{x^3_{ij}}+\sum^n_{\substack{k,j=1,\\j\not=i}}\frac{b_{ijk}\mu_k}{x^3_{ij}}\bigg)\frac{\sigma_i\Lambda\widetilde{R}_i}{1+\mu_i}.
\end{align*}
Hence, 
\begin{align*}
W_1=&-\frac{1}{2}\int\big[|D|\e+\e-(|V+\e|^{p-1}(V+\e)-|V|^{p-1}V)\big]^2\partial_y\Phi_1\\
&-\int E_V\big[|D|\e+\e-(|V+\e|^{p-1}(V+\e)-|V|^{p-1}V)\big]\Phi_1\\
&+\sum_{k=1}^2\sum_{i=1}^n\int P_{i,k}(|D|\e+\e-p\R^{p-1}_i\e)\Phi_1\\
&-\sum_{k=1}^2\sum_{i=1}^n\int P_{i,k}\big[|V+\e|^{p-1}(V+\e)-|V|^{p-1}V-p|V|^{p-1}\e\big]\Phi_1\\
&-\sum_{k=1}^2\sum_{i=1}^n\int P_{i,k}(p|V|^{p-1}\e-p\R_i^{p-1}\e)\Phi_1.
\end{align*}
From \eqref{323}, \eqref{330}, \eqref{331} and the bootstrap assumptions \eqref{415}--\eqref{417}, we have
\begin{align}
\label{438}
&\int E_V\big[|D|\e+\e-(|V+\e|^{p-1}(V+\e)-|V|^{p-1}V)\big]\Phi_1\nonumber\\
&=O\bigg(\frac{1}{t^{5/2}}+\frac{\|\e(t)\|_{\h}}{t^{3/2}}+\|\e(t)\|^2_{\h}\bigg)\|\e\|_{\h}=O\big(t^{-1-\delta_1}\|\e\|_{\h}^2+t^{-9/4}\|\e\|_{\h}\big),
\end{align}
and
\begin{align}
\label{439}
&\sum_{k=1}^2\sum_{i=1}^n\int P_{i,k}\big[|V+\e|^{p-1}(V+\e)-|V|^{p-1}V-(p-1)|V|^{p-1}\e\big]\Phi_1\nonumber\\
&=O\big(t^{-1-\delta_1}\|\e\|_{\h}^2+t^{-9/4}\|\e\|_{\h}\big).
\end{align}
Using the fact that%
\footnote{Here we use the fact that $\delta_1<p-2$.}
$$\||V|^{p-1}-\R_i^{p-1}\|_{L^\infty}\lesssim \frac{1}{d^{2p-2}(t)}\lesssim \frac{1}{t^{p-1}}\lesssim \frac{1}{t^{1+\delta_1}},$$
we also have
\begin{align}
&\sum_{k=1}^2\sum_{i=1}^n\int P_{i,k}(p|V|^{p-1}\e-p\R_i^{p-1}\e)\Phi_1\nonumber\\
&=O\bigg(\frac{1}{t^{5/2}}+\|\e(t)\|_{\h}\bigg)\frac{\|\e(t)\|_{\h}}{t^{1+\delta_1}}=O\big(t^{-1-\delta_1}\|\e\|_{\h}^2+t^{-9/4}\|\e\|_{\h}\big),
\end{align}
Meanwhile, by a similar argument and Lemma \ref{L3}, we have
\begin{align*}
&\sum_{k=1}^2\sum_{i=1}^n\int P_{i,k}(|D|\e+\e-p\R^{p-1}_i\e)(\Phi_1-1+\mu_i)\nonumber\\
&=O\bigg(\frac{1}{t^{5/2}}+\|\e(t)\|_{\h}\bigg)\sum_{i=1}^n\int_{|y-x_i(t)|\geq \frac{d(t)}{10}}\frac{|\e(y)|}{[y-x_i(t)]^3}\,\dd y+O\bigg(\frac{\|\e\|_{\h}^2}{t^{1+\delta_1}}+\frac{\|\e\|_{\h}}{t^{\frac{9}{4}}}\bigg)\nonumber\\
&=O\big(t^{-1-\delta_1}\|\e\|_{\h}^2+t^{-9/4}\|\e\|_{\h}\big).
\end{align*}
Together with the following identity:
$$|D|(\partial_y\R_i)+(1+\mu_i)\partial_y\R_i-p\R^{p-1}_i(\partial_y\R_i)=0,$$
we have
\begin{align}
\label{440}
&\sum_{k=1}^2\sum_{i=1}^n\int P_{i,k}(|D|\e+\e-p\R^{p-1}_i\e)\Phi_1\nonumber\\
&=\sum_{i=1}^n(\dot{x}_i-\mu_i)\int \sigma_i(1-\mu_i)\partial_y\R_i\big[|D|\e+\e-p\R^{p-1}_i\e)\big]+O\bigg(\frac{\|\e\|_{\h}^2}{t^{1+\delta_1}}+\frac{\|\e\|_{\h}}{t^{\frac{9}{4}}}\bigg)\nonumber\\
&=-\sum_{i=1}^n\sigma_i(1-\mu_i)\mu_i(\dot{x}_i-\mu_i)\int\R_i\e+O\bigg(\frac{\|\e\|_{\h}^2}{t^{1+\delta_1}}+\frac{\|\e\|_{\h}}{t^{\frac{9}{4}}}\bigg)\nonumber\\
&=O\bigg(\frac{\|\e\|_{\h}^2}{t^{1+\delta_1}}+\frac{\|\e\|_{\h}}{t^{\frac{9}{4}}}\bigg)
\end{align}
Combining \eqref{438}--\eqref{440}, we have
\begin{align*}
W_1=&-\frac{1}{2}\int(|D|\e)^2\partial_{y}\Phi_1-\frac{1}{2}\int\e^2\partial_y\Phi_1-\int|D|\e\e\partial_y\Phi_1+O\bigg(\frac{\|\e\|_{\h}^2}{t^{1+\delta_1}}+\frac{\|\e\|_{\h}}{t^{\frac{9}{4}}}\bigg).
\end{align*}
Using \cite[Lemma 2.16]{MP} and Lemma \ref{L3}, we have
\begin{align*}
&\bigg|\int|D|\e\e\partial_y\Phi_1-\int\big|\D\e\big|^2\partial_y\Phi_1\bigg|\lesssim\|\e\|_{\h}^2\|\partial_{yy}\Phi_1\|_{L^2}^{\frac{3}{4}}\|\partial_{y}\Phi_1\|_{L^2}^{\frac{1}{4}}\lesssim\frac{\|\e\|_{\h}^2}{t^{\frac{9}{8}}}.
\end{align*}
Using the fact that $\partial_y\Phi_1\leq 0$, we have
\begin{align}
\label{441}
W_1\geq -\int\big|\D\e\big|^2\partial_y\Phi_1-\frac{1}{2}\int\e^2\partial_{y}\Phi_1+O\big(t^{-1-\delta_1}\|\e\|_{\h}^2+t^{-\frac{9}{4}}\|\e\|_{\h}\big).
\end{align}

\noindent\textbf{Step 3:} Estimate of $W_2$. 

Using the equation of $\e$, we have
\begin{align*}
W_2=W_{2,1}+W_{2,2}+W_{2,3}+W_{2,4}+W_{2,5},
\end{align*}
where
\begin{align*}
&W_{2.1}=\int(|D|\e+\e)_y\sqrt{\Phi_1}\big[|D|,\sqrt{\Phi_1}\big]\e,\\
&W_{2,2}=-\int\big(|V+\e|^{p-1}(V+\e)-|V|^{p-1}V\big)_y\sqrt{\Phi_1}\big[|D|,\sqrt{\Phi_1}\big]\e,\\
&W_{2,3}=-\int \bigg(E_V-\sum_{k=1}^2\sum_{i=1}^nP_{i,k}\bigg)\sqrt{\Phi_1}\big[|D|,\sqrt{\Phi_1}\big]\e,\\
&W_{2,4}=\int\Big\{\partial_y\big[|D|\e+\e-|V+\e|^{p-1}(V+\e)+|V|^{p-1}V\big]\Big\}\e\Phi_2,\\
&W_{2,5}=-\int \bigg(E_V-\sum_{k=1}^2\sum_{i=1}^nP_{i,k}\bigg)\e\Phi_2.
\end{align*}
We estimate each term separately. First, for $W_{2,1}$, we have%
\footnote{Recall that $\mathcal{H}$ denotes the Hilbert transform on $\mathbb{R}$.}
\begin{align}
\label{442}
\sqrt{\Phi_1}\big[|D|,\sqrt{\Phi_1}\big]\e&=\sqrt{\Phi_1}\big[\mathcal{H}\partial_y,\sqrt{\Phi_1}\big]\e=\sqrt{\Phi_1}\mathcal{H}\big((\partial_y\sqrt{\Phi_1})\e\big)+\sqrt{\Phi_1}\big[\mathcal{H},\sqrt{\Phi_1}\big]\e_y\nonumber\\
&=\frac{1}{2}\partial_y\Phi_1(\mathcal{H}\e)+\sqrt{\Phi_1}\big[\mathcal{H},\partial_y\sqrt{\Phi_1}\big]\e+\sqrt{\Phi_1}\big[\mathcal{H},\sqrt{\Phi_1}\big]\e_y.
\end{align}
Hence, we have
\begin{align*}
W_{2,1}&=\int(|D|\e+\e)_y\sqrt{\Phi_1}\Big(\big[\mathcal{H},\partial_y\sqrt{\Phi_1}\big]\e+\big[\mathcal{H},\sqrt{\Phi_1}\big]\e_y\Big)\nonumber\\
&=\sum_{\ell=0}^2\binom{2}{\ell}\int(\mathcal{H}\e)\big(\partial_y^\ell\sqrt{\Phi_1}\big)\big(\partial_y^{3-\ell}\big[\mathcal{H},\sqrt{\Phi_1}\big]\e\big)\nonumber\\
&\quad-\int\e\big(\partial_y\sqrt{\Phi_1}\big)\big(\big[\mathcal{H},\sqrt{\Phi_1}\big]\e_y\big)-\int\e\sqrt{\Phi_1}\big(\partial_y\big[\mathcal{H},\sqrt{\Phi_1}\big]\e_y\big)\nonumber\\
&\quad+\int(|D|\e+\e)_y\sqrt{\Phi_1}\big[\mathcal{H},\partial_y\sqrt{\Phi_1}\big]\e+\frac{1}{2}\int(|D|\e+\e)_y(\partial_y\Phi_1)\mathcal{H}\e.
\end{align*}
Using \cite[Lemma 2.15, Lemma 2.16]{MP}, Lemma \ref{L3} and integration by parts, we have
\begin{align}
\label{443}
W_{2,1}&=\frac{1}{2}\int(|D|\e+\e)_y(\partial_y\Phi_1)\mathcal{H}\e+O\big(\|\e\|^2_{\h}[\|\partial_y\Phi_1\|^3_{L^\infty}+\|\partial_{yy}\Phi_1\|_{L^\infty}^{\frac{3}{2}}+\|\partial_{yyy}\Phi_1\|_{L^\infty}]\big)\nonumber\\
&=\frac{1}{2}\int(|D|\e+\e)_y(\partial_y\Phi_1)\mathcal{H}\e+O\big(t^{-1-\delta_1}\|\e\|^2_{\h}\big)\nonumber\\
&=-\frac{1}{2}\int\big||D|\e\big|^2\partial_y\Phi_1+\frac{1}{4}\int\big|\mathcal{H}\e\big|^2\partial^3_{y}\Phi_1-\frac{1}{2}\int\big[|D|(\mathcal{H}\e)\big]\mathcal{H}\e\partial_y\Phi_1+O\big(t^{-1-\delta_1}\|\e\|_{\h}\big)\nonumber\\
&=-\frac{1}{2}\int\big||D|\e\big|^2\partial_y\Phi_1-\frac{1}{2}\int\big|\D(\mathcal{H}\e)\big|^2\partial_y\Phi_1+O\big(t^{-1-\delta_1}\|\e\|^2_{\h}\big)\nonumber\\
&\geq O\big(t^{-1-\delta_1}\|\e\|^2_{\h}\big),
\end{align}
where we use the fact that $\partial_y\Phi_1\leq 0$ for the last inequality. For $W_{2,1}$, using the \cite[Lemma 2.15, Lemma 2.16]{MP} and \eqref{442} again, we have
\begin{align*}
W_{2,2}&=-\frac{p}{2}\int\Big(\e_y|V|^{p-1}\partial_{y}\Phi_1(\mathcal{H}\e)+\e\partial_y(|V|^{p-1})\partial_y\Phi_1(\mathcal{H}\e)\Big)+O(\|\e\|_{\h}^3).
\end{align*}
From Lemma \ref{L3}, we know that for all $y\in\text{Supp }\partial_y\Phi_1$, there holds
$$|V|^{p-1}(y)+\big|\partial_y(|V|^{p-1})(y)\big|\lesssim \frac{1}{t^{p-1}}\lesssim\frac{1}{t^{1+\delta_1}},$$
which implies that
\begin{align}
\label{444}
W_{2,2}=O\big(t^{-1-\delta_1}\|\e\|^2_{\h}\big).
\end{align}
Similarly, we have
\begin{align}
\label{445}
W_{2,3}=O\bigg(\frac{1}{t^{1+\delta_1}}\|\e\|^2_{\h}+\frac{1}{t^{\frac{9}{4}}}\|\e\|_{\h}\bigg).
\end{align}
Next, for $W_{2,4}$, using \cite[Lemma 2.16]{MP} and integration by parts, we have,
\begin{align*}
W_{2,4}&=-\int(|D|\e)\e_y\Phi_2-\int(|D|\e)\e\partial_y\Phi_2-\frac{1}{2}\int\e^2\partial_y\Phi_2\\
&\quad+p\int\e|V|^{p-1}\partial_y(\e\Phi_2)+O(\|\e\|^3_{\h})\\
&=-\int\big|\big|\D\e\big|^2\partial_y\Phi_2-\frac{1}{2}\int\e^2\partial_y\Phi_2-\frac{p(p-1)}{2}\int\e^2|V|^{p-3}V(\partial_yV\Phi_2)\\
&\quad+O\bigg(\frac{1}{t^{1+\delta_1}}\|\e\|^2_{\h}+\frac{1}{t^{\frac{9}{4}}}\|\e\|_{\h}\bigg).
\end{align*}
Recall from the definition of $V$ and $\Phi_2$, we have
\begin{align*}
\|V\partial_y\Phi_2\|_{L^\infty}\lesssim \frac{1}{t^2},\quad\bigg\|\Phi_2\partial_yV-\sum_{i=1}^n\sigma_i\mu_i\partial_y\R_i\bigg\|\lesssim \frac{1}{t^{\frac{3}{2}}}.
\end{align*}
Collecting all estimates above, we have
\begin{align}\label{446}
W_{2,4}&=-\int\big|\D\e\big|^2\partial_y\Phi_2-\frac{1}{2}\int\e^2\partial_y\Phi_2-\frac{p(p-1)}{2}\sum_{i=1}^n\sigma_i\mu_i\int\e^2|V|^{p-3}V(\partial_y\R_i)\nonumber\\
&\quad+O\bigg(\frac{1}{t^{1+\delta_1}}\|\e\|^2_{\h}+\frac{1}{t^{\frac{9}{4}}}\|\e\|_{\h}\bigg).
\end{align}
Finally, for $W_{2,5}$, by a similar argument as we did for $W_1$, we obtain that
\begin{equation}\label{447}
W_{2,5}=O\bigg(\frac{1}{t^{1+\delta_1}}\|\e\|^2_{\h}+\frac{1}{t^{\frac{9}{4}}}\|\e\|_{\h}\bigg).
\end{equation}
Combining \eqref{443}--\eqref{447}, we have
\begin{align}
\label{448}
W_2=&-\int\big|\D\e\big|^2\partial_y\Phi_2-\frac{1}{2}\int\e^2\partial_y\Phi_2-\frac{p(p-1)}{2}\sum_{i=1}^n\sigma_i\mu_i\int\e^2|V|^{p-3}V(\partial_y\R_i)\nonumber\\
&\quad+O\bigg(\frac{1}{t^{1+\delta_1}}\|\e\|^2_{\h}+\frac{1}{t^{\frac{9}{4}}}\|\e\|_{\h}\bigg).
\end{align}

\noindent\textbf{Step 4:} Estimate of $W_3$. 

Recall that
$$W_3=-\int\partial_tV\big(|V+\e|^{p-1}(V+\e)-|V|^{p-1}V-p|V|^{p-1}\e\big)\Phi_1.$$
From the bootstrap assumptions, we know that $\dot{x}_i\sim \frac{1}{\sqrt{t}}$, $\dot{\mu}_i\sim \frac{1}{t^{\frac{3}{2}}}$. Combining with the definition of $V$ and $\Phi_1$, we have
$$\bigg\|\partial_tV\Phi_1-\sum_{i=1}^n\dot{x}_i\sigma_i\partial_y\R_i\bigg\|_{L^\infty}\lesssim \frac{1}{t^{\frac{3}{2}}}.$$
From the fact that $\R_i>0$ and Taylor's expansion, we have
\begin{align*}
&|V+\e|^{p-1}(V+\e)-|V|^{p-1}V-p|V|^{p-1}\e\\
&=\sum_{i=1}^n\phi_i(y/\sqrt{t})\big(|V+\e|^{p-1}(V+\e)-|V|^{p-1}V-p|V|^{p-1}\e\big)\\
&=\frac{p(p-1)}{2}\sum_{i=1}^n\int_0^1\e^2\phi_i(y/\sqrt{t})|V+s\e|^{p-3}(V+s\e)\,\dd s\\
&=\frac{p(p-1)}{2}\sum_{i=1}^n\int_0^1\e^2\phi_i(y/\sqrt{t})\sigma_i\R_i^{p-2}\biggl|1+\frac{V+s\e-\sigma_i\R_i}{\sigma_i\R_i}\bigg|^{p-3}\bigg(1+\frac{V+s\e-\sigma_i\R_i}{\sigma_i\R_i}\bigg)\,\dd s\\
&=\frac{p(p-1)}{2}\sum_{i=1}^n\e^2\phi_i(y/\sqrt{t})\sigma_i\R_i^{p-2}+O\bigg(|\e|^3+\frac{\e^2}{t}\bigg).
\end{align*}
Combining all of the above estimates, using the definition of $\phi_i$ as well as \eqref{420}, we obtain that
\begin{align}\label{449}
W_3&=\int\bigg(\sum_{i=1}^n\dot{x}_i\sigma_i\partial_y\R_i\bigg)\times\bigg(\frac{p(p-1)}{2}\sum_{i=1}^n\e^2\phi_i(y/\sqrt{t})\sigma_i\R_i^{p-2}\bigg)+O\bigg(\frac{\|\e\|^2_{\h}}{t^{1+\delta_1}}\bigg)\nonumber\\
&=\frac{p(p-1)}{2}\sum_{i=1}^n\sigma_i\mu_i\int\e^2|V|^{p-3}V(\partial_y\R_i)+O\bigg(\frac{\|\e\|^2_{\h}}{t^{1+\delta_1}}\bigg).
\end{align}

\noindent\textbf{Step 5:} Estimate of $W_4$. 

From Lemma \ref{L3}, we have
$$|\partial_t\Phi_1(t,y)|+|\partial_t\Phi_2(t,y)|\lesssim \frac{1}{t^{\frac{3}{2}}},$$
which implies immediately that
\begin{align}
\label{450}
W_4=O\bigg(\frac{1}{t^{1+\delta_1}}\|\e\|^2_{\h}+\frac{1}{t^{\frac{9}{4}}}\|\e\|_{\h}\bigg).
\end{align}

\noindent\textbf{Step 6:} End of the proof of \eqref{48}

Combining \eqref{441} ,\eqref{448}--\eqref{450} with the last conclusion of Lemma \ref{L3}, we have
\begin{align*}
\partial_tW&\geq-\|\e\|^2_{\h}\big\|\partial_y\Phi_1+\partial_y\Phi_2\big\|_{L^\infty}+ O\bigg(\frac{1}{t^{1+\delta_1}}\|\e\|^2_{\h}+\frac{1}{t^{\frac{9}{4}}}\|\e\|_{\h}\bigg)\\
&\gtrsim -\frac{1}{t^{1+\delta_1}}\|\e\|^2_{\h}-\frac{1}{t^{\frac{9}{4}}}\|\e\|_{\h},
\end{align*}
which concludes the proof of \eqref{48}.
\end{proof}

\subsubsection{A topological argument} In this part, we will show that with a suitable choice of $\vec{x}_0(\tin)$, $\vec{\mu}_0(\tin)$ and $\e_0$, the following properties hold:
\begin{lemma}\label{L7}
There exist $\vec{x}_0(\tin)$, $\vec{\mu}_0(\tin)$ and $\e_0\in \h$ satisfying \eqref{453}--\eqref{437}, such that the following properties hold:
\begin{enumerate}
\item \textbf{the energy functional $W(t)$ is coercive: }
there exists $\lambda_0>0$ such that for all $t\in [T_0,\tin]$, we have
\begin{equation}
\label{49}
W(t)\geq \lambda_0\|\e(t)\|_{\h}^2-\frac{t^{-5/2}}{\lambda_0};
\end{equation}
\item \textbf{estimates on the geometrical parameters: }for all $t\in [T_0,\tin]$ and $k=1,\ldots,n$, there holds%
\footnote{Recall that $y_k$ and $z_k$ are defined by \eqref{478}.}
\begin{equation}\label{436}
|y_k(t)|+|z_k(t)|\lesssim \frac{1}{t^{1/4-\delta_0}}.
\end{equation}
\end{enumerate}
\end{lemma}
\begin{remark}
In subcritical and supercritical cases, the proof will be different. Since, in the subcritical case, the orthogonality condition is enough to control the unstable directions. However, in the supercritical cases, we need to control the directions $(\e,Z^\pm)$, which requires a different topological argument.
\end{remark}

\begin{proof}[Proof of Lemma \ref{L7} when $2<p<3$.]
In the subcritical case ($2<p<3$), we first denote by%
\footnote{Recall here $\{\lambda_k\}$ are eigenvalues of the matrix $M$ defined by \eqref{413} and \eqref{414}.} 
$$S_<=\big\{k\in\{1,\ldots,n\}:\lambda_k<5/16\big\},\quad S_>=\big\{k\in\{1,\ldots,n\}:\lambda_k\geq 5/16\big\}.$$
Since $\lambda_k$ is independent of the choice of the universal constant $\delta_0$,  we may easily see that there exists a small enough constant $\delta_0>0$ such that%
\footnote{Here $\omega_k^1$ is the root of $z^2+z-\lambda_k$ in $\mathbb{C}$ such that $\Re \omega_k^1\leq -\frac{1}{2}$.} 
$$S_<=\big\{k\in\{1,\ldots,n\}:\Re\omega_k^1>-\frac{5}{4}+\delta_0\big\},\quad S_>=\big\{k\in\{1,\ldots,n\}:\Re\omega_k^1<-\frac{5}{4}+\delta_0\big\}.$$
Let $S_>=\{k_1,\ldots,k_{n_0}\}$. From the definition of $\x_k, \U_k$ and an argument of implicit function theorem, we know that for all $\vec{v}\in \mathbb{B}=\{x\in \mathbb{R}^{n_0}:\,|x|\leq 1\}$, there exist $\vec{x}_0(\tin)$, $\vec{\mu}_0(\tin)$ and $\e_0$ satisfying \eqref{453}--\eqref{437} such that 
\begin{align}
&\forall k\in S_<,\; \x_k(\tin)=0,\,\U_k(\tin)=0,\label{467}\\
&\big(z_{k_1}(\tin),\ldots,z_{k_{n_0}}(\tin)\big)=\tin^{-\frac{1}{4}+\delta_0}\vec{v},\label{469}\\
&\big(y_{k_1}(\tin),\ldots,y_{k_{n_0}}(\tin)\big)=0,\quad\e(\tin)=0\label{471}.
\end{align}
Recalling from \eqref{434}, we have for all $t\in[T_0,\tin]$ and all $k=1,\ldots, n$
$$\partial_t[t^{-1-\omega^0_k}y_k]=O\bigg(\frac{1}{t^{9/4+\Re\omega_k^0-\delta_0}}\bigg).$$
Since $\Re \omega_k^0\geq-\frac{1}{2}$, we have $9/4+\Re\omega_k^0-\delta_0>1$. Integrating the above inequality from $t$ to $\tin$, using \eqref{467}--\eqref{471}, we have for all $t\in[T_0,\tin]$ and all $k=1,\ldots, n$
\begin{equation}
\label{468}
|y_k(t)|\lesssim t^{1+\Re\omega_k^0}\tin^{-\frac{1}{4}+\delta_0-1-\Re\omega_k^0}+t^{1+\Re\omega_k^0}\int_t^{\tin}\frac{1}{s^{9/4+\Re\omega_k^0-\delta_0}}\,\dd s\lesssim \frac{1}{t^{1/4-\delta_0}}.
\end{equation}
Using \eqref{435} and a similar argument, we can also show that for all $k\in S_<$ and $t\in[T_0,\tin]$, there holds:
\begin{equation}
\label{470}
|z_k(t)|\lesssim \frac{1}{t^{1/4-\delta_0}}.
\end{equation}
Now we denote by%
\footnote{Here we use the fact that $\Re\omega_{k_\ell}^1< -\frac{5}{4}+\delta_0$.}
$$\mathcal{Z}(t)=t^{\frac{1}{2}-2\delta_0}\sum_{\ell=1}^{n_0}|z_{k_\ell}(t)|^2,\quad \delta_2=-\max_{\ell=1,\ldots,n_0}(5/2-2\delta_0+2\Re\omega_{k_\ell}^1)>0.$$
Then from \eqref{435} and the bootstrap assumptions \eqref{415}--\eqref{416}, we have
\begin{align}\label{472}
\partial_t \mathcal{Z}(t)&=\sum_{\ell=1}^{n_0}\bigg(t^{\frac{1}{2}-2\delta_0}\frac{2\Re(1+\omega_{k_\ell}^1)}{t}|z_{k_\ell}(t)|^2+(\frac{1}{2}-2\delta_0)t^{-\frac{1}{2}-2\delta_0}|z_{k_\ell}(t)|^2\bigg)+O\bigg(\frac{1}{t^{\frac{3}{2}-4\delta_0}}\bigg)\nonumber\\
&=\sum_{\ell=1}^{n_0}\frac{5/2-2\delta_0+2\Re\omega_{k_\ell}^1}{t}[t^{1/2-2\delta_0}|z_{k_{\ell}}(t)|^2]+O\bigg(\frac{1}{t^{\frac{3}{2}-4\delta_0}}\bigg)\nonumber\\
&\leq -\frac{\delta_2\mathcal{Z}(t)}{t}+O\bigg(\frac{1}{t^{\frac{3}{2}-4\delta_0}}\bigg).
\end{align}
Let $T_1=T_1(\vec{v})\in[T_0,\tin)$ given by
$$T_1=\inf\{t\in[T_0,\tin):\,\forall\tau\in[t,\tin],\,\mathcal{Z}(\tau)\leq 1\}.$$
Assume that $T_1>T_0$,  then for $t=T_1$, we must have $\mathcal{Z}(T_1)=1$ and
$$\partial_t\mathcal{Z}(T_1)\leq -\frac{\delta_2}{T_1}+O(T_1^{-\frac{3}{2}+4\delta_0})<0,$$
provided that $T_1\geq T_0\gg1$. Since $\big(\vec{x}(t),\vec{\mu}(t),\e(t)\big)$ depends continuously on the initial data, we conclude that the following map
$$\mathcal{M}:\quad \mathbb{B}\ni\vec{v}\longmapsto T_1^{\frac{1}{4}-\delta_0}\big(z_{k_1}(T_1),\ldots,z_{k_{n_0}}(T_1)\big)\in\partial\mathbb{B}.$$
is continuous. Now, for all $\vec{v}\in\partial\mathbb{B}$, from \eqref{467}--\eqref{471} and \eqref{472}, we know that $\mathcal{Z}(\tin)=1$ and $\partial_t\mathcal{Z}(\tin)<0$, which implies that $T_1=\tin$. Then we have $\mathcal{M}\vec{v}=\vec{v}$, for all $\vec{v}\in\partial\mathbb{B}$. That means $\mathcal{M}$ is a continuous map from $\mathbb{B}$ to $\partial\mathbb{B}$, whose restriction on $\partial\mathbb{B}$ is the identity. From Brouwer's fixed point theorem, this is a contradiction. We then obtain that $T_0=T_1$. Together with \eqref{470} and \eqref{468}, we obtain \eqref{436}.

Now we only need to show that with above choice of initial data the estimate \eqref{49} holds. We choose $\widetilde{\phi}_i\in C_0^\infty(\mathbb{R})$ such that
\begin{itemize}
\item $0\leq \widetilde{\phi}_i\leq 1$, and for all $1<i<n$, $\widetilde{\phi_i}=\phi_i$;
\item $\widetilde{\phi}_1(y)=\phi_1$, if $y<(4\alpha_1-\alpha_{2})/3$ and $\widetilde{\phi}_1(y)=0$, if $y>(5\alpha_1-2\alpha_{2})/3$;
\item  $\widetilde{\phi}_n(y)=\phi_n$, if $y>(4\alpha_n-\alpha_{n-1})/3$ and $\widetilde{\phi}_1(y)=0$, if $y<(5\alpha_n-2\alpha_{n-1})/3$;
\end{itemize}
We may easily see that for all $y\in\big([4\alpha_n-\alpha_{n-1}]/3,[4\alpha_1-\alpha_{2}]/3\big)$ and $k\in\mathbb{N}$, we have
$$\sum_{i=1}^n\big(\widetilde{\phi}_i(y)\big)^k=1+o_{H^1}(1).$$
From \eqref{442} and \cite[Lemma 2.15]{MP}, we have%
\footnote{Recall that $\tau_i$ is defined by \eqref{339}.}
\begin{align}
\label{451}
W&=\int\big|\D\e\big|^2+\e^2-p|V|^{p-1}\e^2+o(\|\e\|^2_{\h})\nonumber\\
&=\int\Big(\big|\D\e\big|^2+\e^2\Big)\bigg(1-\sum_{i=1}^n\widetilde{\phi}^2_i\bigg)+\int\Big(\big|\D\e\big|^2+\e^2-p|V|^{p-1}\e^2\Big)\sum_{i=1}^n\widetilde{\phi}^2_i+o(\|\e\|^2_{\h})\nonumber\\
&=\int\Big(\big|\D\e\big|^2+\e^2\Big)\bigg(1-\sum_{i=1}^n\widetilde{\phi}^2_i\bigg)\nonumber\\
&\qquad+\sum_{i=1}^n\int\Big(\big|\D(\e\widetilde{\phi}_i)\big|^2+(\e\widetilde{\phi}_i)^2-p(\tau_iQ)^{p-1}(\e\widetilde{\phi}_i)^2\Big)+o(\|\e\|^2_{\h}).
\end{align}
Now, we can use the coercivity \eqref{226} to obtain that
\begin{align*}
W&\geq \int\Big(\big|\D\e\big|^2+\e^2\Big)\bigg(1-\sum_{i=1}^n\widetilde{\phi}^2_i\bigg)+\lambda_0\sum_{i=1}^n\int \Big(\big|\D(\e\widetilde{\phi}_i)\big|^2+(\e\widetilde{\phi}_i)^2\Big)\nonumber\\
&\quad-\sum_{i=1}^n\frac{1}{\lambda_0}\big[(\e\widetilde{\phi}_i,\tau_iQ)^2+(\e\widetilde{\phi}_i,\tau_iQ')^2\big]+o(\|\e\|^2_{\h})\\
&\geq\frac{\lambda_0}{2}\|\e\|^2_{\h}-\frac{1}{\lambda_0}\big[(\e\widetilde{\phi}_i,\tau_iQ)^2+(\e\widetilde{\phi}_i,\tau_iQ')^2\big].
\end{align*}
From the construction of $\widetilde{\phi}_i$ and the orthogonality condition \eqref{38}, we obtain \eqref{49} immediately.
\end{proof}

While for the supercritical case ($p>3$), we first consider $\e_0\in\h$ of the following form
:
\begin{align}
\e_0=&\sum_{i=1}^nb_i^\pm{Z}^\pm_{i,0}+\sum_{i=1}^nc_i\R_{i,0}+\sum_{i=1}^nd_i\partial_y\R_{i,0}
\end{align}
where%
\footnote{Recall that $Z^\pm$ are defined in Proposition \ref{P3} and \ref{P4}, while $x_{i,0}(\tin)$ and $\mu_{i,0}(\tin)$ are given by \eqref{455}}%
$${Z}^\pm_{i,0}(y)=Z^{\pm}_{1+\mu_{i,0}(\tin)}(y-x_{i,0}(\tin)),\quad \R_{i,0}(y)=Q_{1+\mu_{i,0}(\tin)}(y-x_{i,0}(\tin))$$
and $b_i^\pm, c_i,d_i$ are constants in a small neighborhood of $0$. For suitable choice of $\vec{x}_0(\tin)$ and $\vec{\mu}_0(\tin)$, we consider initial data of the following form:
$$u(\tin,y)=\Theta(\vec{x}_0(\tin),\vec{\mu}_0(\tin),y)+\sum_{i=1}^nb_i^\pm{Z}^\pm_{i,0}(y)+\sum_{i=1}^nc_i\R_{i,0}(y)+\sum_{i=1}^nd_i\partial_y\R_{i,0}(y).$$
Let $\big(\vec{x}(t),\vec{\mu}(t),\e(t)\big)$ be the corresponding geometrical parameters and error terms given by Proposition \ref{P7}. We denote by 
$$a_i^\pm(t)=\int\e(t,y)\widetilde{Z}_i^\pm(t,y)\,\dd y\quad \widetilde{Z}_i^\pm(t,y)=Z^\pm_{1+\mu_i(t)}(y-x_i(t)),$$
for all $i=1,\ldots,n$. 

Then we have the following two lemmas:
\begin{lemma}\label{L8}
There exists universal constants $C_0>0$ such that for all $\vec{a}_{\rm in},\ \vec{x}_{\rm in},\ \vec{\mu}_{\rm in}\in \mathbb{R}^n$ with 
$$|\vec{a}_{\rm in}|\leq \tin^{-3/2},\quad |\vec{x}_{\rm in}|\leq \tin^{-\frac{1}{4}+\delta_0},\quad |\vec{\mu}_{\rm in}|\leq \tin^{-\frac{5}{4}+\delta_0},$$
there exists a unique choice of $x_{i,0}(\tin)$, $\mu_{i,0}(\tin)$ and $b_i^\pm, c_i,d_i$ such that 
\begin{align*}
&\big(a_1^-(\tin),\ldots,a_n^-(\tin)\big)=\vec{a}_{\rm in}, \quad a_i^+(\tin)=0,\\
&\big(x_1(\tin),\ldots,x_n(\tin)\big)=\vec{x}_{\rm in},\quad \big(\mu_1(\tin),\ldots,\mu_n(\tin)\big)=\vec{\mu}_{\rm in},
\end{align*}
and
$$\sum_{i=1}^n\big(|b_i^\pm|^2+c^2_i+d^2_i\big)\leq C_0|\vec{a}_{\rm in}|^2.$$
\end{lemma}
\begin{proof}
Consider a continuous map $\Omega:\mathbb{R}^{6n}\rightarrow \mathbb{R}^{6n}$ as follows:
\begin{align*}
&\mathbb{R}^{6n}\ni\Big(\vec{x}_0(\tin),\ \vec{\mu}_0(\tin),\ \big\{(b_i^\pm,c_i,d_i)\big\}_{i=1}^n\Big)\longmapsto u(\tin)\\
&\quad\longmapsto \Big(\vec{x}(\tin),\ \vec{\mu}(\tin),\ \big\{(\e(\tin),\widetilde{Z}^\pm_{i}(\tin)),(\e(\tin),\R_{i}(\tin)),(\e(\tin),\partial_y\R_{i}(\tin))\big\}_{i=1}^n\Big)\in\mathbb{R}^{6n}.
\end{align*}
From Proposition \ref{P7}, we can easily see that
$$\Omega=
\begin{bmatrix}
{\rm I}_n&&&&\\
&{\rm I}_n&&&\\
&&N & & \\
&&&  \ddots&\\
&&& &N
\end{bmatrix}
+O\big(\tin^{-1/2}\big),$$
where ${\rm I}_n$ is the unit matrix of order $n$ and $N$ is the Gramian matrix of $Z^\pm,Q,Q'$. Since $Q,Z^\pm$ are eigenvectors of $\mathcal{L}\partial_y$ associated to different eigenvalues, we know that they are linearly independent. Moreover, we have $$(Q',Z^\pm)=\pm\frac{1}{e_0}(Q',L\partial_yZ^\pm)=\pm\frac{1}{e_0}(LQ',\partial_yZ^\pm)=0,$$
which implies that $Z^\pm$, $Q$ and $Q'$ are linear independent. Hence, $\Omega$ is invertible around $0$ and $\|\Omega\|\sim1$. We then conclude the proof of Lemma \ref{L8}.
\end{proof}

\begin{lemma}\label{L9}
For all $t\in[T_0,\tin]$, there holds
\begin{equation}
\bigg|\frac{\dd a_i^\pm}{\dd t}(t)\mp e_0(1+\mu_i)^2a_i^\pm(t)\bigg|\lesssim \frac{1}{t^{2}}.
\end{equation}
\end{lemma}
\begin{proof}[Proof of Lemma \ref{L9}]
We recall the equation of $\e$:
\begin{align*}
\partial_t\e=&\partial_y\big[|D|\e+\e-|V+\e|^{p-1}(V+\e)+|V|^{p-1}V\big]-E_V+\sum_{i=1}^nP_{i,1}+\sum_{i=1}^nP_{i,2},
\end{align*}
where
\begin{align*}
&P_{i,1}=(\dot{x}_i-\mu_i)\sigma_i\partial_y\widetilde{R}_i,\\
&P_{i,2}=-\bigg(\dot{\mu}_i+\sum^n_{\substack{j=1,\\j\not=i}}\frac{a_{ij}}{x^3_{ij}}+\sum^n_{\substack{k,j=1,\\j\not=i}}\frac{b_{ijk}\mu_k}{x^3_{ij}}\bigg)\frac{\sigma_i\Lambda\widetilde{R}_i}{1+\mu_i}.
\end{align*}
Then, we have
\begin{align*}
\frac{\dd a_i^\pm}{\dd t}&=\int(\partial_t\e)\widetilde{Z}_i^\pm+\int\e \partial_t \widetilde{Z}_i^\pm\\
&=-\int\big(|D|\e+(1+\mu_i)\e-p|V|^{p-1}\e\big)\partial_y\widetilde{Z}_i^\pm\\
&\quad+\int \big[|V+\e|^{p-1}(V+\e)-|V|^{p-1}V-p|V|^{p-1}\e\big]\partial_y\widetilde{Z}_i^\pm\\
&\quad+\sum_{j=1}^n\sum_{k=1}^2\int P_{j,k}\widetilde{Z}_i^\pm-\int E_V\widetilde{Z}_i^\pm+\dot{\mu}_i\int \e\Lambda\widetilde{Z}_i^\pm-(\dot{x}_i-\mu_i)\int\e \partial_y\widetilde{Z}_i^\pm.
\end{align*}
Since $Z^\pm\in\mathcal{Y}_2$, we have
\begin{align*}
\int\big(|D|\e+(1+\mu_i)\e-p|V|^{p-1}\e\big)\partial_y\widetilde{Z}_i^\pm=\int\e \mathcal{L}_i(\partial_y\widetilde{Z}_i^\pm)+O\bigg(\frac{\|\e(t)\|_{\h}}{d^2(t)}\bigg),
\end{align*}
where $\mathcal{L}_i=|D|+(1+\mu_i)-p\R_i^{p-1}$. Using the fact that
$$\mathcal{L}\partial_yZ^\pm=\pm e_0Z^\pm,$$
we have
$$\int\e\mathcal{L}_i(\partial_y\widetilde{Z}_i^\pm)=\pm e_0(1+\mu_i)^2\int\e\widetilde{Z}_i^\pm=\pm e_0(1+\mu_i)^2a_i^{\pm}.$$
Next, using the bootstrap assumptions \eqref{415}--\eqref{417}, we have
\begin{align*}
&\int \big[|V+\e|^{p-1}(V+\e)-|V|^{p-1}V-p|V|^{p-1}\e\big]\partial_y\widetilde{Z}_i^\pm=O\bigg(\frac{1}{t^{2}}\bigg),\\
&\sum_{j=1}^n\int P_{j,2}\widetilde{Z}_i^\pm-\int E_V\widetilde{Z}_i^\pm+\dot{\mu}_i\int \e\Lambda\widetilde{Z}_i^\pm-(\dot{x}_i-\mu_i)\int\e \partial_y\widetilde{Z}_i^\pm=O\bigg(\frac{1}{t^{2}}\bigg).
\end{align*}
Finally, since $(Z^\pm,Q')=0.$
we have 
$$(P_{j,1},\widetilde{Z}_i^\pm)=O\bigg(\frac{1}{t^2}\bigg).$$
for all $i,j\in\{q,\ldots,n\}$. Now the proof is concluded.
\end{proof}

Now we can give the proof of Lemma \ref{L7} in the case of $p>3$.

\begin{proof}[Proof of Lemma \ref{L7} when $p>3$]
We claim that there exists suitable choice of $\vec{x}_0(\tin)$, $\vec{\mu}_0(\tin)$ and $\e_0$ such that \eqref{436} holds and  for all $t\in[T_0,\tin]$, $i\in\{1,\ldots,n\}$, we have
\begin{equation}
\label{454}
|a_i^\pm(t)|\leq t^{-\frac{3}{2}}.
\end{equation}

From Lemma \ref{L9}, we know that as long as the initial conditions \eqref{452} holds true, we have
\begin{align*}
&|e^{\int_t^{\tin}e_0(1+\mu_i(s))^2\,\dd s}a_i^+(t)|\lesssim \int_t^{\tin} e^{\int_s^{\tin}e_0(1+\mu_i(r))^2\,\dd r}s^{-2}\,\dd s+\tin^{-3/2}.
\end{align*}
Since $|\mu_i|\ll 1$, we then obtain that
\begin{align}\label{457}
&|a_i^+(t)|\lesssim t^{-\frac{3}{2}}+\int_t^{\tin} e^{\int_s^{t}e_0(1+\mu_i(r))^2\,\dd r}s^{-2}\,\dd s\lesssim\int_t^{\tin}e^{\frac{e_0}{2}(t-s)}s^{-2}\,\dd s+t^{-\frac{3}{2}}\nonumber\\
&=\frac{2e^{\frac{e_0t}{2}}}{e_0}\big(e^{-\frac{e_0t}{2}}t^{-2}-e^{-\frac{e_0\tin}{2}}\tin^{-2}\big)-\frac{4e^{\frac{e_0t}{2}}}{e_0}\int_t^{\tin}e^{-\frac{e_0s}{2}}s^{-3}\,\dd s+t^{-\frac{3}{2}}\lesssim t^{-\frac{3}{2}}.
\end{align}

Let $\vec{v}=(\vec{v}_0,\vec{v}_1)\in\mathbb{B}_0\times \mathbb{B}_1$, where
$$\mathbb{B}_0=\{\vec{v}_0\in\mathbb{R}^{n_0}:\,|\vec{v}_0|\leq 1\},\;\;\mathbb{B}_1=\{\vec{v}_1\in\mathbb{R}^{n}:\,|\vec{v}_1|\leq 1\}$$
From Lemma \ref{L8} and the the local uniqueness of the geometrical decomposition, we can choose constants $b_i^\pm, c_i ,d_i$ such that%
\footnote{Here the set $S_<$ and $S_>=\{k_1,\ldots,k_{n_0}\}$ are defined similarly as in the subcritical cases.}
\begin{itemize}
\item for all $k\in S_<$, there holds
\begin{equation}
\label{456}
\x_k(\tin)=0,\quad\U_k(\tin)=0,
\end{equation}
\item for all $\ell=1,\ldots,n_0$, we have
\begin{equation}
\label{473}
\big(y_{k_1}(\tin),\ldots,y_{k_{n_0}}(\tin)\big)=0,\quad\big(z_{k_1}(\tin),\ldots,z_{k_{n_0}}(\tin)\big)=\tin^{-\frac{1}{4}+\delta_0}\vec{v}_0;
\end{equation}
\item there holds 
\begin{equation}
\label{458}
\e(\tin)=\e_0,\;\;\big(a_1^-(\tin),\ldots,a_n^-(\tin)\big)=\vec{a}_{\rm in}=\tin^{-\frac{3}{2}}\vec{v}_1, \;\; a_i^+(\tin)=0;
\end{equation}
\item there holds 
$$\sum_{i=1}^n\big(|b_i^\pm|^2+c^2_i+d^2_i\big)\leq C|\vec{a}_{\rm in}|^2.$$
\end{itemize}
Similarly, we can show that with this choice of initial data, estimates \eqref{468} and \eqref{470} also hold true in the supercritical case. 

Now we denote by
$$\mathcal{N}(t)=t^3\sum_{i=1}^n|a_i^-(t)|^2,\quad \mathcal{Z}(t)=t^{\frac{1}{2}-2\delta_0}\sum_{\ell=1}^{n_0}|z_{k_\ell}(t)|^2.$$
Let $T_1=T_1(\vec{v})\in[T_0,\tin)$ be as follows:
$$T_1=\inf\{t\in[T_0,\tin):\,\text{$\forall\tau\in[t,\tin]$, $\mathcal{N}(\tau)\leq 1$ and $\mathcal{Z}(\tau)\leq 1$}\}.$$
From Lemma \ref{L9}, we have for all $t\in[T_1,\tin]$,
\begin{align*}
\partial_t\mathcal{N}(t)&\leq t^3(3t^{-1}-2e_0)\sum_{i=1}^n|a_i^-(t)|^2+O\Bigg(t\bigg(\sum_{i=1}^n|a_i^-(t)|^2\bigg)^{1/2}\Bigg).
\end{align*}
Assume that $T_1>T_0$,  then we have for $t=T_1$, either $\mathcal{N}(T_1)=1$ or $\mathcal{Z}(T_1)=1$ and
\begin{align}
&\partial_t\mathcal{N}(T_1)\leq -\frac{3e_0}{4}+O(T_1^{-1/2})\leq -e_0<0,\;\;\text{if $\mathcal{N}(T_1)=1$},\label{476}\\
&\partial_t\mathcal{Z}(T_1)\leq -\frac{\delta_2}{T_1}+O\bigg(\frac{1}{T_1^{\frac{3}{2}-4\delta_0}}\bigg)<0,\;\;\text{if $\mathcal{Z}(T_1)=1$}\label{477}.
\end{align}
We consider the following map
$$\mathcal{M}:\quad \mathbb{B}_0\times\mathbb{B}_1\ni\vec{v}\longmapsto (\vec{A},\vec{B})\in \mathbb{B}_0\times\mathbb{B}_1,$$
where 
\begin{align*}
\vec{A}=T_1^{\frac{1}{4}-\delta_0}\big(z_{k_1}(T_1),\ldots,z_{k_{n_0}}(T_1)\big)\in\mathbb{B}_0,\;\;\vec{B}=T_1^{\frac{3}{2}}\big(a_1^-(T_1),\ldots,a_n^-(T_1)\big)\in\mathbb{B}_1.
\end{align*}
Since  $\partial (\mathbb{B}_0\times\mathbb{B}_1)=(\partial\mathbb{B}_0)\times \mathbb{B}_1\cup \mathbb{B}_0\times (\partial\mathbb{B}_1)$, from \eqref{476} and \eqref{477}, we know that $\mathcal{M}$ is actually a continuous map from $\mathbb{B}_0\times\mathbb{B}_1$ to $\partial (\mathbb{B}_0\times\mathbb{B}_1)$. On the other hand, for all $\vec{v}\in\partial(\mathbb{B}_0\times\mathbb{B}_1)$, from \eqref{456}--\eqref{458}, \eqref{476} and \eqref{477}, we know that either $\mathcal{Z}(\tin)=1$, $\partial_t\mathcal{Z}(\tin)<0$ or $\mathcal{N}(\tin)=1$, $\partial_t\mathcal{N}(\tin)<0$, which implies that $T_1=\tin$. Then we have $\mathcal{M}\vec{v}=\vec{v}$, for all $\vec{v}\in\partial(\mathbb{B}_0\times\mathbb{B}_1)$. That means $\mathcal{M}$ is a continuous map from $\mathbb{B}_0\times\mathbb{B}_1$ to $\partial (\mathbb{B}_0\times\mathbb{B}_1)$, whose restriction on $\partial (\mathbb{B}_0\times\mathbb{B}_1)$ is the identity. From Brouwer's fixed point theorem, this is a contradiction. We then conclude that $T_0=T_1$, which implies \eqref{436} immediately. Together with \eqref{457}, we also have
\begin{equation}
\label{459}
|a_i^\pm(t)|\lesssim t^{-3/2}
\end{equation}
for all $i\in\{1,\ldots,n\}$ and $t\in[T_0,\tin]$.
On the other hand, using \eqref{27} and \eqref{451}, we have
\begin{align*}
W&\geq \int\Big(\big|\D\e\big|^2+\e^2\Big)\bigg(1-\sum_{i=1}^n\widetilde{\phi}^2_i\bigg)+\lambda_0\sum_{i=1}^n\int \Big(\big|\D(\e\widetilde{\phi}_i)\big|^2+(\e\widetilde{\phi}_i)^2\Big)\nonumber\\
&\quad-\frac{1}{\lambda_0}\sum_{i=1}^n\big[(\e\widetilde{\phi}_i,\tau_iZ^+)^2+(\e\widetilde{\phi}_i,\tau_iZ^-)^2+(\e\widetilde{\phi}_i,\tau_iQ')^2\big]+o(\|\e\|^2_{\h})\\
&\geq\frac{\lambda_0}{2}\|\e\|^2_{\h}-\frac{1}{\lambda_0}\sum_{i=1}^n\big[(\e\widetilde{\phi}_i,\tau_iZ^+)^2+(\e\widetilde{\phi}_i,\tau_iZ^-)^2+(\e\widetilde{\phi}_i,\tau_iQ')^2\big].
\end{align*}
Together with \eqref{459} and the orthogonality condition \eqref{38}, we obtain \eqref{49} immediately for the supercritical case.
\end{proof}

\subsubsection{End of the proof of Proposition \ref{P9}}
Now, we can conclude the proof of Proposition \ref{P9} by improving the bootstrap assumptions \eqref{415}--\eqref{417}. With our choice of initial data, we have already shown that $W(t)$ is coercive for all $t\in[T_0,\tin]$. Now, we integrating \eqref{48} from $t$ to $\tin$, using the initial condition \eqref{452}, the bootstrap assumptions \eqref{415}--\eqref{417} and the coercivity \eqref{49} to obtain
\begin{align}\label{475}
W(t)&\lesssim \tin^{-3}+t^{-\frac{5}{2}}+\int_t^{\tin}\big(s^{-1-\delta_1-\frac{5}{2}+2\delta_0}+s^{-\frac{9}{4}-\frac{5}{4}+\delta_0}\big)\,\dd s\nonumber\\
&\lesssim t^{-\frac{5}{2}}+t^{-\frac{5}{2}+2\delta_0+\delta_1}+t^{-\frac{5}{2}+\delta_0},
\end{align}
which improves \eqref{417} provided that $t_0$ is chosen large enough. While for the geometrical parameters, we recall from Remark \ref{R2} that we only need to show that for all $k=1,\ldots,n$ and all $t\in[T_0,\tin]$, there holds
\begin{equation}
\label{474}
|\x_k(t)|\lesssim \frac{1}{t^{1/4-\delta_0}},\quad |\U_k(t)|\lesssim \frac{1}{t^{5/4-\delta_0}}.
\end{equation}
Indeed, if $\lambda_k\not=-\frac{1}{4}$, then we have $\omega_k^0\not=\omega_k^1$. We can see that \eqref{474} is a direct consequence of \eqref{436}. If $\lambda_k=-\frac{1}{4}$, then from \eqref{431}, \eqref{432} and \eqref{436}, we have
\begin{align*}
\partial_t[t^{\frac{1}{2}}\U_k]&=\frac{\U_k}{2t^{1/2}}+\frac{\lambda_k\x_k}{t^{3/2}}+O\bigg(\frac{1}{t^{7/4-\delta_0}}\bigg)=\frac{y_k}{2t^{3/2}}+O\bigg(\frac{1}{t^{7/4-\delta_0}}\bigg)=O\bigg(\frac{1}{t^{7/4-\delta_0}}\bigg).
\end{align*}
From the choice of the initial data (see \eqref{467}, \eqref{469}, \eqref{456} and \eqref{473}), we have
$$|t^{\frac{1}{2}}\U_k(t)|\lesssim \frac{1}{t^{3/4-\delta_0}},$$
which together with \eqref{436} implies \eqref{474} immediately.

Now, from \eqref{475}, \eqref{474} and a continuity argument, we conclude that $T_0=t_*$, hence $T_0=t_*=t_0$. Moreover, the uniform backward estimates \eqref{44}-\eqref{46} hold on $[t_0,\tin]$.

\subsection{Construction of multi-soliton solutions}
In this subsection, we finish the proof of Theorem \ref{MT1} by construction a strongly interacting multi-soliton using a compactness argument.

From Proposition \ref{P9}, by choosing $\tin=m$ for all $m\in\mathbb{N}$ large enough, we have a solution $u_m(t)$ defined on $[t_0,m]$ such that \eqref{44}--\eqref{46} hold.
We denote by  
$$U(t,x)=\sum_{i=1}^n\sigma_iQ\big(x-t-\alpha_i\sqrt{t}-\beta_i\log t-\gamma_i\big).$$
Then from \eqref{44}--\eqref{46} and the construction of $V$, we have
\begin{equation}
\label{460}
\|u_m(t,\cdot)-U(t,\cdot)\|_{\h}\lesssim t^{-\frac{1}{4}+\delta_0},
\end{equation}
for all $t\in[t_0,m]$ and $m\in\mathbb{N}$ large enough. 

\begin{lemma}\label{L10}
There exists a sub-sequence of $\{u_m\}$ (still denoted by $u_m$) and $u_0\in\h$, such that $u_m(t_0)$ converges to $u_0$ weakly in $\h$ and strongly in $H^\sigma$, for all $\sigma\in[0,1/2)$, as $m\rightarrow+\infty$.
\end{lemma}

\begin{proof}
Since $\{u_m(t_0)\}$ is uniformly bounded in $\h$ (from \eqref{460}), it suffices to show that for some sub-sequence $u_m$ (still denoted by $u_m$), there holds $u_m(t_0)$ converges to $u_0$ strongly in $L^2$. Since for all finite interval $[a,b]$, $\h([a,b])$ is compactly embedded into $L^2([a,b])$, we only need to show that for all $\delta>0$, there exists $K=K(\delta)>0$, such that
\begin{equation}
\label{461}
\int_{|x|>K}|u_m(t_0,x)|^2\,\dd x<\delta.
\end{equation} 
From \eqref{44}--\eqref{46} and \eqref{460}, we know that there exists a $t_1=t_1(\delta)>t_0$ independent of $m$, such that for all $m>t_1$, there holds 
$$\int|u_m(t_1,x)-U(t_1,x)|^2\,\dd x\leq C t_1^{-\frac{1}{4}+\delta_0}<\frac{\delta}{10}.$$
Now for this fixed $t_1$, there exits $K_1=K_1(\delta)>0$ such that
$$\int_{|x|>K_1}|U(t_1,x)|^2\,\dd x<\frac{\delta}{10}.$$
Now, we choose a smooth function $\eta$, such that $\eta(y)=0$, if $|y|<\frac{1}{2}$; $\eta(y)=1$ if $|y|>1$ and $0\leq \eta\leq 1$. Let
$$\chi(x)=\eta\bigg(\frac{x}{R_0K_1}\bigg),\quad R_0>1.$$
Then we have
\begin{align*}
&\frac{1}{2}\frac{\dd}{\dd t}\int |u_m(t,x)|^2\chi(x)\,\dd x\\
&=-\int |D|u_m(\partial_xu_m)\chi-\int|D|u_mu_m\chi'+\frac{1}{p+1}\int|u_m|^{p+1}\chi'\\
&=O\bigg(\frac{\|u_m\|^2_{\h}+\|u_m\|^{p+1}_{\h}}{R_0K_2}\bigg)\leq \frac{\delta}{10(t_1-t_0)},
\end{align*}
provided that $R_0=R_0(\delta)$ is large enough. By integrating from $t_0$ to $t_1$, we have
\begin{align*}
&\int_{|x|>R_0K_1}|u(t_0,x)|^2\,\dd x\leq \int|u(t_0,x)|^2\chi(x)\,\dd x\\
&\leq \int|u(t_1,x)|^2\chi(x)\,\dd x+\int_{t_0}^{t_1}\bigg|\frac{\dd}{\dd t}\int |u_m(s,x)|^2\chi(x)\,\dd x\bigg|\,\dd s\\
&\leq \int_{|x|>K_1}|u(t_1,x)|^2\,\dd x+\frac{\delta}{10}<\delta.
\end{align*}
By choosing $K=R_0K_1$, we conclude the proof of \eqref{461}, hence the proof of Lemma \ref{L10}.
\end{proof}

Finally, we consider solutions $u(t)$ of \eqref{CP} with initial data $u(t_0)=u_0$ as  in Lemma \ref{L10}. From \eqref{460} and the local wellposedness of \eqref{CP} in $\h$, we have
$$\|u(t)-U(t)\|_{\h}\leq t^{-\frac{1}{4}+\delta_0},$$
for all $t\geq t_0$. Now we conclude the proof of Theorem \ref{MT1}.

\section{Uniqueness in the supercritical case}\label{S5}
In this section, we will show that in the case of $p>3$ and $n=2$, for strongly interacting two-soliton of the form \eqref{17}, the sign of the two bubbles are always the same and the relative velocity between the two bubbles are unique.

\subsection{An alternative geometrical decomposition}
Assume $p>3$ and $n=2$. Let $u\in \mathcal{C}([t_0,+\infty),\h)$ be a solution of the form \eqref{17}. Without loss of generality we assume that $\sigma_1=1$ and $\sigma_2=\sigma\in\{\pm1\}$. We also assume that  $x_1(t),x_2(t)$
are $C^1$-functions defined on $[t_0,+\infty)$ satisfying
\begin{equation}
\label{533}
x_1(t)-x_2(t)\rightarrow+\infty,
\end{equation}
as $t\rightarrow+\infty$.

We consider
$$V_0(t,y)=\Theta(\vec{\mathfrak{\mathfrak{q}}}(t),0,y),\quad \vec{\mathfrak{q}}(t)=(\mathfrak{q}_1(t),\mathfrak{q}_2(t)),$$
for some $C^1$-functions defined on $[t_0,+\infty)$. Then from the proof of Proposition \ref{P6} and Proposition \ref{P11}, we obtain that
\begin{lemma}\label{L13}
Let $\Psi_{V_0}=\partial_tV_0-\partial_x(|D|V_0+V_0-|V_0|^{p-1}V_0)$, and
$$R_k(t,y)=Q(y-\mathfrak{q}_k(t)),$$
for $k=1,2$. Then we have
\begin{equation}\label{58}
\Psi_{V_0}=\mathcal{E}_{V_0}-\sum_{i=1}^2\dot{\mathfrak{q}}_i\sigma_i\partial_yR_i+\sum_{\substack{i,j=1,\\j\not=i}}^2\frac{a_{ij}\sigma_i}{(\mathfrak{q}_i-\mathfrak{q}_j)^3}\Lambda R_i+\sum_{i=1}^2\dot{\mathfrak{q}}_i(\mathcal{P}_i+\mathcal{Q}_i)
\end{equation}
where
\begin{equation}\label{59}
\quad\mathcal{Q}_i=O_{\h}\bigg(\frac{1}{(\mathfrak{q}_1-\mathfrak{q}_2)^3}\bigg), \quad\mathcal{E}_{V_0}=O_{\h}\bigg(\frac{1}{(\mathfrak{q}_1-\mathfrak{q}_2)^{9/2}}\bigg),
\end{equation}
and 
\begin{equation}
\label{518}
\mathcal{P}_i(t,y)=\frac{\mathcal{R}_i(y-\mathfrak{q}_i(t))}{[\mathfrak{q}_1(t)-\mathfrak{q}_2(t)]^2}
\end{equation} 
for some odd functions $\mathcal{R}_i\in\mathcal{Y}_2$.
\end{lemma}

Similar as Proposition \ref{P7}, we have the following alternative geometrical decomposition:
\begin{proposition}\label{P13}
For $T_0\geq t_0$ large enough, and $\rho\gg1$ large enough, there exist $C^1$-functions $\mathfrak{q}_1,\mathfrak{q}_2\in \mathcal{C}^1([T_0,+\infty))$, and $\epsilon\in ([T_0,+\infty),\h)$ such that the following properties hold.
\begin{enumerate}
\item Geometrical decomposition:
\begin{equation}
\label{51}
u(t,y+t)=\Theta(\vec{\mathfrak{q}}(t),0,y)+\epsilon(t,y),
\end{equation}
\item Orthogonality conditions:
\begin{equation}
\label{52}
\big(\mathcal{Z}_1(t),\epsilon(t)\big)=\big(\mathcal{Z}_2(t),\epsilon(t)\big)=0,
\end{equation}
where
$$\mathcal{Z}_k(t,y)=\mathcal{Z}\big(y-\mathfrak{q}_k(t)\big),\quad \mathcal{Z}(y):=\eta\bigg(\frac{y}{\rho}\bigg)\int_0^y\Lambda Q(y')\,\dd y',$$
for $k=1,2$. Here $\eta$ is a smooth function such that $\eta(y)=1$, if $|y|<\frac{1}{2}$; $\eta(y)=0$ if $|y|>1$ and $0\leq \eta\leq 1$.
\item There holds
\begin{equation}
\label{54}
\mathfrak{q}_i(t)-x_i(t)\rightarrow0,\quad \epsilon(t)\xrightarrow{\h} 0
\end{equation}
as $t\rightarrow+\infty$, for all $i=1,2$.
\end{enumerate}
\end{proposition}

\begin{remark}
We mention here since $p>3$, there holds 
$$(Q',\mathcal{Z})=\int_{\mathbb{R}}Q'(y)\bigg(\int_0^{y}\Lambda Q(y')\,\dd y'\bigg)\,\dd y+O(1/\rho)=(Q,\Lambda Q)+O(1/\rho)\not=0,$$
if $\rho$ is large enough. Together with \eqref{17}, we can prove Proposition \ref{P13} by using implicit function theorem.
\end{remark}
\begin{remark}
Since $(Q',\mathcal{Z})\not=0$, we know that the conclusion of Proposition \ref{P5} can be applied to $\mathcal{Z}$.
\end{remark}
\begin{remark}
We mention here the parameters $\mathfrak{q}_1,\mathfrak{q}_2$ depend on the choice of $\rho$, but they are independent on the choice of the time $T_0$.
\end{remark}

From \cite[Lemma 2.15]{MP} and a standard localization argument, we have
\begin{lemma}\label{L11}
For all $c_0>0$, there exists $\rho_0>0$, such that if $\rho>\rho_0$, then 
\begin{equation}
\label{53}
\|\mathcal{L}(\partial_y\mathcal{Z})+Q\|_{\h}\leq c_0.
\end{equation}
\end{lemma}

\subsection{Modulation estimates}
Recall that $\psi$ is a smooth function such that $\psi(y)=1$ if $y>-1$; $\psi(y)=0$ if $y<-2$, and $\psi'\geq0$. We denote by 
$$\psi_1(t,y)=\psi\bigg(\frac{3[y-\mathfrak{q}_1(t)]}{\mathfrak{q}_1(t)-\mathfrak{q}_2(t)}\bigg),\quad \psi_2(t,y)=1-\psi_1(t,y).$$
We also define 
$$\mathfrak{p}_k(t)=\big(\sigma_kR_k(t),\epsilon(t)\big)+\frac{1}{2}\int_{\mathbb{R}}\psi_k(t)\epsilon^2(t),\quad \forall k=1,2.$$
We claim that
\begin{lemma}\label{L12}
For all $c>0$, there exists $T_1\geq t_0$ and $\rho_0\gg1$, such that for all $k=1,2$, $t\geq T_1$ and $\rho\geq \rho_0$, there holds 
\begin{align}
&|\dot{\mathfrak{q}}_k-(Q,\Lambda Q)^{-1}\mathfrak{p}_k|\leq c\bigg(\frac{1}{(\mathfrak{q}_1-\mathfrak{q}_2)^4}+\|\epsilon\|^2_{\h}\bigg)^{\frac{1}{2}},\quad \forall k=1,2,\label{55}\\
&\bigg|\dot{\mathfrak{p}}_1-\frac{2\kappa_0\sigma\int Q^p}{(\mathfrak{q}_1-\mathfrak{q}_2)^3}\bigg|\leq c\bigg(\frac{1}{(\mathfrak{q}_1-\mathfrak{q}_2)^4}+\|\epsilon\|^2_{\h}\bigg),\label{56}\\
&\bigg|\dot{\mathfrak{p}}_2+\frac{2\kappa_0\sigma\int Q^p}{(\mathfrak{q}_1-\mathfrak{q}_2)^3}\bigg|\leq c\bigg(\frac{1}{(\mathfrak{q}_1-\mathfrak{q}_2)^4}+\|\epsilon\|^2_{\h}\bigg)\label{57}.
\end{align}
\end{lemma}
\begin{proof}
We first prove \eqref{55}. From \eqref{54}, we know that for $t\geq T_1$ large enough, 
$$\|\epsilon(t)\|_{\h}\ll1,\quad \mathfrak{q}_1(t)-\mathfrak{q}_2(t)\gg1.$$
By differentiating the orthogonality condition \eqref{52}, using the equation of $\epsilon$
\begin{align}\label{516}
&\partial_t\epsilon-\partial_y(|D|\epsilon+\epsilon-p|V_0|^{p-1}\epsilon)\nonumber\\
&\qquad\qquad+\partial_y\big[|V_0+\epsilon|^{p-1}(V+\epsilon)-|V_0|^{p-1}V_0-p|V_0|^{p-1}\epsilon\big]+\Psi_{V_0}=0,
\end{align} 
we obtain that
\begin{align}\label{510}
&\dot{\mathfrak{q}}_1(\epsilon,\partial_y\mathcal{Z}_1)-\big(\partial_y(\mathcal{L}_1\epsilon),\mathcal{Z}_1\big)-p\big([|V_0|^{p-1}-R_1^{p-1}]\epsilon,\partial_y\mathcal{Z}_1\big)\nonumber\\
&=\Big(\big[|V_0+\epsilon|^{p-1}(V_0+\epsilon)-|V_0|^{p-1}V_0-p|V_0|^{p-1}\epsilon\big],\partial_y\mathcal{Z}_1\Big)-(\Psi_{V_0},\mathcal{Z}_1),
\end{align}
where $\mathcal{L}_1=|D|+1-pR_1^{p-1}$. Since $\mathfrak{p}_k=O(\|\epsilon\|_{\h})$ and $\|\partial_y\mathcal{Z}_1\|_{\h}\sim 1$, we have
\begin{align}\label{511}
\dot{\mathfrak{q}}_1(\epsilon,\partial_y\mathcal{Z}_1)=O\bigg(\sum_{k=1}^2|\dot{\mathfrak{q}}_k-(Q,\Lambda Q)^{-1}\mathfrak{p}_k|\|\epsilon\|_{\h}+\|\epsilon\|^2_{\h}\bigg).
\end{align}
Using \eqref{53}, we have for all $c_0>0$, if $t\geq T_1$, $\rho\geq\rho_0$, there holds%
\footnote{Recall that $\delta(c_0)$ is defined by \eqref{117}.}
\begin{align}\label{512}
\big(\partial_y(\mathcal{L}_1\epsilon),\mathcal{Z}_1\big)=(\epsilon,R_1)+\delta(c_0)\|\epsilon\|_{\h}=\sigma_1\mathfrak{p}_1+O\big(\delta(c_0)\|\epsilon\|_{\h}\big).
\end{align}
From the construction of $V_0$, we know that
\begin{align}
&\big|\big([|V_0|^{p-1}-R_1^{p-1}]\epsilon,\partial_y\mathcal{Z}_1\big)\big|\lesssim\frac{ \|\epsilon\|_{\h}}{(\mathfrak{q}_1-\mathfrak{q}_2)^2}\lesssim\delta(c_0)\|\epsilon\|_{\h},\label{513}\\
&\Big|\Big(\big[|V_0+\epsilon|^{p-1}(V_0+\epsilon)-|V_0|^{p-1}V_0-p|V_0|^{p-1}\epsilon\big],\partial_y\mathcal{Z}_1\Big)\Big|\lesssim\|\epsilon\|_{\h}^2\lesssim\delta(c_0)\|\epsilon\|_{\h}.\label{514}
\end{align}
By choosing $\rho\geq \rho_0$ large enough, we have
$$(\partial_yR_1,\mathcal{Z}_1)=-(Q,\Lambda Q)+O(1/\rho).$$
Together with  \eqref{58}--\eqref{59}, as well as the construction of $V_0$ we have
\begin{align}\label{515}
(\Psi_{V_0},\mathcal{Z}_1)=\dot{\mathfrak{q}}_1(Q,\Lambda Q)+O(1/\mathfrak{q}^3+\|\epsilon\|^2_{\h})+\sum_{i=1}^2O\big(\delta(c_0)|\dot{\mathfrak{q}}_k-(Q,\Lambda Q)^{-1}\mathfrak{p}_k|\big)
\end{align}
Combining \eqref{510}--\eqref{515}, we have
\begin{align*}
&|\dot{\mathfrak{q}}_1-(Q,\Lambda Q)^{-1}\mathfrak{p}_1|\lesssim\delta(c_0)\bigg(\frac{1}{(\mathfrak{q}_1-\mathfrak{q}_2)^4}+\|\epsilon\|^2_{\h}\bigg)^{\frac{1}{2}}+\sum_{k=1}^2O\big(\delta(c_0)|\dot{\mathfrak{q}}_k-(Q,\Lambda Q)^{-1}\mathfrak{p}_k|\big)
\end{align*}
Similarly, we have
\begin{align*}
&|\dot{\mathfrak{q}}_2-(Q,\Lambda Q)^{-1}\mathfrak{p}_2|\lesssim\delta(c_0)\bigg(\frac{1}{(\mathfrak{q}_1-\mathfrak{q}_2)^4}+\|\epsilon\|^2_{\h}\bigg)^{\frac{1}{2}}+\sum_{k=1}^2O\big(\delta(c_0)|\dot{\mathfrak{q}}_k-(Q,\Lambda Q)^{-1}\mathfrak{p}_k|\big)
\end{align*}
By choosing $c_0$ small enough, we conclude the proof of \eqref{55}.

Now we give the proof of \eqref{56} and \eqref{57}. By direct computation, we have
\begin{align}\label{517}
\dot{\mathfrak{p}}_1=(\partial_t\epsilon,R_1)-\dot{\mathfrak{q}_1}(\epsilon,\partial_yR_1)+\frac{1}{2}\int\partial_t\psi_1\epsilon^2+\int\psi_1\partial_t\epsilon\epsilon.
\end{align}
From \eqref{516}, we have
\begin{align*}
&(\partial_t\epsilon,R_1)=\big(\partial_y(\mathcal{L}_1\epsilon),R_1\big)+p\big([|V_0|^{p-1}-R_1^{p-1}]\epsilon,\partial_yR_1\big)\\
&+\bigg(\Big[|V_0+\epsilon|^{p-1}(V_0+\epsilon)-|V_0|^{p-1}V_0-p|V_0|^{p-1}\epsilon-\frac{p(p-1)}{2}|V_0|^{p-3}V_0\epsilon^2\Big],\partial_yR_1\bigg)\\
&+\frac{p(p-1)}{2}(R_1^{p-2}\epsilon^2,\partial_yR_1)+\frac{p(p-1)}{2}(|V_0|^{p-3}V_0-R_1^{p-2},\epsilon^2\partial_yR_1)-(\Psi_{V_0},R_1).
\end{align*}
We argue similarly as \eqref{510} using the fact that $\mathcal{L}_1(\partial_yR_1)=0$ and \eqref{518} to obtain that for all for all $c_0>0$, if $t\geq T_1$, $\rho\geq\rho_0$, there holds
\begin{align}\label{519}
(\partial_t\epsilon,R_1)=&\frac{p(p-1)}{2}(R_1^{p-2}\epsilon^2,\partial_yR_1)-\frac{a_{12}(Q,\Lambda Q)}{(\mathfrak{q}_1-\mathfrak{q}_2)^3}+O\Bigg[\delta(c_0)\bigg(\frac{1}{(\mathfrak{q}_1-\mathfrak{q}_2)^4}+\|\epsilon\|^2_{\h}\bigg)\Bigg].
\end{align}
Since 
$$|\dot{\mathfrak{q}}_k|\lesssim\bigg(\frac{1}{(\mathfrak{q}_1-\mathfrak{q}_2)^4}+\|\epsilon\|^2_{\h}\bigg)^{\frac{1}{2}},$$
we have $\|\partial_t\psi_1\|_{L^\infty}\lesssim \delta(c_0)$, which implies that
\begin{equation}
\label{520}
\int\partial_t\psi_1\epsilon^2=O\Bigg[\delta(c_0)\bigg(\frac{1}{(\mathfrak{q}_1-\mathfrak{q}_2)^4}+\|\epsilon\|^2_{\h}\bigg)\Bigg].
\end{equation}
Finally, for $\int\psi_1\partial_t\epsilon\epsilon$, using the fact that 
$$\text{Supp }\psi_1\subset\bigg[\mathfrak{q}_1-\frac{2(\mathfrak{q}_1-\mathfrak{q}_2)}{3},+\infty\bigg),$$
and \eqref{516}, we have
\begin{align*}
\int\psi_1\partial_t\epsilon\epsilon=&-\int|D|\epsilon\epsilon_y\psi_1-\int|D|\epsilon\epsilon(\partial_y\psi_1)+p(R_1^{p-1}\epsilon,\partial_y\epsilon)+\dot{\mathfrak{q}_1}(\epsilon,\partial_yR_1)\\
&+O\Bigg[\delta(c_0)\bigg(\frac{1}{(\mathfrak{q}_1-\mathfrak{q}_2)^4}+\|\epsilon\|^2_{\h}\bigg)\Bigg].
\end{align*}
From \cite[Lemma 2.16]{MP}, we have
\begin{align*}
\bigg|\int|D|\epsilon\epsilon_y\psi_1\bigg|\lesssim \|\epsilon\|^2_{L^2}\|\partial_{yy}\psi_1\|_{L^\infty}\lesssim \delta(c_0)\bigg(\frac{1}{(\mathfrak{q}_1-\mathfrak{q}_2)^4}+\|\epsilon\|^2_{\h}\bigg),
\end{align*}
and
\begin{align*}
\bigg|\int|D|\epsilon\epsilon(\partial_y\psi_1)\bigg|&\lesssim \|\epsilon\|^2_{\h}\big(\|\partial_y\psi_1\|_{L^\infty}+\|\partial_{yy}\psi_1\|^{\frac{3}{4}}_{L^2}\|\partial_y\psi_1\|_{L^2}^{\frac{1}{4}}\big)\lesssim \delta(c_0)\bigg(\frac{1}{(\mathfrak{q}_1-\mathfrak{q}_2)^4}+\|\epsilon\|^2_{\h}\bigg).
\end{align*}
Therefore, we have 
\begin{align}
\label{521}
\int\psi_1\partial_t\epsilon\epsilon=&-\frac{p(p-1)}{2}(R_1^{p-2}\epsilon^2,\partial_yR_1)+\dot{\mathfrak{q}_1}(\epsilon,\partial_yR_1)+O\Bigg[\delta(c_0)\bigg(\frac{1}{(\mathfrak{q}_1-\mathfrak{q}_2)^4}+\|\epsilon\|^2_{\h}\bigg)\Bigg].
\end{align}
Combining \eqref{517}--\eqref{521} and \eqref{320}, we obtain \eqref{56}, by choosing $c_0$ sufficiently small. The proof of \eqref{57} is similar.
\end{proof}
\subsection{Stable and unstable directions}
We define
$$Z_i^\pm(t,y)=Z^\pm(y-\mathfrak{q}_i(t)),$$
for $i=1,2$.
We also define 
$$\mathfrak{a}_i^\pm(t)=(\epsilon(t),Z_i^\pm(t)).$$
\begin{lemma}\label{L14}
For all $c>0$, there exist $T_0\geq t_0$ and $\rho_0\gg1$, such that for all $i=1,2$, $t\geq T_0$ and $\rho\geq \rho_0$, there holds 
\begin{equation}
\label{522}
\bigg|\frac{\dd}{\dd t}\mathfrak{a}_i^\pm\pm e_0\mathfrak{a}_i^\pm\bigg|\leq c\bigg(\frac{1}{(\mathfrak{q}_1-\mathfrak{q}_2)^4}+\|\epsilon\|^2_{\h}\bigg)^{\frac{1}{2}}.
\end{equation}
\end{lemma}
The proof of Lemma \ref{L14} is similar to the proof of Lemma \ref{L9}, by using the equation \eqref{516} and the fact that $Z^\pm$ are eigenfunctions of $\mathcal{L}\partial_y$.%
\footnote{Here since the scaling parameters $\mu_i$ is not considered, the computation is actually simpler.}
 We omit the routine details here.

\subsection{Uniqueness in the supercritical case}
In this subsection, we will complete the proof of Theorem \ref{MT2}.
\begin{proposition}\label{P14}
The sign $\sigma=1$ and there exist $C_0>0$, $T_0\geq t_0$, such that for all $t\geq T_0$, there holds
\begin{equation}\label{523}
\|\epsilon(t)\|_{\h}^2\leq \frac{C_0}{[\mathfrak{q}_1(t)-\mathfrak{q}_2(t)]^4}.
\end{equation}
\end{proposition}
We need the following two lemmas:

\begin{lemma}\label{L15}
For all $c>0$ and sufficiently large $T_0\geq t_0$, there exists $T_1\geq T_0$ such that 
\begin{equation}
\label{524}
\sum_{i=1}^2[\mathfrak{a}_i^-(T_1)]^2\leq c\bigg(\frac{1}{[\mathfrak{q}_1(T_1)-\mathfrak{q}_2(T_1)]^4}+\|\epsilon(T_1)\|^2_{\h}\bigg).
\end{equation}
\end{lemma}
\begin{proof}
We denote by 
$$N_1(t)=\sum_{i=1}^2[\mathfrak{a}_i^-(t)]^2,$$
and suppose that for all $t\geq T_0$, there holds
\begin{equation}
\label{532}
N_1(t)\geq \frac{1}{[\mathfrak{q}_1(t)-\mathfrak{q}_2(t)]^4}+\|\epsilon(t)\|^2_{\h}.
\end{equation}
From Lemma \ref{L14} and \eqref{532}, if $T_0$ is large enough, we have 
$$|\dot{N}_1(t)+2e_0N_1(t)|\leq ce_0\bigg( \frac{1}{[\mathfrak{q}_1(t)-\mathfrak{q}_2(t)]^4}+\|\epsilon(t)\|^2_{\h}\bigg)\leq ce_0N_1(t).$$
After integration, we have
$$N_1(t)\leq e^{-e_0(t-T_0)}N_1(T_0.)$$
Together with \eqref{532}, we have $\mathfrak{q}_1(t)-\mathfrak{q}_2(t)\gtrsim e^{t/100}$, as $t\rightarrow+\infty$. But, this contradicts with \eqref{533} and \eqref{54}. We then conclude the proof of \eqref{524}.
\end{proof}

\begin{lemma}\label{L16}
Let $c_0>0$. There exist $C_0>0$ and $T_0\geq t_0$ such that for all $t\geq T_0$, there holds
\begin{align}
&\sum_{i=1}^2[\mathfrak{a}_i^+(t)]^2\leq c_0\,\sup_{\tau\geq t}\bigg(\frac{1}{[\mathfrak{q}_1(\tau)-\mathfrak{q}_2(\tau)]^4}+\sum_{i=1}^2[\mathfrak{a}_i^-(\tau)]^2\bigg),\label{528}\\
&\|\epsilon(t)\|_{H^1}^2\leq C_0\,\sup_{\tau\geq t}\bigg(\frac{1}{[\mathfrak{q}_1(\tau)-\mathfrak{q}_2(\tau)]^4}+\sum_{i=1}^2[\mathfrak{a}_i^-(\tau)]^2\bigg)\label{529}.
\end{align}
\end{lemma}
\begin{proof}
Let $t\geq T_0$, and $T_1\geq t$ such that
\begin{equation}
\label{534}
\frac{1}{[\mathfrak{q}_1(T_1)-\mathfrak{q}_2(T_1)]^4}+\|\epsilon(T_1)\|^2_{\h}=\sup_{\tau\geq t}\bigg(\frac{1}{[\mathfrak{q}_1(\tau)-\mathfrak{q}_2(\tau)]^4}+\|\epsilon(\tau)\|_{\h}^2\bigg).
\end{equation}
We first show that for all $c>0$, if $T_0$ is chosen large enough, then
\begin{equation}
\label{535}
\sum_{i=1}^2[\mathfrak{a}_i^+(T_1)]^2\leq c\bigg(\frac{1}{[\mathfrak{q}_1(T_1)-\mathfrak{q}_2(T_1)]^4}+\|\epsilon(T_1)\|_{\h}^2\bigg).
\end{equation}
For $t\geq T_0$, we denote by 
$$N_2(t)=\sum_{i=1}^2[\mathfrak{a}_i^+(t)]^2,$$
and assume by contradiction that \eqref{535} does not hold. We let $T_2=\sup\{t\geq T_0:\,N_2(t)\geq N(t_1)\}$. We must have $T_2<+\infty$, since $N(t)\rightarrow0$, as $t\rightarrow+\infty$. We also have $\partial_t{N}_2(T_2)\leq 0$ and 
\begin{align*}
N_2(T_2)&\geq N_2(T_1)\geq c\bigg(\frac{1}{[\mathfrak{q}_1(T_1)-\mathfrak{q}_2(T_1)]^4}+\|\epsilon(T_1)\|_{\h}^2\bigg)\\
&\geq c\bigg(\frac{1}{[\mathfrak{q}_1(T_2)-\mathfrak{q}_2(T_2)]^4}+\|\epsilon(T_2)\|_{\h}^2\bigg).
\end{align*}
Hence, we have
$$-\partial_t{N}_2(T_2)+2e_0N_2(T_2)\geq 2e_0c\bigg(\frac{1}{[\mathfrak{q}_1(T_2)-\mathfrak{q}_2(T_2)]^4}+\|\epsilon(T_2)\|_{\h}^2\bigg),$$
which contradicts with Lemma \ref{L14}. Hence, \eqref{535} holds.

Now, from Proposition \ref{P5}, we have
\begin{align*}
&\frac{1}{[\mathfrak{q}_1(T_1)-\mathfrak{q}_2(T_1)]^4}+\|\epsilon(T_1)\|_{\h}^2\leq \frac{C_0}{2}\bigg(\frac{1}{[\mathfrak{q}_1(T_1)-\mathfrak{q}_2(T_1)]^4}+\sum_{i=1}^2[\mathfrak{a}_i^-(T_1)]^2+\sum_{i=1}^2[\mathfrak{a}_i^+(T_1)]^2\bigg).
\end{align*}
By letting $c=1/C_0$ in \eqref{535}, using \eqref{534}, we obtain \eqref{529}. By letting $c=c_0/C_0$, and using \eqref{529} and \eqref{534}, we obtain \eqref{528}.
\end{proof}

\begin{proof}[Proof of Proposition \ref{P14}]
We first show that 
\begin{equation}
\label{536}
\|\epsilon(t)\|_{\h}^2\leq C_0\sup_{\tau\geq t}\frac{1}{[\mathfrak{q}_1(\tau)-\mathfrak{q}_2(\tau)]^4}.
\end{equation}
Let $c_0>0$ and $c=\frac{c_0}{2(C_0+1)}$, where $C_0$ is the constant introduced in Lemma \ref{L16}. From Lemma \ref{L15}, there exists $T_1$ arbitrarily large such that
\begin{equation}
\label{537}
N_1(T_1)\leq \frac{c_0}{2(C_0+1)}\bigg(\frac{C_0+1}{[\mathfrak{q}_1(T_1)-\mathfrak{q}_2(T_1)]^4}+C_0N_1(T_1)\bigg).
\end{equation}
We denote by
$$\widetilde{N}_1(t)=\sup_{\tau\geq t}N_1(\tau),\quad N_3(t)=\frac{1}{[\mathfrak{q}_1(t)-\mathfrak{q}_2(t)]^4},\quad \widetilde{N}_3(t)=\sup_{\tau\geq t}N_3(\tau).$$
We will show that
\begin{equation}
\label{538}
\widetilde{N}_1(t)\leq c_0\widetilde{N}_3(t),
\end{equation}
which implies \eqref{536} immediately. It is easy to see that the function $\widetilde{N}_3(t)$ is increasing and locally Lipschtiz. Moreover,
\begin{equation}
\label{539}
|\partial_t\widetilde{N}_3|\leq \frac{|\dot{\mathfrak{q}}_1|+|\dot{\mathfrak{q}}_2|}{(\mathfrak{q}_1-\mathfrak{q}_2)^4}\ll \frac{1}{(\mathfrak{q}_1-\mathfrak{q}_2)^4}=N_3.
\end{equation}
We claim that for $t$ sufficiently large and $N_1(t)=\widetilde{N}_1(t)$, then 
\begin{equation}
\label{540}
N_1(t)\geq \widetilde{N}_3(t)\implies \partial_tN_1(t)\leq -e_0N_1(t).
\end{equation}
Indeed, from Lemma \ref{L16}, we have 
$$\|\epsilon(t)\|_{\h}^2\leq C_0[\widetilde{N}_1(t)+\widetilde{N}_3(t)]\leq 2C_0N_1(t).$$
Together with Lemma \ref{L14}, we have for all $c>0$, there exists $T_1\geq T_0$ such that 
$$|\partial_tN_1(t)+2e_0N_1(t)|\leq c(\|\epsilon(t)\|^2_{\h}+N_3(t))\leq c(2C_0+1)N_1(t),$$
which implies \eqref{540} by taking $c\leq \frac{e_0}{2C_0+1}$. Now, suppose that \eqref{538} does not hold, and $T_2>T_1$ such that $\widetilde{N}_1(T_2)>\widetilde{N}_3(T_2)$. Without loss of generality, we can assume that $\widetilde{N}_1(T_2)\geq N_1(T_2)$. Let 
$$T_3=\inf\Big\{t\in[T_1,T_2]:\,\partial_tN_1(t)\leq -\frac{e_0}{2}N_1(t),\,\text{for all }\tau\in[t,T_2]\Big\}.$$
By \eqref{540} we have $T_3<T_2$. Suppose $T_3>T_1$. By \eqref{539}, we have $\partial_t\widetilde{N}_3(t)\geq -\frac{e_0}{4}\widetilde{N}_3(t)$, for $t\in[T_3,T_2]$ (provided that $T_1$ is large enough). Since $$N_1(T_2)>\widetilde{N}_3(T_2)$$
and $\partial_t N_1(t)<0$ for all $t\in[T_3,T_2]$, we obtain that $N_1(T_3)<\widetilde{N}_3(T_3)$. The function $N_1(t)$ is strictly decreasing for $t\in[T_3,T_2]$, so we have $N_1(T_3)=\widetilde{N}_1(T_3)$. Together with \eqref{540}, we have $T_3=T_1$. In particular, we have $N_1(T_1)=\widetilde{N}_1(T_1)>\widetilde{N}_3(T_1)$, which contradicts with \eqref{537}. So we obtain \eqref{538}.

Next, we prove that $\sigma=1$. Suppose $\sigma=-1$. From Proposition \ref{P5}, we obtain that for all $t\geq T_0$,
$$\frac{1}{[\mathfrak{q}_1(t)-\mathfrak{q}_2(t)]^2}\lesssim \sum_{i=1}^2[\mathfrak{a}_i^+(t)]^2+\sum_{i=1}^2[\mathfrak{a}_i^-(t)]^2$$
We choose $T_1$ large enough such that
$$\frac{1}{[\mathfrak{q}_1(T_1)-\mathfrak{q}_2(T_1)]^4}=\sup_{t\geq T_1}\frac{1}{[\mathfrak{q}_1(T_1)-\mathfrak{q}_2(T_1)]^4},$$
Together with \eqref{528}, we have
$$\frac{1}{[\mathfrak{q}_1(T_1)-\mathfrak{q}_2(T_1)]^4}\lesssim \sum_{i=1}^2[\mathfrak{a}_i^-(t)]^2.$$
But this contradicts with \eqref{538}, if $c_0$ is small enough.

Finally, we can conclude the proof of \eqref{523} by showing that there exists a $T_0>t_0$, such that $\mathfrak{q}_1(t)-\mathfrak{q}_2(t)$ is a non-decreasing function for $t\geq T_0$. Let $\mathfrak{q}(t)=\mathfrak{q}_1(t)-\mathfrak{q}_2(t)$, and $T_1\geq T_0$ for some sufficiently large $T_0$ to be chosen later. We need to show that $\mathfrak{q}(t)\geq \mathfrak{q}(T_1)$ for all $t>T_1$. Suppose this does not hold true, and let 
$$T_2=\sup\{t:\,\mathfrak{q}(t)=\inf_{\tau\geq T_1}\mathfrak{q}(\tau)\}.$$
Then $T_2>T_1$, $\mathfrak{q}(T_2)=\inf_{\tau\leq T_2}\mathfrak{q}(\tau)$ and $\dot{\mathfrak{q}}(T_2)=0$. Let $\mathfrak{p}(t)=\mathfrak{p}_1(t)-\mathfrak{p}_2(t)$, and
$$T_3=\inf\{t\geq T_2:\,\mathfrak{q}(t)=\mathfrak{q}(T_2)+1\},$$
We know that $T_3<+\infty$, since $\mathfrak{q}(t)\rightarrow+\infty$, as $t\rightarrow+\infty$. Let $t\in[T_2,T_3]$, from Proposition \ref{P14}, we have $\|\epsilon(t)\|^2_{\h}\lesssim \frac{1}{\mathfrak{q}^4(T_2)}$. Together with \eqref{56}--\eqref{57} and the fact that $\sigma=1$, we have 
\begin{equation}
\label{541}
\dot{\mathfrak{p}}(t)\geq \frac{3\kappa_0\int Q^p}{\mathfrak{q}^3(t)}\geq \frac{3\kappa_0\int Q^p}{\mathfrak{q}^3(T_2)}\geq \frac{2\kappa_0\int Q^p}{\mathfrak{q}^3(T_2)}.
\end{equation}
Since $\dot{\mathfrak{q}}(T_2)=0$, from \eqref{55}, we have $\mathfrak{p}(T_2)\geq -\frac{c}{\mathfrak{q}^2(T_2)}$. Integrating \eqref{541}, we have for all $t\in[T_2,T_3]$,
$$\mathfrak{p}(t)\geq  -\frac{c}{\mathfrak{q}^2(T_2)}+ \frac{2(t-T_2)\kappa_0\int Q^p}{\mathfrak{q}^3(T_2)},$$
Using \eqref{55} again, we have for all%
\footnote{Here, we use the fact that $(Q,\Lambda Q)<0$ when $p>3$.}
 $t\in[T_2,T_3]$,
$$-(Q,\Lambda Q)\dot{\mathfrak{q}}(t)\leq\frac{c}{\mathfrak{q}^2(T_2)}- \frac{2(t-T_2)\kappa_0\int Q^p}{\mathfrak{q}^3(T_2)}.$$
Integrating the above inequality from $T_2$ to $T_3$, we obtain that
\begin{align*}
-(Q,\Lambda Q)[\mathfrak{q}(T_3)-\mathfrak{q}(T_2)]&\leq \int_{T_2}^{T_3}\bigg(  \frac{c}{\mathfrak{q}^2(T_2)}- \frac{2(t-T_2)\kappa_0\int Q^p}{\mathfrak{q}^3(T_2)}\bigg)\,\dd t\\
&\leq   \frac{c(T_3-T_2)}{\mathfrak{q}^2(T_2)}- \frac{(T_3-T_2)^2\kappa_0\int Q^p}{\mathfrak{q}^3(T_2)}.
\end{align*}
Let 
$$F(x)=\frac{cx}{\mathfrak{q}^2(T_2)}- \frac{x^2\kappa_0\int Q^p}{\mathfrak{q}^3(T_2)}.$$
Then, we have 
$$F(x)\leq \frac{c^2}{4\kappa_0\mathfrak{q}(T_2)\int Q^p}\lesssim c^2.$$
Hence, we have $[\mathfrak{q}(T_3)-\mathfrak{q}(T_2)]\leq \frac{1}{2}$, if we choose $c$ sufficiently small, which contradicts with the definition of $T_3$. Then, the proof of Proposition \ref{P14} is completed.
\end{proof}

Now, we can conclude the proof of Theorem \ref{MT2} by showing that
$$\lim_{t\rightarrow+\infty}\frac{x_1(t)-x_2(t)}{\sqrt{t}}=\alpha_1-\alpha_2.$$
From Lemma \ref{L12} and Proposition \ref{P14}, we have for all $c>0$, there exist $T_0\geq t_0$ and $\rho_0\gg1$, such that for all $t\geq T_0$ and $\rho\geq \rho_0$, there holds
\begin{align}
&|\dot{\mathfrak{q}}(t)-(Q,\Lambda Q)^{-1}\mathfrak{p}(t)|\leq \frac{c}{\mathfrak{q}^2(t)},\label{525}\\
&\bigg|\dot{\mathfrak{p}}(t)-\frac{4c_0\int Q^p}{\mathfrak{q}^3(t)}\bigg|\leq \frac{c}{\mathfrak{q}^4(t)},\label{526}
\end{align}
where $\mathfrak{q}(t)=\mathfrak{q}_1(t)-\mathfrak{q}_2(t)$, $\mathfrak{p}(t)=\mathfrak{p}_1(t)-\mathfrak{p}_2(t)$. We claim that there exists a universal constant $K_0>100$ (depending only on $p$ and $n$) such that for all $t\geq T_0$, we have
\begin{equation}\label{530}
\bigg|\mathfrak{p}(t)- \frac{2\alpha^2}{\mathfrak{q}(t)}(Q,\Lambda Q)\bigg|\leq \frac{K_0c}{\mathfrak{q}^2(t)},
\end{equation}
where $\alpha>0$ is given by:
$$\alpha=\bigg(-\frac{c_0\int Q^p}{(Q,\Lambda Q)}\bigg)^{\frac{1}{4}}.$$
Let 
$$\mathfrak{r}(t)=\mathfrak{p}(t)- \frac{2\alpha^2}{\mathfrak{q}(t)}(Q,\Lambda Q).$$
From \eqref{525} and \eqref{526}, we have
\begin{align*}
\dot{\mathfrak{r}}&=\dot{\mathfrak{p}}+\frac{2\alpha^2\dot{\mathfrak{q}}}{\mathfrak{q}^2}(Q,\Lambda Q)=\frac{4c_0\int Q^p}{\mathfrak{q}^3}+\frac{2\alpha^2\mathfrak{p}}{\mathfrak{q}^2}+O\bigg(\frac{c}{\mathfrak{q}^4}\bigg)\\
&=\frac{4c_0\int Q^p}{\mathfrak{q}^3}+\frac{2\alpha^2}{\mathfrak{q}^2}\bigg(\mathfrak{r}+ \frac{2\alpha^2}{\mathfrak{q}}(Q,\Lambda Q)\bigg)+O\bigg(\frac{c}{\mathfrak{q}^4}\bigg),
\end{align*}
which implies that there exists a universal constant $K_1>100$, such that
\begin{equation}
\label{531}
\bigg|\dot{\mathfrak{r}}-\frac{2\alpha^2\mathfrak{r}}{\mathfrak{q}^2}\bigg|\leq \frac{K_1c}{\mathfrak{q}^4}.
\end{equation}
Let $K_0=\frac{K_1}{\alpha^2}$. Suppose for some $T_1\geq T_0$, we have $\mathfrak{r}(T_1)> K_0c/\mathfrak{q}(T_1)^2$. Then, since $\mathfrak{r}(t)\rightarrow0$, as $t\rightarrow+\infty$, there exists a $T_2\in(T_1,+\infty)$ such that
$$T_2=\sup\{t:\,\mathfrak{r}(t)=K_0c/\mathfrak{q}^2(t)\}.$$
Then we must have $\dot{\mathfrak{r}}(T_2)\leq0$. But from \eqref{531} and the choice of $K_0$, we have $\dot{\mathfrak{r}}(T_2)>0$, which is a contradiction. The proof for $\mathfrak{r}(T_1)<0$ is similar.

Now, from \eqref{525} and \eqref{530}, we have for all $c>0$, there exist $T_0\geq t_0$ and $\rho_0\gg1$, such that for all $t\geq T_0$ and $\rho\geq \rho_0$, there holds
$$\bigg|\dot{\mathfrak{q}}-\frac{2\alpha^2}{\mathfrak{q}}\bigg|\leq \frac{c}{\mathfrak{q}^2}\implies |\partial_t(\mathfrak{q}^2)-4\alpha^2|\leq c.$$
After integration, using the fact that $\mathfrak{q}\rightarrow+\infty$, as $t\rightarrow+\infty$, we obtain that
$$2\alpha-c\leq \liminf_{t\rightarrow+\infty}\frac{\mathfrak{q}(t)}{\sqrt{t}}\leq \limsup_{t\rightarrow+\infty}\frac{\mathfrak{q}(t)}{\sqrt{t}}\leq 2\alpha+c.$$
Since, $c$ is arbitrary, together with \eqref{54}, we conclude the proof of Theorem \ref{MT2}.

\appendix
\section{Proof of Proposition \ref{P1}}\label{Ap1}

We only need to prove the asymptotic formula \eqref{21}. The proof follows from a similar argument as \cite[Appendix C.1]{FLS}. We also mention here the first order asymptotics has already been proved in  \cite[Appendix C.1]{FLS}. We start with the following lemma:

\begin{proof}[Proof of Proposition \ref{P1}]
For all $t>0$, we denote by 
$$p(t,y)=\mathcal{F}^{-1}(e^{-t|\xi|}),\quad G(y)=\mathcal{F}^{-1}\bigg(\frac{1}{1+|\xi|}\bigg),$$
where $\mathcal{F}$ is the Fourier transform on $\mathbb{R}$. As a standard result, there exists a positive constant $C_0>0$, such that 
$$p(t,y)=\frac{C_0t}{t^2+y^2}.$$ 
We also have
$$G(y)=\mathcal{F}^{-1}\bigg(\frac{1}{1+|\xi|}\bigg)=\mathcal{F}^{-1}\bigg(\int_0^\infty e^{-t}e^{-t|\xi|}\,\dd t\bigg)=\int_0^\infty e^{-t}p(t,y)\,\dd t,$$
which implies that $G$ is a smooth even function. Moreover, by Fubini's theorem, we have
$$\int_{\mathbb{R}}G(y)\,\dd y=\int_0^\infty C_0e^{-t}\bigg(\int_{\mathbb{R}}\frac{t}{t^2+y^2}\,\dd y\bigg)\,\dd t=\int_0^\infty C_0e^{-t}\bigg(\int_{\mathbb{R}}\frac{1}{1+y^2}\,\dd y\bigg)\,\dd t,$$
which shows that $G\in L^1(\mathbb{R})$. On the other hand, a direct computation shows that
\begin{align}
\label{227}
G(y)-\frac{C_0}{y^2}=\Bigg[\int_0^{\infty}te^{-t}\bigg(\frac{C_0}{t^2+y^2}-\frac{C_0}{y^2}\bigg)\,\dd t\Bigg]\lesssim\frac{1}{y^4}\int_0^\infty t^3e^{-t}\,\dd t=O\bigg(\frac{1}{y^4}\bigg),
\end{align}
as $|y|\rightarrow+\infty$.

Let $F(y)=Q^p(y)$, $\kappa_0=C_0\int F$, then $Q=F\star G$ and $F(y)\lesssim (1+|y|)^{-2p}$. We have
\begin{align*}
&y^2Q(y)-\kappa_0=\int_{\mathbb{R}}F(x)\bigg(y^2G(y-x)-C_0\bigg)\,\dd x\\
&=\int_{\{|y|<2|y-x|\}}F(x)\bigg(y^2G(y-x)-C_0\bigg)\,\dd x+\int_{\{|y|>2|y-x|\}}F(x)\bigg(y^2G(y-x)-C_0\bigg)\,\dd x\\
&:={\rm I}+{\rm II}.
\end{align*}
For $\rm I$, using \eqref{227}, we obtain that
$$y^2G(y-x)-C_0=C_0\bigg(\frac{2x}{y-x}+\frac{x^2}{(y-x)^2}\bigg)+O\bigg(\frac{y^2}{(y-x)^4}\bigg).$$
Hence, 
\begin{align}\label{228}
\bigg|{\rm I}-C_0\int_{\{|y|<2|y-x|\}}\frac{2xF(x)}{y-x}\,\dd x\bigg|\lesssim \frac{1}{y^2}\int_{\mathbb{R}}(1+|x|^2)F(x)\,\dd x=O\bigg(\frac{1}{y^2}\bigg).
\end{align}
For $y\gg1$, using the fact that $F$ is even, we have
\begin{align}\label{229}
&\bigg|\int_{\{|y|<2|y-x|\}}\frac{2xF(x)}{y-x}\,\dd x\bigg|=\bigg|\int_{x<-\frac{y}{2}}\frac{2xF(x)}{y-x}\,\dd x+\int_{-\frac{y}{2}}^{\frac{y}{2}}[2xF(x)]\bigg(\frac{1}{y-x}-\frac{1}{y}\bigg)\,\dd x\bigg|\nonumber\\
&\qquad\lesssim\frac{1}{y}\int_{x<-\frac{y}{2}}|x|^{-2p+1}\,\dd x+\int_{-\frac{y}{2}}^{\frac{y}{2}}\frac{x^2F(x)}{|y||y-x|}\,\dd x=O\bigg(\frac{1}{y^2}\bigg).
\end{align}
Combining \eqref{228} and \eqref{229}, we have $|{\rm I}|\lesssim 1/y^2$. While for $\rm II$, we have $|x|\sim |y|$ in this case. Using the fact that $p\geq 2$, we have
$$|{\rm II}|\lesssim \frac{1}{y^{2p-2}}\int_{\mathbb{R}}G(y-x)\,\dd x+\int_{\{|x|\sim |y|\}}\frac{1}{|x|^{2p}}\,\dd x=O\bigg(\frac{1}{y^2}\bigg).$$
Now we have already proved that
\begin{equation}
\label{230}
Q(y)=\frac{\kappa_0}{y^2}+O\bigg(\frac{1}{y^4}\bigg),\;\;\text{as }y\rightarrow+\infty.
\end{equation}

Finally, we let 
$$f(x)=
\begin{cases}
\frac{1}{x^2}Q(1/x),&\text{if }x\not=0,\\
\kappa_0, &\text{if }x=0,
\end{cases}.
$$
It is easy to see from \eqref{230} that
$$f(x)=\kappa_0+O(x^2),\;\;\text{as }x\rightarrow0^+.$$
This formula implies that $f\in C^1([0,1])$ is an even function. We also have $f'(0)=0$, $f'$ is absolutely continuous on $[0,1]$ and $f''\in L^\infty([0,1])$. For all $x\in(0,1]$, using \eqref{230}, we have
$$Q'(1/x)=-x^4f'(x)-2xQ(1/x)=-2\kappa_0x^3+O(|x|^5),$$
as $x\rightarrow0^+$, which is exactly \eqref{233}. On the other hand, since $f'(0)=0$, for all $x\in(0,1]$, we have
$$f(x)-\kappa_0=\int_0^x(x-s)f''(s)\,\dd s.$$
Let 
$$g(y)=\int_y^{+\infty}\bigg(\frac{1}{y}-\frac{1}{s}\bigg)\frac{y^2}{s^2}f''(1/s)\,\dd s,\quad \forall y>1.$$
It is easy to see that $g$ can be extended to an even function which belongs to $L^\infty(\mathbb{R})\cap C^1(\mathbb{R})$. We still denote it by $g$. By direct computation, we see that \eqref{21} and \eqref{234} hold.
\end{proof}

 \section*{Data Availability}
 The data that supports the findings of this study are available within the article.

 \section*{Conflict of interest}
 The authors have no conflicts to disclose.

\bibliographystyle{amsplain}
\bibliography{ref}

\end{document}